\documentclass{article}

     \PassOptionsToPackage{numbers, compress}{natbib}

\usepackage[preprint]{neurips_2021}

\usepackage[utf8]{inputenc} 
\usepackage[T1]{fontenc}    

\usepackage{url}            
\usepackage{booktabs}       
\usepackage{amsfonts}       
\usepackage{nicefrac}       
\usepackage{microtype}      
\usepackage{comment}

\usepackage{amsmath}
\usepackage{amssymb}
\usepackage{bbm}
\usepackage{amsthm}
\usepackage{multirow}
\usepackage{todonotes}
\usepackage{enumitem}

\usepackage[symbol]{footmisc}

\def\sD{\mathcal{D}}
\def\sH{\mathcal{H}}
\def\sV{\mathcal{V}}
\def\sX{\mathcal{X}}
\def\sT{\mathcal{T}}

\def\sR{\mathcal{R}}
\def\sP{\mathcal{P}}
\def\sQ{\mathcal{Q}}
\def\sF{\mathcal{F}}
\def\sB{\mathcal{B}}
\def\sW{\mathcal{W}}
\def\sL{\mathcal{L}}

\def\tsT{\widetilde{\mathcal{T}}}
\def\tsH{\widetilde{\mathcal{H}}}
\def\tsP{\widetilde{\sP}}
\def\tsX{\widetilde{\sX}}

\def\ot{\otimes}

\def\bv{\mathbf{v}}
\def\be{\mathbf{e}}

\def\tp{\widetilde{p}}
\def\tq{\widetilde{q}}
\def\tw{\widetilde{w}}

\def\tP{\widetilde{P}}

\def\th{\widetilde{h}}
\def\tX{\widetilde{X}}
\def\tW{\widetilde{W}}
\def\tf{\widetilde{f}}
\def\be{\mathbf{e}}
\def\fQ{\mathfrak{Q}}

\def\hp{\hat{p}}
\def\kmax{k_{\text{max}}}
\def\bmax{b_{\text{max}}}
\def\1{\mathbbm{1}}
\newcommand{\rn}{\mathbb{R}}
\def\nn{\mathbb{N}}

\def\E{\mathbb{E}}

\def\conv{\operatorname{Conv}}
\def\mix{\operatorname{-mix}}
\def\lip{\operatorname{Lip}}

\def\simiid{\overset{iid}{\sim}}
\def\spn{\operatorname{span}}
\def\lip{\operatorname{Lip}}
\def\proj{\operatorname{Proj}}
\def\cip{\,{\buildrel p \over \rightarrow}\,}
\def\area{\operatorname{area}}

\DeclareMathOperator{\argmin}{argmin}

\def\l{\left}
\def\r{\right}

\newtheorem{lem}{Lemma}[section]
\newtheorem{thm}{Theorem}[section]
\newtheorem{cor}{Corollary}[section]
\newtheorem{prop}{Proposition}[section]

    \def\sD{\mathcal{D}}
    \def\sH{\mathcal{H}}
    \def\sV{\mathcal{V}}
    \def\sX{\mathcal{X}}
    \def\sT{\mathcal{T}}
    
    \def\sR{\mathcal{R}}
    \def\sP{\mathcal{P}}
    \def\sQ{\mathcal{Q}}
    \def\sF{\mathcal{F}}
    \def\sB{\mathcal{B}}
    \def\sW{\mathcal{W}}
    \def\sL{\mathcal{L}}
    \def\sG{\mathcal{G}}

    \def\tsT{\widetilde{\mathcal{T}}}
    \def\tsH{\widetilde{\mathcal{H}}}
    \def\tsP{\widetilde{\sP}}
    \def\tsX{\widetilde{\sX}}
    
    \def\tsG{\widetilde{\sG}}

    \def\ot{\otimes}

    \def\bv{\mathbf{v}}
    \def\be{\mathbf{e}}

    \def\tp{\widetilde{p}}
    \def\tq{\widetilde{q}}
    \def\tw{\widetilde{w}}
    
    \def\tP{\widetilde{P}}
    
    \def\th{\widetilde{h}}
    \def\tX{\widetilde{X}}
    \def\tW{\widetilde{W}}
    \def\tf{\widetilde{f}}
    \def\fL{\mathfrak{L}}

    \def\E{\mathbb{E}}

    \def\conv{\operatorname{Conv}}
    \def\mix{\operatorname{-mix}}
    \def\lip{\operatorname{Lip}}

    \def\simiid{\overset{iid}{\sim}}
    \def\spn{\operatorname{span}}
    \def\lip{\operatorname{Lip}}
    \def\proj{\operatorname{Proj}}
    \def\cip{\,{\buildrel p \over \rightarrow}\,}
    \def\area{\operatorname{area}}

    \DeclareMathOperator{\emi}{\theta}
    \DeclareMathOperator{\why}{\zeta}

\newtheorem*{remark*}{Remark}

\title{Beyond Smoothness: Incorporating Low-Rank Analysis into Nonparametric Density Estimation}

\author{
  Robert A. Vandermeulen\\
  Machine Learning Group\\
 Technische Universi\"at Berlin\\
 Berlin, Germany\\
 \texttt{vandermeulen@tu-berlin.de} \\
  \And
  Antoine Ledent \\
  Machine Learning Group \\
  Technische Universit\"at Kaiserslautern \\
  Kaiserslautern, Germany\\
  \texttt{ledent@cs.uni-kl.de} \\
  }
\begin{document}

\maketitle

\begin{abstract}
  The construction and theoretical analysis of the most popular universally consistent nonparametric density estimators hinge on one functional property: smoothness. In this paper we investigate the theoretical implications of incorporating a multi-view latent variable model, a type of low-rank model, into nonparametric density estimation. To do this we perform extensive analysis on histogram-style estimators that integrate a multi-view model. Our analysis culminates in showing that there exists a universally consistent histogram-style estimator that converges to any multi-view model with a finite number of Lipschitz continuous components at a rate of $\widetilde{O}(1/\sqrt[3]{n})$ in $L^1$ error. In contrast, the standard histogram estimator can converge at a rate slower than $1/\sqrt[d]{n}$ on the same class of densities. We also introduce a new nonparametric latent variable model based on the Tucker decomposition. A rudimentary implementation of our estimators experimentally demonstrates a considerable performance improvement over the standard histogram estimator. We also provide a thorough analysis of the sample complexity of our Tucker decomposition-based model and a variety of other results. Thus, our paper provides solid theoretical foundations for extending low-rank techniques to the nonparametric setting.
\end{abstract}

\section{Introduction}\label{sec:introduction}
Nonparametric density estimators are density estimators capable of estimating a density $p$ while making few to no assumptions on $p$. Two commonly used nonparametric density estimators include the histogram estimator and the kernel density estimator (KDE). A common characteristic of these two estimators is that they make estimation tractable via a hyperparameter that relates to smoothness, namely bin width and bandwidth. Large bin width or bandwidth allows for good control of estimation error\footnote{For an estimator $V$ restricted space a of densities $\sP$, the \emph{estimation error} refers to the difference between $\left\|V-p\right\|$ and $\min_{q\in \sP} \left\|p - q\right\|$, where $p$ is the target density. This is similar to estimator variance.} at the cost of increased approximation error via large bin volume for the histogram and smoothing for KDEs. Selection of this parameter is crucial for estimator performance. In fact a recent survey on bandwidth selection for KDEs found at least 30 methods for setting this value along with a few surveys dedicated to the topic \cite{heidenreich13}.

While nonparametric density estimation has been shown to be effective for many tasks, it has been observed empirically that estimator performance typically declines as data dimensionality increases, a manifestation of the curse of dimensionality. For the histogram and KDE this phenomenon has concrete mathematical analogs. For example, these estimators are only universally consistent\footnote{A density estimator is \emph{universally consistent} if it asymptotically recovers \emph{any} density.} if $n\to \infty$ and $h\to 0$, with $nh^d \to \infty$, where $n$ is the number of samples, $d$ is the data dimension, and $h$ is the bin width for the histogram and bandwidth parameter for the KDE \cite{gyorfi85}. One may then wonder whether there exists some other way to constrain model capacity so as to alleviate this exponential dependence on dimension. 

In this paper we theoretically analyze the advantages of including a constraint akin to matrix/tensor rank in a nonparametric density estimator. We do so by analyzing histogram-style estimators (estimators that output a density that are piecewise constant on bins defined by a grid) that are enforced to have a low-rank PARAFAC or Tucker decomposition. Enforcing this low-rank constraint in the histogram estimator allows for much faster rates on $h\to 0$, while still controlling estimation error. This can remove the exponential penalty on rate of convergence that dimensionality produces with the standard histogram estimator.

The bulk of this work focuses on analysis of hypothetical estimators that, while offering large statistical advantages over the standard histogram estimator, are not computationally tractable. At the end of this work we include experiments demonstrating that low-rank histograms, computed using an off-the-shelf nonnegative tensor factorization library \cite{tensorly}, consistently outperform the standard histogram estimator.
\subsection{Density Models}\label{ssec:models}
Here we introduce the low-rank models that will be used to constrain our estimators. The first is a \emph{multi-view model}. A multi-view model $p$ is a finite mixture $p  = \sum_{i=1}^k w_i f_i$ whose components are separable densities $f_i = p_{i,1}\otimes p_{i,2} \otimes \cdots \otimes p_{i,d}$ \footnote{For two functions $f,g$ let $f\otimes g: (x,y)\mapsto f(x)g(y)$. This is analogous to the tensor product of $L^2$ functions.} \cite{kasahara14, allman09,song13,song14}. A multi-view model has the following form
\begin{equation} \label{eqn:mv-intro}
  p \left(x_1,x_2,\ldots,x_d\right) = \sum_{i=1}^k w_i p_{i,1}\left(x_1\right)p_{i,2}\left(x_2\right) \cdots p_{i,d}\left(x_d\right).
\end{equation}
A multi-view model with one component is typically called a \emph{naive Bayes model}. For the estimators we propose, the component marginals $p_{i,j}$ will have the form of one-dimensional histograms and the number of components $k$ will be limited to restrict estimator capacity. When $k=1$ the model is equivalent to a naive Bayes model, $p \left(x_1,x_2,\ldots,x_d\right) = p_{1}\left(x_1\right)p_{2}\left(x_2\right) \cdots p_{d}\left(x_d\right)$. When the component marginals $p_{i,j}$ are histograms increasing $k$ expands the set of potential estimates from naive Bayes models when $k=1$ to all possible histogram estimates. 

The models in this paper are motivated by nonnegative tensor factorizations so we term them generally as \emph{nonparametric nonnegative tensor factorization} (NNTF) \emph{models}. The previous model was related to nonnegative PARAFAC \cite{shashua05}.
Our second model is based on the nonnegative Tucker decomposition \cite{kim07}. A density in this model utilizes $d$ collections, $\sF_1,\ldots, \sF_d$, of $k$ one-dimensional densities, $\sF_i = \left\{p_{i,1},\ldots, p_{i,k} \right\}$, and some probability measure that randomly selects one density from each $\sF_i$. This measure a can be represented by a tensor $W \in\rn^{k^{\times d}}$ where the probability of selecting $\left(p_{1,i_1},\ldots, p_{d,i_d}\right)$ from $ \sF_1\times\cdots\times \sF_d$ is $W_{i_1,\ldots,i_d}$. To sample from this model we first randomly select the marginal distributions $p_{1,i_1},\ldots,p_{d,i_d}$ according to $W$, and an observation is sampled according to the $d$-dimensional distribution $p_{1,i_1}\otimes p_{2,i_2}\otimes\cdots \otimes p_{2,i_2}$. The density of this model is
\begin{align}
  p\left(x_1,x_2,\ldots,x_d\right) = \sum_{i_1 =1}^k \cdots \sum_{i_d =1}^k W_{i_1,\ldots,i_d}p_{1,i_1}(x_1)p_{2,i_2}(x_2)\cdots p_{d,i_d}(x_d).\label{eqn:tucker}
\end{align}
We are unaware of previous works investigating this model for general probability distributions so we will simply term it the \emph{Tucker model}. Again we will investigate estimators where the component marginals are one-dimensional histograms and small $k$ corresponds to reduced model capacity. We remark that Tucker decompositions typically have a rank \emph{vector} $[k_1,\ldots,k_d]$ which for our model would mean that each $\sF_i$ contains $k_i$ vectors and $W$ would lie in $\rn^{k_1\times \cdots \times k_d}$. This sort of rank restriction could be used in our methods however, for simplicity, we just set $k_1=k_2=\cdots =k_d \triangleq k$.

\subsection{Overview of Results} \label{ssec:overview}
 The principal contribution of this paper is to analyze the advantage of incorporating rank restriction into density estimation. In Section \ref{ssec:notation} we precisely introduce the multi-view histogram and a histogram based on the Tucker model. With these models we investigate how quickly we can let bin width $h$ go to zero and rank $k$ grow, in relation to the amount of data $n$, while still being able to control estimation error and select an estimator that is nearly optimal for the allowable set.  For the multi-view histogram we show that one can control estimation error so long as $k/h$ is asymptotically dominated by $n$ and for the Tucker histogram we need $k/h + k^d$ to be asymptotically dominated by $n$ (we are omitting logarithmic factors here for convenience). This stands in stark contrast to standard space of histogram estimators which requires $1/h^{d}$ to be asymptotically dominated by $n$. We furthermore show that these estimators are universally consistent and that these rates cannot be significantly improved.
 
 For a second style of analysis we provide finite-sample bounds on the convergence of NNTF histogram estimators. We then construct a class of universally consistent density estimators that converge at rate $\widetilde{O}(1/\sqrt[3]{n})$\footnote{$f_n \in \widetilde{O}(g_n )\iff \exists k \text{ s.t. } f_n \in O\left(g_n\log^k g_n\right) $} on all densities that are a multi-view model with Lipschitz continuous component marginals. Note that the NNTF histogram estimators we construct do not require the target density to be an NNTF model to function well. Our estimators will select a good estimator so long as there exists any NNTF histogram that approximates the target density well; so our hypothetical estimators ``fail elegantly'' in some sense. We further show that the standard histogram can converge at a rate $\omega(1/\sqrt[d]{n})$\footnote{$f_n \in \omega\left(g_n\right) \iff \left|f_n/ g_n\right| \to \infty$ in probability. In $d$-dimensional space the standard histogram can converge slower than $1/\sqrt[d]{n}$.} on the same class of densities.
In Section \ref{sec:experiments} we experimentally investigate the efficacy of using NNTF histograms on real-world data. In lieu of the computationally intractable methods we investigated in our theoretical analyses, we use an existing low-rank nonnegative Tucker factorization algorithm to fit an NNTF histogram to data. Surprisingly even this method outperforms the standard histogram estimator with very high statistical significance. Lastly we mention that this paper is an extension of \cite{vandermeulen20} and contains a fair amount of overlap with that text.

\subsection{Previous Work}

Nonparametric density estimation has been extensively studied with the histogram estimator and KDE being some of the most well-known methods. There do exist, however, alternative methods for density estimation, e.g. the \emph{forest density estimator} \cite{liu11} and \emph{$k$-nearest neighbor density estimator} \cite{mack79}. The $L^1,L^2$, and $L^\infty$ convergence of the histogram and KDE has been studied extensively \cite{gyorfi85,devroye01,tsybakov08,jiang17}. The KDE is generally regarded as the superior density estimator, with some mathematical justification \cite{gyorfi85,silverman86}. Numerous modifications and extensions of the KDE have been proposed including utilizing variable bandwidth \cite{terrell92}, robust KDEs \cite{kim12,vandermeulen13,vandermeulen14}, methods for enforcing support boundary constraints \cite{schuster85}, and a supervised variant \cite{vandermeulen20proposal}. Finally we mention \cite{kim18} that demonstrated that uniform convergence of a KDE to its population estimate suffered when the intrinsic dimension of the data was lower than the ambient dimension, a phenomenon seemingly at odds with the curse of dimensionality. 

For our review of NNTF models we also include a general review of tensor/matrix factorizations since both can be viewed being low-rank models. In particular, for the multi-view model we have the following analogy
\begin{equation}\label{eqn:lrtens}
	\sum_{i=1}^k w_i p_{i,1}\left(x_1\right)p_{i,2}\left(x_2\right) \cdots p_{i,d}\left(x_d\right) \sim \sum_{i=1}^k \lambda_i \bv_{i,1}\otimes\bv_{i,2}\otimes  \cdots \otimes \bv_{i,d}.
\end{equation}

A great deal of work has gone into leveraging low-rank assumptions to improve matrix estimation, particularly in the field of \emph{compressed sensing} \cite{donoho06,recht10}. The most basic version of compressed sensing is concerned with estimating a ``tall'' vector $x$ from a ``short'' vector $y \triangleq A x$ where $A$ is a known ``short and fat'' matrix. One can recover $y$ if it is sparse and $A$ satisfies a property known as the restricted isometry property (RIP) \cite{vershynin18,candes08}. These methods can be extended to the estimation of matrices when $x$ is a low-rank matrix \cite{recht10,negahban11,negahban12}. In this extension $A$ is an order-3 tensor which acts as a linear operator on $x$ and satisfies an adjusted form of RIP. RIP commonly arises from matrices and tensors whose entries are random. Because of this compressed sensing techniques are useful in settings where one wants to estimate $x$ from random linear transforms of $x$. For example, in matrix completion observing the $(i,j)$-th entry of $x$ can be represented as an inner product of $x$ with an indicator matrix $\mathbf{e}_{i,j}$, i.e. $\left<x,\mathbf{e}_{i,j}\right>_F$. Thus the random observed indices $\left(i_1,j_1\right),\left(i_2,j_2\right),\ldots$ can be represented as random matrices $\be_{i_1,j_1},\be_{i_2,j_2},\ldots$ which are then stacked into $A$. Now $A$ is a random linear operator that is applied to $x$ to represent the observation of random entries of $x$ and the methods of compressed sensing can be used to recover $x$. Compressed sensing has also proven useful for multivariate regression and autoregressive models \cite{negahban11,negahban12}. Such techniques don't appear to be extensible to histogram estimation due to the lack of a linear sampling scheme.

General matrix/tensor factorization, including nonnegative matrix/tensor factorizations, has been extensively studied despite being inherently difficult due to non-convexity. The works \cite{donoho03,arora12} present potential theoretical grounds for avoiding the computational difficulties of nonnegative matrix factorization. One notable approach to tensor factorization is to assume, in the tensor representation in (\ref{eqn:lrtens}), that $d \ge 3$ and the collections of vectors $\bv_{1,j},\ldots, \bv_{k,j}$ are linearly independent for all $j$. Under this assumption we are guaranteed that the factorization (\ref{eqn:lrtens}) is unique \cite{allman09}. In \cite{anandkumar14} the authors present a method for recovering this factorization efficiently and demonstrate its utility for a variety of tasks. This work was extended in \cite{song14} to recover a multi-view KDE satisfying an analogous linear independence assumption and theoretically analyze the estimator's convergence to the true low-rank components. In \cite{song14} the authors investigate the sample complexity of their estimator but do not demonstrate that their technique has potential for improving rates for nonparametric density estimation in general. In \cite{song13} it was observed that using low-rank embeddings can improve density estimation. A multi-view histogram was investigated in \cite{kargas19} where the authors present an identifiability result and algorithm for recovering latent factors of the distribution. The related works \cite{amiridi21a,amiridi21b} consider a low-rank characteristic function as an approach to improving nonparametric density estimation. Though earlier works have observed that a low-rank approach improves nonparametric density estimation \cite{song13,song14,kargas19,amiridi21a,amiridi21b}, we are the first to demonstrate this through theoretical analysis of sample complexity. Finally we note that the Tucker decomposition has been utilized in Bayesian statistics \cite{schein16}.  We are unaware of any literature on factoring functions in $\rn^d \to \rn$ in a Tucker-inspired as we do in \eqref{eqn:tucker}.

\section{Theoretical Results}\label{sec:theory}
In this section we mathematically demonstrate that histogram estimators can achieve greater performance by restricting to NNTF models and using a proper procedure to select a representative from these using data. To simplify analysis we will only consider densities on the unit cube $\left[0,1\right)^d$ and analyze the number of bins per dimension $b$ which is the inverse of the bin width, i.e. $b = 1/h$. To state our results precisely we must introduce a fair amount of notation.
\subsection{Notation}\label{ssec:notation}
We will denote the $L^1$ and $L^2$ Lebesgue space norms via a $1$ or $2$ subscript. Let $\sD_d$ be the set of all densities on $\left[0,1\right)^d$. By \emph{density} we mean probability measures that are absolutely continuous with respect to the $d$-dimensional Lebesgue measure on $[0,1)^d$. We define a \emph{probability vector} or \emph{probability tensor} to simply mean a vector or tensor whose entries are nonnegative and sum to one. Let $\Delta_b$ denote the set of probability vectors in $\rn^b$ and $\sT_{d,b}$ the set of probability tensors in $\rn^{ b^{\times d}}$ \footnote{$\rn^{ b^{\times d}}$ is the set of $ \underbrace{b \times \cdots\times b}_{d\text{ times}}$ tensors. For example $\rn^{b^{\times 2}}$ is the set of $b\times b$ matrices.}. 
The product symbol $\prod$ will always mean the standard outer product, e.g. set product\footnote{For sets $S_1,\ldots,S_d$ we have $\prod_{i=1}^d S_i = S_1\times \cdots \times S_d = \left\{\left(s_1,\ldots,s_d\right):s_i \in S_i \forall i \right\}$.} or tensor product, when the multiplicands are not real numbers\footnote{For functions $f_1,\ldots,f_d$ then $\prod_{i=1}^d f_i = f_1\otimes \cdots \otimes f_d: \left(x_1,\ldots,x_d\right) \mapsto \prod_{i=1}^d f_i\left(x_i\right)$.}. The natural numbers $\nn$ will always denote \emph{positive} integers. For any natural number $b$ let $[b] =\left\{1,\ldots,b\right\}$. We will let $\1$ be the indicator function and $\conv$ be the convex hull. Later we will use projection operator where $\proj_S x \triangleq \arg \min_{s\in S} \left\|x-s\right\|_2$; for every instance in this work this projection is unique.

We will now construct the space of histograms on $\left[0,1\right)^d$. We begin with one-dimensional histograms, which will serve as the $p_{i,j}$ terms in \eqref{eqn:mv-intro} or \eqref{eqn:tucker}. We define $h_{1,b,i}$ with $i \in \left[b\right]$ to be the one-dimensional histogram where all weight is allocated to the $i$th bin. Formally we define this as
\begin{align*}
  h_{1,b,i}\left(x\right) \triangleq b \1\left(\frac{i-1}{b}\le x <\frac{i}{b} \right).
\end{align*}
Note that this is a valid density due to the leading $b$ coefficient.
We use these to construct higher-dimensional histograms. For a multi-index $A\ \in \left[b\right]^d$, let 
\begin{align*}
  h_{d,b,A} \triangleq \prod_{i=1}^d h_{1,b,A_i},
\end{align*}
i.e. the $d$-dimensional histogram with $b$ bins per dimension whose entire density is allocated to the bin indexed by $A$. Finally we define $\Lambda_{d,b,A}$ to be the support of $h_{d,b,A}$, i.e. the ``bins'' of a histogram estimator,
\begin{align*} 
  \Lambda_{d,b,A} \triangleq \prod_{i=1}^d \left[\frac{A_i-1}{b},\frac{A_i}{b}\right). 
\end{align*}

For a dataset $\sX = \left(X_1,\ldots, X_n\right)$ in $[0,1)^d$, the standard histogram estimator is 
\begin{align*}
  H_{d,b}\left(\sX\right) \triangleq \frac{1}{n}\sum_{i=1}^n\sum_{A \in [b]^d} h_{d,b,A}\1\left(X_i\in \Lambda_{d,b,A} \right).
\end{align*}

Let $\sH_{d,b} \triangleq \conv\left(\left\{h_{d,b,A} \middle| A \in \left[ b\right]^d \right\}\right)$, the set of all $d$-dimensional histograms with $b$ bins per dimension. Let $\sH_{d,b}^k$ be the set of histograms with at most $k$ separable components, i.e.
\begin{equation}\label{eqn:lrhist}
  \sH_{d,b}^k \triangleq \left\{\sum_{i=1}^k w_i\prod_{j=1}^d p_{i,j} \middle| w\in \Delta_k, p_{i,j}\in \sH_{1,b} \right\}.
\end{equation}
We will refer to elements in this space as \emph{multi-view histograms}. Elements in this space have the same form as \eqref{eqn:mv-intro} in Section \ref{ssec:models}. Similarly we define the space of \emph{Tucker histograms} to be
\begin{equation*}
  \tsH_{d,b}^k = \left\{\sum_{S \in \left[k\right]^{ d}} W_S \prod_{i=1}^d p_{i,S_i}\middle| W \in \sT_{d,k}, p_{i,j} \in \sH_{1,b} \right\}.
\end{equation*}
These have the same form as \eqref{eqn:tucker} in Section \ref{ssec:models}.

We emphasize that the collections of densities $\sH_{d,b}^k$ and $\tsH_{d,b}^k$ are primary objects of interest in this paper. The results we present are concerned with finding good density estimators restricted to these sets as $k$ and $b$ vary.

\subsection{Estimator Theoretical Results}
We present two approaches to the analysis of NNTF histogram estimators. All proofs and additional results are contained in the appendix.

First we provide an asymptotic analysis of NNTF histogram estimators in terms of estimation error control: how fast can we let $b$ and $k$ grow, with respect to $n$, while still controlling estimation error over all densities? An advantage of this analysis is that we can demonstrate that these rates are approximately optimal (up to logarithmic terms). 

For our second approach we present finite-sample bound analysis. In this analysis we first present a distribution-dependent bound that, for an estimator restricted to $\sH_{d,b}^k$, depends on $n,b,k$ and $\min_{q \in \sH_{d,b}^k} \left\|p-q\right\|_1$ where $p$ is the data generating ``target'' distribution. We follow this up with distribution-free bounds.

Distribution-free bounds for nonparametric density estimation require that the target distribution belong to a well-behaved class of distributions (such as Sobolev or H\"older classes) to enable bounding of the approximation error \cite{tsybakov08}. Our distribution-free finite-sample analysis assumes that the target density is a multi-view model whose component marginals are Lipschitz continuous. We construct an estimator that converges at a rate of approximately $1/\sqrt[3]{n}$ on this class of densities, independent of $d$. For comparison we show that the standard histogram estimator can converge at a rate worse than $1/\sqrt[d]{n}$ on this same class of densities. We mention again that this estimator shouldn't fail catastrophically when these distributional assumptions aren't exactly met so long as its approximation error, $\min_{p \in \sH_{d,b}^k} \left\|p-q \right\|_1$, isn't large. For brevity, and because the results are virtually direct analogues of their multi-view histogram counterparts, we reserve all finite-sample results for Tucker histogram for the appendix.

 \paragraph{Main Technical Tools} Our results rely on finding $L^1$ $\varepsilon$-coverings of the spaces of NNTF histograms and using Theorem 3.7 from \cite{ashtiani18} to select a good representative from that collection. The aforementioned theorem is a slight extension of Theorem 6.3 in \cite{devroye01} and is essentially proven in \cite{yatracos85}. As was mentioned in \cite{ashtiani18}, the application of these results typically does not yield a computationally practical algorithm. Likewise our results are simply meant to highlight the potential of NNTF models and are not practically implementable as is.

\subsubsection{Asymptotic Error Control}The following theorem states that one can control the estimation error of multi-view histograms with $k$ components and $b$ bins per dimension so long as $n \sim bk$ (omitting logarithmic factors). Recall that the standard histogram estimator requires $n\sim b^d$, so we have removed the exponential dependence of bin rate on dimensionality. Here and elsewhere the $\sim$ symbol is not a precise mathematical statement but rather signifies that the two values should be of the same order in a general sense. In the following $b$ and $k$ are functions of $n$ so the space of histograms changes as one acquires more data.
\begin{thm}\label{thm:genrate}
  For any pairs of sequences $b \to \infty$ and $k \to \infty$ satisfying 
  $$
  n/(bk\log(b) + k\log(k)) \to \infty,
  $$
  there exists an estimator $V_n\in \sH_{d,b}^k$ such that, for all $\varepsilon >0$
  \begin{align*}
    &\sup_{p \in \sD_d }P\left(\left\|V_n - p\right\|_1 >
    3 \min_{q\in \sH_{d,b}^k}\left\|p -q \right\|_1 + \varepsilon  \right) \to 0,
  \end{align*}
  where $V_n$ is a function of  $X_1,\ldots, X_n\simiid p$.
\end{thm}
The sample complexity for the multi-view histogram is perhaps more accurately approximated as being on the order of $dbk$ however the $d$ disappears in the asymptotic analysis.
The following theorem states that one can control the error of Tucker histogram estimates so long as $n\sim bk+k^d$ (omitting logarithmic factors).

\begin{thm}\label{thm:tuckrate}
  For any pairs of sequences $b \to \infty$ and $k \to \infty$ satisfying
  $$
  n/\left(bk\log(b) + k^d\log\left(k^d\right)\right) \to \infty,
  $$
  there exists an estimator $V_n\in \tsH_{d,b}^k$ such that, for all $\varepsilon >0$
  \begin{align*}
    &\sup_{p \in \sD_d }P\left(\left\|V_n - p\right\|_1 >
    3 \min_{q\in \tsH_{d,b}^k}\left\|p -q \right\|_1 + \varepsilon  \right) \to 0,
  \end{align*}
  where $V_n$ is a function of $X_1,\ldots, X_n\simiid p$.
\end{thm}

Allowing $b$ to grow as aggressively as possible we achieve consistent estimation so long as $n \sim b\log b$ and $k$ grows sufficiently slowly, regardless of dimensionality.
\begin{cor}\label{cor:fastbin}
  \sloppy For all $d,b,k$ fix  $\sR_{d,b}^k$ to be either $\sH_{d,b}^k$ or $\tsH_{d,b}^k$\footnote{$\sR$ is fixed to $\sH$ or $\tsH$ and doesn't change as $n,b,k$ vary.}. For any sequence $b \to \infty$ with $n/\left(b\log b\right) \to \infty$, there exists a sequence $k \to \infty$ and estimator $V_n\in \sR_{d,b}^k$ such that, for all $\varepsilon >0$
  \begin{align*}
    \sup_{p \in \sD_d }P\left(\left\|V_n - p\right\|_1 
   > 3 \min_{q\in \sR_{d,b}^k}\left\|p -q \right\|_1 + \varepsilon  \right) \to 0,
  \end{align*}
  where $V_n$ is a function of  $X_1,\ldots, X_n\simiid p$.
\end{cor}

The following result shows that the approximation error of the estimators in Theorem \ref{thm:genrate}, Theorem $\ref{thm:tuckrate}$, and Corollary \ref{cor:fastbin} go to zero for all densities. Thus these estimators are universally consistent even when the NNTF model assumption is not satisfied.
\begin{lem}\label{lem:lrbias}
  Let $p \in \sD_d$. If $k\to \infty$ and $b \to \infty$ then $\min_{q\in \sH_{d,b}^k} \left\|p - q\right\|_1 \to 0$.
\end{lem}
A straightforward consequence of this is that the Tucker histogram approximation error also goes to zero.
\begin{lem}\label{lem:tuckbias}
  Let $p \in \sD_d$. If $k\to \infty$ and $b \to \infty$ then $\min_{q\in \tsH_{d,b}^k} \left\|p - q\right\|_1 \to 0$.
\end{lem}

The next theorem shows that the rate on $bk$ in Theorem \ref{thm:genrate} cannot be made significantly faster.
\begin{thm}\label{thm:lower}
	Let $d \ge 2$, $b\to \infty$, and $k\to \infty$ with $b \ge k$ and $n/\left(bk\right) \to 0$. There exists no estimator $V_n\in \sH_{d,b}^k$ such that, for all $\varepsilon >0$, the following limit holds
	\begin{align*}
		\sup_{p \in \sD_d }P\left(\left\|V_n - p\right\|_1 > 3 \min_{q\in \sH_{d,b}^k}\left\|p -q \right\|_1 + \varepsilon  \right) \to 0,
	\end{align*}
  where $V_n$ is a function of  $X_1,\ldots, X_n\simiid p$.
\end{thm}

Likewise the rate on $bk+k^d$ can also not be significantly improved in Theorem \ref{thm:tuckrate}.

\begin{thm}\label{thm:tucklower}
	Let $d \ge 2$, $b\to \infty$, and $k\to \infty$ with $b \ge k$ and $n/\left(bk+k^d\right) \to 0$. There exists no estimator $V_n\in \tsH_{d,b}^k$ such that, for all $\varepsilon >0$, the following limit holds
	\begin{align*}
		\sup_{p \in \sD_d }P\left(\left\|V_n - p\right\|_1 > 3 \min_{q\in \tsH_{d,b}^k}\left\|p -q \right\|_1 + \varepsilon  \right) \to 0,
	\end{align*}
  where $V_n$ is a function of  $X_1,\ldots, X_n\simiid p$.
\end{thm}
\subsubsection{Finite-Sample and Distribution-Independent Bounds}
In this section we investigate the convergence of multi-view histogram estimators to densities that satisfy the multi-view assumption. This will be done via a standard bias/variance decomposition style of argument. We begin with the following distribution-dependent finite-sample bound.

\begin{prop}
  \label{prop:finiteant}
    Let $d,b,k,n \in \nn$ and $0<\delta\le 1$. There exists an estimator $V_n \in \sH_{d,b}^k$ such that 
        \begin{equation*}
		\sup_{p \in \sD_d}P\left(\|p-V_n\|_1 > 3\min_{q\in\sH_{d,b}^k}\|p-q\|_1+ 7\sqrt{\frac{2bdk\log(4bdkn)}{n}}+7\sqrt{\frac{\log(\frac{3}{\delta})}{2n}}\right)<\delta,
        \end{equation*}
    where $V_n$ is a function of $X_1,\ldots,X_n\simiid p$.
\end{prop}
To analyze the approximation error, $\|p-q\|_1$, we first consider the case where $p$ has a single component with $L$-Lipschitz marginals, so $p= \prod_{i=1}^d p_i$. Using the indicator function over the unit cube with H\"older's Inequality we have that $\|p-q\|_1 = \|\left(p-q\right)\1\|_1\le\|p - q\|_2  \left\|\1\right\|_2 =\left\|p-q\right\|_2$. It is possible to show that the $L^2$ projection of $p$ onto $\sH_{d,b}^1$ is achieved by simply projecting each marginal to its best approximating one-dimensional marginal i.e. (see the appendix)
\begin{equation*}
    \arg \min_{q \in \sH_{d,b}^1} \left\|\prod_{i=1}^d p_i - q \right\|_2= \prod_{i=1}^d \proj_{\sH_{1,b}}p_i.
\end{equation*}

Let $\lip_L$ be the set of $L$-Lipschitz probability density functions $[0,1]$ and let $m_L \triangleq \sup_{f \in \lip_L} \left\|f\right\|_2$. The following theorem characterizes the approximation error for separable densities with Lipschitz continuous marginals.

\begin{thm}\label{thm:l1bias}
    Let $b^2 \ge L^2/12$ and $f_1,\ldots,f_d$ be elements of $\lip_L$ then
  \begin{equation*}
        \left \| \prod_{i=1}^d f_i - \proj_{\sH_{d,b}^1}\prod_{i=1}^d f_i \right\|_1 \le \sqrt{ m_L^{2d}- \left(m_L^2 - \frac{L^2}{12b^2}  \right) ^d}.
    \end{equation*}
\end{thm}
We show in the appendix that this decays at a rate of $O\left(b^{-1}\right)$ and that this rate is tight. To characterize the behavior of $m_L$ we have the following.
\begin{prop}\label{prop:l2bias}
	Let $m_L=\sup_{f \in \lip_L} \left\|f\right\|_2$, then
    \begin{equation*}
        m_L^2 = 
        \begin{cases}
                L^2/12 +1 & 0\le L \le 2\\
                \sqrt{\frac{8L}{9}} & L\ge2
        \end{cases}.
    \end{equation*}
\end{prop}

The following theorem is a finite-sample bound for a well-constructed multi-view histogram estimator.

\begin{thm} \label{thm:class}
  Fix $n,d,k \in \nn$, $L\ge 2$, and $1\ge\delta>0$. There exists $b$ and an estimator $V_n\in \sH_{d,b}^k$ such that 
  \begin{equation*}
	  \sup_{p\in Q}P\left(\|V_n-p\|_1 > \frac{21dk^{1/3}L^{\frac{d+3}{12}}}{n^{\frac{1}{3}}}\sqrt{\log(3Ldkn)}+7\sqrt{\frac{\log(\frac{3}{\delta})}{2n}}\right) < \delta,
  \end{equation*}
	where $Q$ is the set of densities of the form $\sum_{i=1}^k w_i \prod_{j=1}^d p_{i,j}$ with $p_{i,j}\in \lip_L$, $w \in \Delta_k$, and $V_n$ is a function of $X_1,\ldots,X_n\simiid p$.
\end{thm}

Here we analyze the asymptotic rate at which the estimator in Theorem \ref{thm:class} converges to a large class of multi-view models. To this end let 
$$\fQ_d= \left\{\sum_{i=1}^{k'} w_i p_{i,1}\otimes \cdots\otimes p_{i,d}\mid k'\in \nn, w \in \Delta_{k'}, L' \ge0, p_{i,j}\in \lip_{L'} \right\},$$

i.e. $\fQ_d$ is the space of all multi-view models whose component marginals are all Lipschitz continuous. Consider letting $L\to \infty$, $k \to \infty$, and $\delta \to 0$ in Theorem \ref{thm:class} arbitrarily slowly as $n\to \infty$. For some element $p$ in $\fQ_d$ its respective maximum Lipschitz constant and component number, $L'$ and $k'$, are fixed, so for sufficiently large $n$ we have $L>L'$ and $k>k'$ and the bound from Theorem \ref{thm:class} applies. From this it follows that $\left\|V_n - p\right\|_1 \in \widetilde{O}\left(1/\sqrt[3]{n} \right)$. So we can construct an estimator that converges at rate $\widetilde{O}\left(1/\sqrt[3]{n} \right)$ for \emph{any} multi-view model in $\fQ_d$, independent of dimension! This rate appears to be approximately optimal since ``for smooth densities, the average $L^1$ error for the histogram estimate must vary at least as $n^{-1/3}$'' (\cite{gyorfi85}, p. 99). For comparison the histogram estimator's convergence rate is hindered exponentially in dimension. 

\begin{prop} \label{prop:multilower}
    Let $V_n$ be the standard histogram estimator with $n/b^d \to \infty$. There exists $p\in \fQ_d$ such that $\left\|V_n - p\right\|_1 \in \omega(1/\sqrt[d]{n})$.
\end{prop}
\subsection{Discussion} \label{ssec:discussion}
While Theorem \ref{thm:class} gives an estimator with good convergence, the class of densities $Q$ is somewhat restrictive and likely not realistic for many situations where one would want to apply a nonparametric density estimator. Proposition \ref{prop:finiteant}, on the other hand, is not so restrictive since it depends on $\min_{q\in \sH_{d,b}^k}\left\|p -q \right\|_1$, and more clearly conveys the message of this paper. Typical works on multi-view nonparametric density estimation assume that the target density $p$ is a multi-view model and are interested in recovering the model components $p_{i,j}$ and $w$ from \eqref{eqn:mv-intro}. Similarly to how the standard histogram estimator doesn't assume $p \in \sH_{d,b}$ our work isn't meant to assume that $p$ is a multi-view model, but is instead meant to explore the benefits of including a hyperparameter $k$, in addition to $b$, to restrict rank of the estimator. Just as the histogram can approximate any density as $b \to \infty$, Lemmas \ref{lem:lrbias} and \ref{lem:tuckbias} show that the inclusion of $k$ does not hinder the approximation power of the estimator.  

To explore the trade-off between $b$ and $k$ first observe that setting $k = b^d$ gives $\sH_{d,b}^k = \sH_{d,b}$ since one can allocate one component to each bin. Theorem \ref{thm:genrate} and Proposition \ref{prop:finiteant} with $k=b^d$ gives a sample complexity of approximately $n \sim b^{d+1}$ which coincides with standard histogram estimator. Alternatively, setting $k=1$ restricts the estimator to separable histograms $\sH_{d,b}^1$, with a sample complexity of approximately $n\sim b$. Thus we have a span of $k$ yielding different estimators with maximal $k$ corresponding to the standard histogram and minimal $k$ corresponding to a naive Bayes assumption. We observe in Section \ref{sec:experiments} that this trade-off is useful in practice: we virtually never want $k$ to be maximized.
  
   \section{Experiments}\label{sec:experiments}
The previous section proved the existence of estimators that select NNTF histograms that offer advantages over the standard histogram. Unfortunately these estimators are not computationally tractable and only demonstrate the potential benefit incorporating an NNTF model in nonparametric density estimation. It would be nonetheless interesting to observe the behavior of an NNTF histogram estimator, even if it lacks the theoretical guarantees developed in the previous section. To this end we consider an $L^2$-minimizing NNTF histogram estimator. For all $d,b,k$ fix  $\sR_{d,b}^k$ to be either $\sH_{d,b}^k$ or $\tsH_{d,b}^k$. We consider an estimator $U_n$, that attempts to minimize 
\begin{equation*}
    U_n = \argmin_{\hat{p}\in \sR_{d,b}^k}\left\|\hat{p}- p\right\|^2_2.
\end{equation*}
Note that 
\begin{equation}\label{eqn:l2pop}
    \left\|\hat{p}- p\right\|^2_2
    = \left\|\hat{p}\right\|_2^2 - 2\left<p,\hat{p}\right> + \left\|p\right\|_2^2.
\end{equation}
The $\left\|p\right\|_2^2$ term in \eqref{eqn:l2pop} is not relevant when minimizing over $\hat{p}$. Additionally for data $X_1,\ldots,X_n \simiid p$ the law of large numbers gives us

$$\left<p,\hat{p}\right>=  \int_{[0,1)^d} p(x)\hp(x)dx = \E_{X\sim p}\left[\hp\left(X\right)\right]\approx \frac{1}{n}\sum_{i=1}^n \hat{p}(X_i)$$

so we may consider
\begin{equation}\label{eqn:l2prac}
	U_n \triangleq \argmin_{\hat{p}\in \sR{d,b}} \left\|\hat{p}\right\|_2^2 -\frac{2}{n} \sum_{i=1}^n \hat{p}(X_i)
\end{equation}
as a practical version of \eqref{eqn:l2pop} which is conveniently equivalent to nonnegative tensor factorization (see the appendix).

Risk expressions for nonparametric density estimation based on $L^2$-minimization similar to \eqref{eqn:l2prac} have appeared in previous works related to kernel density estimation \cite{vandermeulen13,cortes17,ritchie21}. Proving optimal rates on bandwidth in these settings seems challenging, however.

\begin{table}[ht]
  \label{tab:results}
  \centering
  \scriptsize
  \caption{Experimental Results}
  \begin{tabular}{c *{8}{c}}
    \hline
    Dataset &$d$ Red.&Dim.& Hist. Perf. & Tucker Perf. & Hist. Bins & Tucker Bins &Tucker $k$ & $p$-val.\\
    \hline \hline
    \multirow{8}{*}{MNIST}
    &
    \multirow{4}{*}{PCA}
    &2&-1.455$\pm$0.089&-1.502$\pm$0.102&6.531$\pm$1.499&8.375$\pm$1.780&4.968$\pm$1.976&5e-4\\
    &&3&-2.040$\pm$0.196&-2.268$\pm$0.195&4.781$\pm$0.738&6.718$\pm$1.565&5.781$\pm$1.340&2e-4\\
    &&4&-3.532$\pm$0.996&-4.014$\pm$0.655&4.031$\pm$0.585&5.343$\pm$1.018&4.375$\pm$0.695&2e-3\\
    &&5&-4.673$\pm$1.026&-6.157$\pm$2.924&3.468$\pm$0.499&4.343$\pm$0.592&3.281$\pm$0.514&4e-5\\\cline{2-9}

    &\multirow{4}{*}{Rand.}
    &2&-2.034$\pm$0.100&-2.099$\pm$0.102&6.062$\pm$1.197&7.562$\pm$1.657&2.062$\pm$1.784&3e-5\\
    &&3&-3.086$\pm$0.207&-3.331$\pm$0.387&4.812$\pm$0.526&6.843$\pm$1.227&2.687$\pm$1.959&1e-4\\
    &&4&-4.307$\pm$0.290&-5.731$\pm$0.435&3.500$\pm$0.559&5.656$\pm$0.642&2.593$\pm$1.497&8e-7\\
    &&5&-6.327$\pm$0.522&-9.539$\pm$1.053&3.250$\pm$0.433&4.718$\pm$0.571&2.562$\pm$1.087&8e-7\\
    \hline
    
    \multirow{8}{*}{Diabetes}
    &
    \multirow{4}{*}{PCA}
    &2&-2.079$\pm$0.122&-2.212$\pm$0.132&5.718$\pm$1.304&7.468$\pm$1.478&1.062$\pm$0.242&8e-6\\
    &&3&-3.010$\pm$0.364&-3.606$\pm$0.420&3.593$\pm$0.860&7.062$\pm$1.058&1.843$\pm$1.543&2e-6\\
    &&4&-4.002$\pm$0.415&-4.423$\pm$0.701&3.000$\pm$0.000&5.906$\pm$0.804&2.343$\pm$1.107&2e-3\\
    &&5&-6.139$\pm$0.661&-6.043$\pm$1.192&3.000$\pm$0.000&3.750$\pm$0.968&1.843$\pm$0.617&0.91\\\cline{2-9}

    &\multirow{4}{*}{Rand.}
    &2&-3.074$\pm$0.224&-3.277$\pm$0.287&6.843$\pm$1.227&9.250$\pm$1.936&1.093$\pm$0.384&7e-5\\
    &&3&-4.726$\pm$0.483&-5.353$\pm$0.751&4.968$\pm$0.769&8.406$\pm$1.343&1.625$\pm$1.672&2e-5\\
    &&4&-6.017$\pm$0.873&-7.732$\pm$1.497&4.062$\pm$0.704&6.718$\pm$1.328&2.093$\pm$1.155&1e-5\\
    &&5&-8.986$\pm$1.292&-12.61$\pm$2.477&3.062$\pm$0.242&5.093$\pm$0.521&2.531$\pm$0.865&2e-6\\
    \hline
  \end{tabular}
\end{table}
For our experiments we used the Tensorly library \cite{tensorly} to perform the nonnegative Tucker decomposition \cite{kim07} with Tucker rank $[k,k,\ldots,k]$ which was then projected to the simplex of probability tensors using \cite{duchi08}. We also performed experiments with nonnegative PARAFAC decompositions using \cite{shashua05,tensorly}. These decompositions performed poorly. This is potentially because the PARAFAC optimization is more difficult or the additional flexibility of the Tucker decomposition was more appropriate for the experimental datasets.
\subsection{Experimental Setup}\label{sec:expsetup}
Our experiments were performed on the Scikit-learn \cite{sklearn} datasets MNIST and Diabetes \cite{sklearn}, with labels removed. We use the estimated risk from \eqref{eqn:l2prac} to evaluate estimator performance (which may cause negative performance values, however lower values always indicate smaller estimated $L^2$-distance). Our experiments considered estimating histograms in $d=2,3,4,5$ dimensional space. We consider two forms of dimensionality reduction. First we consider projecting the dataset onto its top $d$ principle components. We also performed experiments projecting each dataset onto a random subspace of dimension $d$. These random subspaces were constructed so that each additional dimension adds a new index without affecting the others. To do this, we first randomly select an orthonormal basis for each dataset that remains unchanged for all experiments $v_1,v_2,\ldots$. Then to transform a point $X$ to dimension $d$ we perform the following transform:
  $X_{\text{reduced dim.}  }=\left[
    v_1\cdots v_d
\right]^T X$.
 We consider both transforms since PCA may select dimensions where the features tend to be independent, e.g. a multivariate Gaussian. After dimensionality reduction we scale and translate the data to fit in the unit cube.

With our preprocessed dataset, each experiment consisted of randomly selecting 200 samples for training and using the rest to evaluate performance (again using \eqref{eqn:l2prac}). For the estimators we tested all combinations using 1 to $\bmax$ bins per dimension and $k$ from 1 to $\kmax$. As $d$ increased the best cross validated $b$ and $k$ value decreased, so we reduced $\bmax$ and $\kmax$ for larger $d$ to reduce computational time, while still leaving a sizable gap between the best cross validated $b$ and $k$ and $\bmax$ and $\kmax$ across all runs of all experiment. For $d=2,3$ we have $\bmax = 15$ and $\kmax = 10$; for $d=4$ we have $\bmax = 12$ and $\kmax = 8$; for $d=5$ we have $\bmax = 8$ and $\kmax = 6$.
For parameter fitting we used random subset cross validation repeated 80 times using 40 of the 200 samples to evaluate the performance of the estimator fit using the other 160 samples. Performing 80 folds of cross validation was necessary because of the high variance of the histogram's estimated risk. This high variance is likely due to the noncontinuous nature of the histogram estimator itself and the noncontinuity of the histogram as a function of the data, i.e. slightly moving one training sample can potentially change histogram bin in which it lies. Each experiment was run 32 times and we report the mean and standard deviations of estimator performance as well as the best parameters found from cross validation. We additionally apply the two-tailed Wilcoxon signed rank test to the 32 pairs of performance results to statistically determine if the mean performance between the standard histogram and our algorithm are different and report the corresponding $p$-value. 
\subsection{Results}\label{ssec:results}
Our results are presented in Table 1. Apart from the density estimators' performance (``Hist. Perf.'' and ``Tucker Perf.'') this table contains the mean and standard deviation over the 32 trials for the optimal cross validated parameters for the estimators. This includes the number of bins, corresponding to $b$ from the earlier sections, and the Tucker rank $k$. We see that the Tucker histogram usually outperforms the standard histogram (``Tucker Perf.'' is less than ``Hist. Perf.'') with high statistical significance (small $p$-val). As one would expect, the Tucker histogram cross validates for more bins, presumably reducing the $k$ in exchange for more bins. Note that the MNIST PCA experiments always cross validated for the largest $k$ and yielded less statistically significant performance increases (outside of the unusual Diabetes PCA $d=5$ experiment). Presumably this experiment fit the NNTF assumption the least well (hence the large $k$) and thus received the least benefit.

\section{Conclusion}\label{sec:conclusion}
In this paper we have theoretically and experimentally demonstrated the advantages of including rank restriction into nonparametric density estimation. In particular, rank restriction may be a way to overcome the curse of dimensionality since it reduces the sample complexity penalty incurred from dimensionality from exponential to linear. This paper is an initial theoretical foray for demonstrating the potential of combining rank restriction with nonparametric methods.
\begin{ack} 
    Robert A. Vandermeulen acknowledges support by the Berlin Institute for the Foundations of Learning and Data (BIFOLD) sponsored by the German Federal Ministry of Education and Research (BMBF). Antoine Ledent acknowledges support by the German Research Foundation (DFG) award KL 2698/5-1 and the Federal Ministry of Science and Education (BMBF) award 01IS18051A. The authors thank Marius Kloft and Clayton Scott for useful discussions related to this work. The authors thank Nicholas D. Sidiropoulos for bringing to our attention some previous works on low-rank nonparametric density estimation.
\end{ack}


    \appendix
    \section*{Appendix Overview}
    Section \ref{appx:basic} contains the basic theoretical tools we need for the rest of the appendix. We will also be using some notation from the paper without introduction. In Section \ref{appx:asy} we have all the proofs for the results relating to asymptotic parameter rates of the histogram spaces in our estimators; this section corresponds to Section 2.2.1 in the main text. Section \ref{appx:finite} contains the proofs relating to finite-sample rates, and several expanded or additional results, corresponding to Section 2.2.2 in the main text. Section 2.2.2 in the main text contains some results that are simplified for readability and this appendix has some results which are tighter, but less readable. Additionally Section \ref{appx:finite} contains all the finite-sample analysis for Tucker histograms which were omitted in the main text. Experimental details, specifically the equivalency of $L^2$ minimization and nonnegative tensor factorization can be found in Section \ref{appx:exp}. Section \ref{appx:nonexist} is a bit different from the other sections and contains some discussion relevant to the future of this direction of research. There we show that not all densities can be written as a summation of separable densities with nonnegative coefficients.
    \section{Theoretical Basics} \label{appx:basic}
    \subsection{Notation}
    In this section we include some notation that was not contained in the main text but will be necessary for the rest of the appendix. In particular we will we need to introduce a fair amount of tensor notation which will be used to represent the various histograms. A histogram on the unit cube can naturally be represented as a tensor, which will be useful for deriving many of the results in this paper.

    From the main text recall that $\sT_{d,b}$ is the set of probability tensors in $\rn^{ b^{\times d}}$, i.e. their entries are nonnegative and sum to one. 

    Let $\sT_{d,b}^k$ be the set of tensors that are a convex combination of $k$ separable probability tensors (which are analogous to multi-view models) i.e.
    \begin{align*}
      \sT_{d,b}^k \triangleq \left\{ \sum_{i = 1}^k w_i \prod_{j=1}^d p_{i,j} \middle| w\in \Delta_k, p_{i,j} \in \Delta_b\right\}.
    \end{align*}

    The following is the set of probability tensors constructed via a nonnegative Tucker factorization ($\left[k\right]^d$ represents a multi-index)
    \begin{align*}
      \tsT_{d,b}^k \triangleq \left\{ \sum_{S \in \left[k\right]^{d}} W_S \prod_{i=1}^d p_{i,S_i} \middle| W\in \sT_{d,k}, p_{i,j} \in \Delta_b\right\}.
    \end{align*}

    For a multi-index $A \in \left[b\right]^d$ we define $\be_{d,b,A}$ as the element of $\sT_{d,b}$ where the $(A_1,\ldots, A_d)$-th entry is one and is zero elsewhere.

    Note that there exists a $\ell^1 \to L^1$ linear isometry $U_{d,b}:\sT_{d,b} \to \sH_{d,b}$ with $U_{d,b}$ defined as
    \begin{equation*}
      U_{d,b}(\be_{d,b,A}) = h_{d,b,A}.
    \end{equation*}
    The inverse function, $U^{-1}_{d,b}$, simply transforms a histogram to the tensor representing its bin weights and $U_{d,b}$ performs the reverse transformation. Note that $U_{d,b}$ is also a bijection between $\sT_{d,b}^k \to \sH_{d,b}^k$ and $\tsT_{d,b}^k \to \tsH_{d,b}^k$. Much of our analysis on histograms will be performed on the space of probability tensors with the analysis being translated to histograms via this operator.

    For a set of vectors $\sV$ we define $k\mix\left(\sV\right)\triangleq \left\{\sum_{i=1}^k w_i v_i \middle| w\in \Delta_k, v_i \in \sV \right\}$, i.e. the set of convex combinations of collections of $k$ vectors from $\sV$. We define $N\left(\sV, \varepsilon \right)$ to be the minimum cardinality for a subset of of $\sV$ which $\varepsilon$-covers $\sV$ (with closed balls) with respect to the $\left\|\cdot \right\|_1$ metric. It will be clear from context whether $\left\|\cdot \right\|_1$ represents the $\ell^1$, $L^1$, or total variation norm.
    \subsection{Preliminary Results}
    The following lemmas will be useful for all the theoretical results.
    \begin{lem} \label{lem:histcover}
      For all $0 < \varepsilon \le 1$ we have that 
      $N\left(\Delta_{b},\varepsilon\right)
      \le 
      \left(\frac{2b}{\varepsilon} \right)^b$.
    \end{lem}

    \begin{lem}\label{lem:simplecover}
      \sloppy For all $0 < \varepsilon \le 1 $ we have that $N\left(\sT_{d,b}^1, \varepsilon \right) \le \left(\frac{2bd}{\varepsilon} \right)^{bd}$.
    \end{lem}

    \begin{lem}\label{lem:mixcover}
      Let $\sP$ be a set of probability measures, then
      \begin{align*}
        N\left(k\operatorname{-mix}\left(\sP\right), \varepsilon + \delta \right) \le N\left(\sP,\varepsilon\right)^k N\left(\Delta_k,\delta\right).
      \end{align*}
    \end{lem}

    \begin{lem}\label{lem:lrcover}
      For all $0 < \varepsilon \le 1$ the following holds $N\left(\sT_{d,b}^k, \varepsilon \right) \le \left(\frac{4bd}{\varepsilon} \right)^{bdk} \left( \frac{4k}{\varepsilon}\right)^k$.
    \end{lem}
    Through application of the $U_{d,b}$ operator we now have a characterization of the complexity of the space $\sH_{d,b}^k$.

    \begin{cor}\label{cor:lrcover}
      For all $0 < \varepsilon \le 1$ following holds $N\left(\sH_{d,b}^k,\varepsilon\right)\le \left(\frac{4bd}{\varepsilon} \right)^{bdk} \left( \frac{4k}{\varepsilon}\right)^k.$
    \end{cor}
    The following are analogous results for Tucker histograms.
    \begin{lem}\label{lem:tuckcover}
      For all $0 < \varepsilon \le 1$ the following holds $N\left(\tsT_{d,b}^k, \varepsilon \right) \le \left(\frac{4bd}{\varepsilon} \right)^{bdk} \left( \frac{4k^d}{\varepsilon}\right)^{k^d}$.
    \end{lem}

    \begin{cor}\label{cor:tuckcover}
      For all $0 < \varepsilon \le 1$ following holds $N\left(\tsH_{d,b}^k,\varepsilon\right)\le \left(\frac{4bd}{\varepsilon} \right)^{bdk} \left( \frac{4k^d}{\varepsilon}\right)^{k^d}.$
    \end{cor}

    The following lemma from \cite{ashtiani18} provides us with a way to choose good estimators from finite collections of densities. It can be proven by applying a Chernoff bound to \cite{devroye01}, Theorem 6.3.
    \begin{lem}[Thm 3.4 page 7 of \cite{ashtiani18}, Thm 3.6 page 54 of~\cite{devroye01}] \label{lem:densalg}
      There exists a deterministic algorithm that, given a collection of distributions $p_1,\ldots,p_M$, a parameter $\varepsilon >0$ and at least $\frac{\log \left(3M^2/\delta \right)}{2\varepsilon^2}$ iid samples from an unknown distribution $p$, outputs an index $j\in \left[M\right]$ such that
      \begin{align*} 
        \left\|p_j - p \right\|_1 \le 3 \min_{i \in \left[M\right]} \left\|p_i - p\right\|_1 + 4 \varepsilon
      \end{align*}
      with probability at least $1 - \frac{\delta}{3}$.
    \end{lem}
    We present the following asymptotic version of the previous lemma. We highlight the use of finding sufficiently slow rates on parameters in order to establish asymptotic results, a technique which we will use in later proofs.

    \begin{lem}\label{lem:finiteest}
      Let $\left( \sP_n\right)_{n \in \mathbb{N}}$ be a sequence of finite collections of densities in $\sD_d$ where $\left| \sP_n\right| \to \infty$ with $n/\log\left(\left|\sP_n\right|\right)\to \infty$. Then there exists a sequence of estimators $V_n \in \sP_n$ such that, for all $\gamma >0$,
      \begin{align*}
        \sup_{p \in \sD_d }P\left(\left\|V_n- p\right\|_1
        > 3 \min_{q\in \sP_n}\left\|p -q \right\|_1 + \gamma    \right) \to 0,
      \end{align*}
      where $V_n$ is a function of $X_1,\ldots,X_n \simiid p$.
    \end{lem}

    \begin{proof}[Proof of Lemma \ref{lem:finiteest}]
      Let $M = M(n)= \left| \sP_n\right|$. Since $n/\log\left(M\right) \to \infty$ we have that for all $c>0$ there exists a $N_c$ such that, for all $n\ge N_c$ we have $n/\log\left(M\right) \ge c $ or equivalently $n\ge c\log\left(M\right)$. Because of this there exists sequence of positive values $C = C(n)$ such that $C \to \infty$ and $n \ge C\log\left(M\right)$.

      We will be making use of the algorithm in Lemma \ref{lem:densalg} as well as its notation. If we can show that there exist sequences of positive values $\varepsilon(n)\to 0, \delta(n)\to 0$ such that, for sufficiently large $n$, the following holds
      \begin{align*}
        \frac{\log \left(3M^2/\delta \right)}{2\varepsilon^2}\le n,
      \end{align*}
      then can simply set $V_n$ equal to be the estimator from Lemma \ref{lem:densalg} for sufficiently large $n$ and, because the lemma holds independent of choice of $p$, the theorem statement follows.

      Let $\varepsilon = \left(2/C \right)^{1/4}$ and $\delta = 3/\left(\exp\left(2 \sqrt{\frac{C}{2}} \right) \right)$. Note that these are both positive sequences which converge to zero. Now we have
      \begin{align}
        &\frac{\log\left(3M^2/\delta\right)}{2 \varepsilon^2}\notag = \frac{\log\left(M^2\right) + \log\left(3/\delta \right)}{2 \varepsilon^2}\\
        & = \frac{2\log\left(M\right) + \log\left(3/\delta \right)}{2 \varepsilon^2}\notag = \frac{\log\left(M\right) + \frac{1}{2}\log\left(3/\delta \right)}{\varepsilon^2}\notag\\
        & = \varepsilon^{-2}\left(\log\left(M\right) + \frac{1}{2}\log\left(3/\delta \right)\right)\notag\\
        &= \left(\left(2/C \right)^{1/4}\right)^{-2}\left(\log\left(M\right) + \frac{1}{2}\log\left(\exp\left(2 \sqrt{\frac{C}{2}} \right) \right) \right)\notag\\
        & = \sqrt{\frac{C}{2}} \left(\log(M)+ \sqrt{\frac{C}{2}} \right) = \sqrt{\frac{C}{2}} \log(M)+ \frac{C}{2}\label{eqn:finiteest}.
      \end{align}
      For sufficiently large $C$ and $M$ we have that the RHS of (\ref{eqn:finiteest}) is less than or equal to
      \begin{align*}
        \frac{C}{2} \log(M)+ \frac{C}{2}
        &\le \frac{C}{2} \log(M)+ \frac{C}{2}\log(M)	\\
        &= C\log(M) \le n.
      \end{align*}
      which completes our proof.
    \end{proof}

    \subsection{Theoretical Basics Proofs} \label{appx:proofs}
    \textbf{All norms are either the $\ell^1$, $L^1$, or total variation norm}, which are equivalent with respect to our analysis and the proper norm will be clear from context.

    \begin{proof}[Proof of Lemma \ref{lem:histcover}]
      In Section 7.4 from \cite{devroye01}, the authors show that for any collection of measures $\mu_1,\ldots,\mu_b$, for all $\varepsilon >0$, that 
      \begin{align*}
        N\left( \conv\left(\left\{ \mu_1,\ldots, \mu_b \right\}\right), \varepsilon  \right) \le \left(b + \frac{b}{\varepsilon} \right)^b.
      \end{align*}
      With the additional assumption that $\varepsilon \le 1$ we have that $b + \frac{b}{\varepsilon} \le\frac{b}{\varepsilon}+\frac{b}{\varepsilon} =  \frac{2b}{\varepsilon}$ and thus
      \begin{align*}
        N\left( \conv\left(\left\{ \mu_1,\ldots, \mu_b \right\}\right)\right)\le \left(\frac{2b}{\varepsilon} \right)^b.
      \end{align*}
      If we let $\mu_i = \be_i$, the indicator vector at index $i$, then the lemma follows.
    \end{proof}

    \begin{proof}[Proof of Lemma \ref{lem:simplecover}]
      From Lemma \ref{lem:histcover} we know there exists a finite collection of probability vectors $\tsP$ such that $\tsP$ is an $\varepsilon/d$-covering of $\Delta_b$ and $\left| \tsP\right| \le \left(\frac{2bd}{\varepsilon} \right)^b$.
      Note that the set $\left\{\tp_1\otimes\dots\otimes \tp_d \mid \tp_i \in \tsP \right\}$
      contains at most $\left( \left(\frac{2bd}{\varepsilon} \right)^b\right)^d = \left(\frac{2bd}{\varepsilon} \right)^{bd}$ elements. We will now show that this set is an $\varepsilon$-cover of $\sT_{d,b}^1$.
      Let $p_1\ot \cdots \ot p_d \in \sT_{d,b}^1$ be arbitrary. From our construction of $\tsP$ there exist elements $\tp_1,\ldots,\tp_d \in \tsP$ such that $\left\|p_i - \tp_i\right\|_1 \le \frac{\varepsilon}{d}$.

      We will now make use of Lemma 3.3.7 in \cite{reiss89}, which states that, for any collection of probability vectors $q_1,\ldots, q_d$ and $\tq_1,\ldots, \tq_d$, the following holds
      \begin{align*}
        \left\|\prod_{i=1}^d q_i -\prod_{j=1}^d \tq_j\right\|_1 \le \sum_{i=1}^d \left\|q_i- \tq_i \right\|_1.
      \end{align*}
      From this it follows that
      \begin{align*}
        \left\| \prod_{i=1}^d p_i - \prod_{j=1}^d \tp_j \right\|_1
        \le \sum_{i=1}^d \left\|p_i - \tp_i\right\|_1
        \le d \frac{\varepsilon}{d}
        = \varepsilon
      \end{align*}
      thus completing our proof.
    \end{proof}

    \begin{proof}[Proof of Lemma \ref{lem:mixcover}]
      Let $\tsP$ be the finite collection of probability measures with $|\tsP| = N\left(\sP,\varepsilon\right)$ which $\varepsilon$-covers $\sP$. Similarly let $W \subset \Delta_k$ with $|W| = N\left(\Delta_k,\delta\right)$  such that $W$ is a $\delta$-cover of $\Delta_k$. Consider the set 
      \begin{align*}
        \Omega = \left\{ \sum_{i=1}^k \tw_i \tp_i \middle| \tw \in W, \tp_i \in \tsP   \right\}.
      \end{align*}
      Note that this set contains at most $N\left(\sP,\varepsilon\right)^k N\left(\Delta_k,\delta\right)$ elements. We will now show that it $\left(\delta + \varepsilon\right)$-covers $k\mix\left(\sP\right)$, which completes the proof. Let $\sum_{i=1}^k p_i w_i \in k\operatorname{-mix}\left(\sP\right)$. We know there exists elements $\tp_1,\ldots,\tp_k \in \tsP$ such that $\left\|\tp_i - p_i\right\|_1\le \varepsilon$ and $\tw \in W$ such that $\left\| w - \tw \right\|_1 \le \delta$ and thus $\sum_{i=1}^k \tp_i \tw_i \in \Omega$. Now observe that
      \begin{align*}
        \left\|\sum_{i=1}^k \tp_i \tw_i - \sum_{j=1}^{k} p_j  w_j\right\|_1
        &= \left\|\sum_{i=1}^k \tp_i \tw_i - \sum_{j=1}^k p_j \tw_j + \sum_{l=1}^k p_l \tw_l - \sum_{r=1}^{k} p_r  w_r\right\|_1 \\
        &\le \left\|\sum_{i=1}^k \left(\tp_i - p_i\right) \tw_i     \right\|_1+ \left\| \sum_{i=1}^k p_i \left(\tw_i -  w_i\right)\right\|_1 \\
        &\le \sum_{i=1}^k \tw_i \left\| \tp_i - p_i \right\|_1+  \sum_{i=1}^k  \l|\tw_i -  w_i\r| \\
        &\le \sum_{i=1}^k \tw_i \varepsilon + \left\|\tw -  w\right\|_1 \\
        &\le \varepsilon + \delta.
      \end{align*}
    \end{proof}

    \begin{proof}[Proof of Lemma \ref{lem:lrcover}]
      Note that $\sT_{d,b}^k = k\operatorname{-mix}\left(\sT_{d,b}^1\right)$.
      Applying Lemma \ref{lem:mixcover} followed by Lemmas \ref{lem:histcover} and \ref{lem:simplecover} we have that
      \begin{align*}
        N\left(\sT_{d,b}^k, \varepsilon \right) 
        \le N\left(\sT_{d,b}^1, \varepsilon/2 \right)^k N\left(\Delta_k, \varepsilon/2 \right)
        \le \left(\frac{4bd}{\varepsilon} \right)^{bdk} \left( \frac{4k}{\varepsilon}\right)^k.
      \end{align*}
    \end{proof}

    \begin{proof}[Proof of Lemma \ref{lem:tuckcover}]
      Fix $k,d,b$ and $0<\varepsilon \le 1$. We are going to construct an $\varepsilon$-cover of $\tsT_{d,b}^k$.
      From Lemma \ref{lem:histcover} we know that there exists a set $\sB \subset \Delta_b$ which $\left(\frac{\varepsilon}{2d}\right)$-covers of $\Delta_b$ and contains no more than $\left(\frac{4bd}{\varepsilon}\right)^b$ elements. Let $\sP$ be the collection of all $d\times k$ arrays whose entries are elements from $\sB$. So we have that 
      \begin{align*}
        \left|\sP\right| = \left|\sB\right|^{dk}\le \left(\frac{4bd}{\varepsilon}\right)^{bdk}.
      \end{align*}

      From Lemma \ref{lem:histcover} there exists $\sW$ which is an $\varepsilon/2$-cover of $\sT_{d,k}$ and contains no more than $\left(4k^d/\varepsilon\right)^{\left(k^d\right)}$ elements. Now let
      \begin{align*}
        \sL_{d,b}^k =  \left\{ \sum_{S \in [k]^d} \tW_S \prod_{i=1}^d \tp_{i,S_i} \middle| \tW \in \sW, \tp\in \sP \right\}.
      \end{align*}
      Note that 
      \begin{align*}
        \left| \sL_{d,b}^k\right| \le \left|\sW\right| \left|\sP\right| \le \left(\frac{4k^d}{\varepsilon} \right)^{k^d} \left(\frac{4bd}{\varepsilon}\right)^{bdk}.
      \end{align*}
      We will now show that $\sL_{d,b}^k$ is an $\varepsilon$-cover of $\tsT_{d,b}^k$. To this end let $\sum_{S\in \left[k\right]^d} W_S \prod_{i=1}^{d}p_{i,S_i} \in \tsT_{d,b}^k$ be arbitrary, where $W\in \sT_{d,k}$ and $p_{i,j} \in \Delta_b$. From our construction of $\sW$, there exists $\tW \in \sW$ such that $\left\|W - \tW\right\|_1 \le \varepsilon/2$. There also exists $\tp \in \sP$ such that $\left\|\tp_{i,j} - p_{i,j}\right\|_1 \le \varepsilon/2$ for all $i,j$. Therefore we have that
      \begin{align*}
        \sum_{S \in \left[k\right]^d} \tW_S \prod_{i=1}^d \tp_{i,S_i} \in \sL_{d,b}^k.
      \end{align*}
      So finally 
      \begin{align*}
        &\left\|\sum_{S\in \left[k\right]^d} W_S \prod_{i=1}^{d}p_{i,S_i} - \sum_{R \in \left[k\right]^d} \tW_R \prod_{j=1}^d \tp_{j,R_j} \right\|_1\\
        &\le \left\|\sum_{S\in \left[k\right]^d} W_S \prod_{i=1}^{d}p_{i,S_i} - \sum_{R \in \left[k\right]^d} W_R \prod_{j=1}^d \tp_{j,R_j} \right\|_1
        +\left\|\sum_{S \in \left[k\right]^d} W_S \prod_{i=1}^d \tp_{i,S_i} - \sum_{R \in \left[k\right]^d} \tW_R \prod_{j=1}^d \tp_{j,R_j} \right\|_1\\
        &\le\sum_{S\in \left[k\right]^d}W_S \left\|  \prod_{i=1}^{d}p_{i,S_i} - \prod_{j=1}^d \tp_{j,S_j} \right\|_1
        +\sum_{R \in \left[k\right]^d} |W_R- \tW_R|\left\|   \prod_{j=1}^d \tp_{j,R_j} \right\|_1\\
        &\le\sum_{S\in \left[k\right]^d}W_S \sum_{i=1}^d\left\| p_{i,S_i} - \tp_{i,S_i} \right\|_1
        +\sum_{R \in \left[k\right]^d} |W_R- \tW_R|\left\|   \prod_{j=1}^d \tp_{j,R_j} \right\|_1\\
        &\le\sum_{S\in \left[k\right]^d}W_S \frac{\varepsilon}{2}
        +\left\|W - \tW\right\|_1\\
        &\le \varepsilon/2+ \varepsilon/2 = \varepsilon.
      \end{align*}

    \end{proof}

    \newpage
    \section{Asymptotic Theoretical Results} \label{appx:asy}
    This section contains results related to the asymptotic results related to the growth of model parameters with respect to the number of training samples. It corresponds to Section 2.2.1 in the main text.
    \begin{proof}[Proof of Theorem \ref{thm:genrate}]
      We will be applying the estimator from Lemma \ref{lem:finiteest} to a series of $\delta$-covers of $\sH_{d,b}^k$. We begin by constructing a series of $\delta$-covers whose cardinality doesn't grow too quickly. Corollary \ref{cor:lrcover} states that, for all $0< \delta \le 1$, that $N\left(\sH_{d,b}^k,\delta\right) \le \left( \frac{4bd}{\delta} \right)^{bdk} \left( \frac{4k}{\delta}\right)^k$. \emph{For sufficiently large} $b$ and $k$ and \emph{sufficiently small} $\delta$, the following holds
      \begin{align}
        \log \left(\left(\frac{4bd}{\delta} \right)^{bdk} \left( \frac{4k}{\delta}\right)^k \right)\notag
        & =bdk \log\left(\frac{4bd}{\delta}\right) + k \log\left(\frac{4k}{\delta} \right)\notag\\
        & = bdk\left[\log \left(b\right) + \log\left(\frac{4d}{\delta}\right)\right] + k \left[\log\left(k\right) + \log\left(\frac{4}{\delta} \right)\right]\notag\\
        & \le bdk\left[\log \left(b\right) + \log\left(b\right)\log\left(\frac{4d}{\delta}\right)\right] + dk \left[\log\left(k\right) + \log \left(k\right)\log\left(\frac{4d}{\delta} \right)\right]\notag\\
        & = \left(bk \log\left( b\right) + k\log\left(k\right)\right) d\left( 1+\log \left(\frac{4d}{\delta}\right) \right) .\label{eqn:same}
      \end{align}
      Using the argument from the proof of Lemma \ref{lem:finiteest} we have that, because  $n/(bk\log(b) + k\log(k)) \to \infty$ there exists a sequence of positive values $C = C(n)$ such that $C \to \infty$ and $n>C\left[bk\log(b) + k\log(k)\right]$. If we let $\delta = \frac{4d}{\exp\left(\frac{C}{d}-1\right)}$ we have that $\delta \to 0$ and
      \begin{align*}
        \left(bk \log\left( b\right) + k\log\left(k\right)\right) d\left( 1+\log \left(\frac{4d}{\delta}\right) \right) \le n.
      \end{align*}
      Because of this we can construct collections of densities $\tsP_n\subset \sH_{d,b}^k$ such that $\tsP_n$ is a $\delta$-covering of $\sH_{d,b}^k$ with $\l| \tsP\r| \to \infty$, $n/\log\l|\tsP_n\r| \to \infty$ and $\delta \to 0$. Let $V_n$ be the estimator from Lemma \ref{lem:finiteest} applied to the sequence $\tsP_n$.

      Let $\varepsilon >0$ be arbitrary. Due to the way that we have constructed the sequence $\tsP_n$, for sufficiently large $n$, we have that $3\sup_{q \in \sH_{d,b}^k} \min_{\tq \in \tsP_n}\left\|q - \tq \right\|_1 \le \varepsilon/2$. It therefore follows that, for sufficiently large $n$, the following holds for all $p \in \sD_d$
      \begin{align*}
        3 \min_{q \in \sH_{d,b}^k} \left\|p-q\right\|_1 + \varepsilon
        &\ge 3 \min_{q \in \sH_{d,b}^k} \left\|p-q\right\|_1 + 3\sup_{q \in \sH_b^k} \min_{\tq \in \tsP_n}\left\|q - \tq \right\|_1+ \varepsilon/2 \\
        &\ge 3 \min_{q \in \sH_{d,b}^k}\left[ \left\|p-q\right\|_1 + \min_{\tq \in \tsP_n}\left\|q - \tq \right\|_1\right]+ \varepsilon/2 \\
        &= 3 \min_{q \in \sH_{d,b}^k}\min_{\tq \in \tsP_n} \left\|p-q\right\|_1 + \left\|q - \tq \right\|_1+ \varepsilon/2 \\
        &\ge 3 \min_{\tq \in \tsP_n} \left\|p- \tq \right\|_1+ \varepsilon/2.
      \end{align*}
      From this we have that, for sufficiently large $n$
      \begin{align*}
        \sup_{p \in \sD_d }P\left(\left\|V_i - p\right\|_1 > 3 \min_{q\in \sH_{d,b}^k}\left\|p -q \right\|_1 + \varepsilon  \right)
        \le \sup_{p \in \sD_d }P\left(\left\|V_i - p\right\|_1 > 3 \min_{\tq \in \tsP_n} \left\|p- \tq \right\|_1+ \varepsilon/2  \right)
      \end{align*}
      and the right side goes to zero due to Lemma \ref{lem:finiteest}, thus completing the proof.
    \end{proof}

    \begin{proof}[Proof of Theorem \ref{thm:tuckrate}]
      This proof is very similar to the proof of Theorem \ref{thm:genrate}. We will be applying the estimator from Lemma \ref{lem:finiteest} to a series of $\delta$-covers of $\tsH_{d,b}^k$. We begin by constructing a series of $\delta$-covers whose cardinality doesn't grow too quickly. Corollary \ref{cor:tuckcover} states that, for all $0< \delta \le 1$, that $N\left(\tsH_{d,b}^k,\delta\right) \le\left(\frac{4bd}{\delta} \right)^{bdk} \left( \frac{4k^d}{\delta}\right)^{k^d}$. \emph{For sufficiently large} $b$ and $k$ and \emph{sufficiently small} $\delta$, the following holds
      \begin{align*}
        \log \left(\left(\frac{4bd}{\delta} \right)^{bdk} \left(\frac{4k^d}{\delta} \right)^{k^d} \right)
        &=bdk \log \left(\frac{4bd}{\delta}\right) + k^d \log \left(\frac{4k^d}{\delta} \right)\\
        &\le d\left(bk \log \left(\frac{4bd}{\delta}\right) + k^d \log \left(\frac{4k^d}{\delta}\right) \right)\\
        &= d\left(bk \left(\log(b) + \log \left(\frac{4d}{\delta}\right)\right) + k^d \left( \log\left(k^d\right) + \log\left(\frac{4}{\delta}\right)  \right) \right)\\
        &\le d\left(bk \left(\log(b)+ \log \left(\frac{4d}{\delta}\right)\right) + k^d \left( \log\left(k^d\right) + \log\left(\frac{4d}{\delta}\right)  \right) \right)\\
        &= \left(bk \log(b) + k^d  \log\left(k^d\right)  \right)d\left(1 + \log \left(\frac{4d}{\delta}\right)\right).
      \end{align*}
      Note that replacing $bk \log\left(b\right) + k \log\left(k\right)$ with $bk \log \left(b\right)  + k^d  \log\left(k^d\right)$ in the last line is exactly (\ref{eqn:same}) in our proof of Theorem \ref{thm:genrate} . From here we can proceed exactly as in the proof of Theorem \ref{thm:genrate} by replacing $\sH_{d,b}^k$ with $\tsH_{d,b}^k$ and $bk \log\left(b\right) + k \log\left(k\right)$ with $bk \log \left(b\right)  + k^d  \log\left(k^d\right)$.
    \end{proof}
    
    \begin{proof}[Proof of Corollary \ref{cor:fastbin}]
        This follows directly from Theorems \ref{thm:genrate} and \ref{thm:tuckrate} and selecting an appropriately slow rate for $k \to \infty$.
    \end{proof}

    \begin{proof}[Proof of Lemma \ref{lem:lrbias}]
      Let $\varepsilon >0$. Theorem 5 in Chapter 2 of \cite{gyorfi85}\footnote{See p. 20 in this text for the application to histograms.} states that, for any $p \in \sD_d$, that $\min_{h\in \sH_{d,b}}\left\|p-h\right\|_1 \to 0$ as $b\to \infty$, i.e. the bias of a histogram estimator goes to zero as the number of bins per dimension goes to infinity. Thus there exists a sufficiently large $B$ such that there exists a histogram $h \in \sH_{d,B}$ which is a good approximation of $p$, $\left\|p - h \right\|_1 <\varepsilon/2$. In this proof we we will argue that once $k\ge B^d$ and $b$ is sufficiently large, we can find an element of $\sH_{d,b}^k$ where the multi-view components can approximate the $B^d$ bins of $h$.

      We have that, for some $w\in \sT_{d,B}$
      \begin{align*}
        h 
        &= \sum_{A \in \left[B\right]^{ d}} w_A h_{d,B,A}.
      \end{align*}
      From the same theorem in \cite{gyorfi85} there exists $a_0$ such that, for all $a\ge a_0$,  for all $i$, there exists  $\th_{1,a,i} \in \sH_{1,a}$ such that $\left\|h_{1,B,i} -\th_{1,a,i} \right\|_1 < \varepsilon/(2d)$ for all $i\in [B]$. For any multi-index $A \in [B]^d$, we define
      \begin{align*}
        \th_{d,a,A} = \prod_{j=1}^d \th_{1,a,A_j}.
      \end{align*}
      Now we have that, for all $a \ge a_0$ and $A \in [B]^d$,
      \begin{align}
        \left\|h_{d,B,A} - \th_{d,a,A}\right\|_1
        & = \left\| \prod_{i=1}^d h_{1,B,A_i} - \prod_{j=1}^d \th_{1,a,A_j}\right\|_1\notag \\
        & \le \sum_{i=1}^d \left\|h_{1,B,A_i} - \th_{1,a,A_i} \right\|_1\label{eqn:supp}\\
        & \le d \frac{\varepsilon}{2d} \notag\\
        & = \varepsilon/2,\notag
      \end{align}
      where we use the previously mentioned product measure inequality for (\ref{eqn:supp}).
      As soon as $k\ge B^d$ and $a\ge a_0$ the set $\sH_{d,a}^k$ contains the element,
      \begin{align*}
        \th \triangleq \sum_{A \in \left[B\right]^{ d}} w_A \th_{d,a,A}.
      \end{align*}
      Now we have that, for all $a\ge a_0$.
      \begin{align*}
        \left\|h - \th\right\|_1
        & = \left\|\sum_{A \in \left[B\right]^{ d}} w_A h_{d,B,A} - \sum_{Q \in \left[B\right]^{ d}} w_Q \th_{d,a,Q}  \right\|_1\\
        & \le \sum_{A \in \left[B\right]^{ d}}w_A \left\|  h_{d,B,A} -  \th_{d,a,A}  \right\|_1\\
        & \le \varepsilon/2.
      \end{align*}
      From the triangle inequality we have that 
      \begin{align*}
        \left\|p - \th \right\|_1 \le 
        \left\|p- h \right\|_1 + \left\| h- \th\right\|_1 \le\varepsilon.
      \end{align*}
      So we have that, for sufficiently large $b$ and $k$ 
      \begin{align*}
        \min_{q \in \sH_{d,b}^k} \left\| p - q \right\|_1 \le \varepsilon
      \end{align*}
      which completes our proof.
    \end{proof}

    \begin{proof}[Proof of Lemma \ref{lem:tuckbias}]
      We will show that $\sH_{d,b}^k \subset \tsH_{d,b}^k$ and the lemma clearly follows due to Lemma \ref{lem:lrbias}. Any element of $\sH_{d,b}^k$ will have the following representation
      \begin{align}
        \sum_{i=1}^k w_i \prod_{j=1}^d f_{i,j}: w \in\Delta_k, f_{i,j} \in \sH_{1,b}.\label{eqn:tuckbias}
      \end{align}
      Letting $W \in \sT_{d,k}$ with $W_{i,\ldots,i} = w_i$ for all $i$, the rest of the entries of $W$ be zero, and $\tf_{j,i} = f_{i,j}$ for all $i,j$ we have that 
      \begin{align*}
        \sum_{S\in [k]^d} W_S \prod_{j=1}^d \tf_{j,S_j}
        &= \sum_{i=1}^k W_{i,\ldots,i}\prod_{j=1}^d \tf_{j,i}\\
        &= \sum_{i=1}^k w_i \prod_{j=1}^d f_{i,j}
      \end{align*}
      so we have that (\ref{eqn:tuckbias}) is an element of $\tsH_{d,b}^k$ and we are done.
    \end{proof}

    \begin{proof}[Proof of Theorem \ref{thm:lower}]
      We will proceed by contradiction. Suppose $V_n$ is an estimator violating the theorem statement, i.e. there exist sequences $b\to \infty$ and $k\to \infty$ with $n/\left(bk\right) \to 0$ and $b\ge k$ such that, for all $\varepsilon >0$,
      \begin{align*} 
        \sup_{p \in \sD_d }P\left(\left\|V_n - p\right\|_1 > 3 \min_{q\in \sH_{d,b}^k}\left\|p -q \right\|_1 + \varepsilon  \right) \to 0.
      \end{align*}

      Let $\left(p_n\right)_{n=1}^\infty$ be a sequence of probability vectors $p_n \in \Delta_{b(n)\times k(n)}$ which represent distributions over $\left[b(n)\right] \times \left[k(n)\right]$. Let $\sX_n \triangleq  \left(X_{n,1},\ldots,X_{n,n} \right)$ with $X_{n,1},\ldots,X_{n,n} \simiid p_n$.

      We will now construct a series of estimators for $p_n$ using $V_n$. Let $\tsX_n = \left(\tX_{n,1},\ldots, \tX_{n,n} \right)$ which are independent random variables with $\tX_{n,i} \sim h_{d,b,\left(X_{n,i},1,\ldots,1\right)}$, so $\tX_{n,i}$ is uniformly distributed over the bin designated by $X_{n,i},1,\ldots,1$. For this proof we will assume $d > 2$ but the proof can be simplified in a straightforward manner to the $d=2$ case by ignoring the indices and modes beyond the second. Note that that $X_{n,i}$ contains two indices.  Now we have the following for the densities of $\tX_{n,i}$
      \begin{align}
        p_{\tX_{n,i}} \notag
        &= \sum_{(j,\ell) \in \left[b\right] \times \left[k\right]} p_{\tX_{n,i} | X_{n,i} = \left(j,\ell\right)} P(X_{n,i} = \left(j,\ell\right))\\\notag
        &= \sum_{(j,\ell) \in \left[b\right] \times \left[k\right]} h_{d,b,\left(j,\ell,1,\ldots, 1\right)}  p_{n}\left(j,\ell\right)\\\notag
        &= \sum_{ \ell \in \left[k\right]} \sum_{ j \in \left[b\right]} h_{d,b,\left(j,\ell,1,\ldots, 1\right)}  p_{n}\left(j,\ell\right)\\\notag
        &= \sum_{ \ell \in \left[k\right]} \sum_{ j \in \left[b\right]}p_{n}\left(j,\ell\right) h_{1,b,j} \otimes h_{1,b,\ell} \otimes \prod_{a \in [d-2]}h_{1,b,1}\notag  \\
        &= \sum_{ \ell \in \left[k\right]} \left(\sum_{ j \in \left[b\right]}p_{n}\left(j,\ell\right) h_{1,b,j}\right) \otimes h_{1,b,\ell} \otimes \prod_{a \in [d-2]}h_{1,b,1} \label{eqn:histform} \\
        &= \sum_{ \ell \in \left[k\right]}\left(\sum_{ q \in \left[b\right]}p_{n}\left(q,\ell\right)\right) \left(\sum_{ j \in \left[b\right]}\frac{p_{n}\left(j,\ell\right)}{\sum_{ q \in \left[b\right]}p_{n}\left(q,\ell\right)} h_{1,b,j}\right) \otimes h_{1,b,\ell} \otimes \prod_{a \in [d-2]}h_{1,b,1}. \label{eqn:khistform}
      \end{align}
      For \eqref{eqn:khistform} we let $0/0$ be equal to zero as is common for discrete conditioning. This last line is in the form of (\ref{eqn:lrhist}) in the main text and is thus an element of $\sH_{d,b}^k$. To see this we will show the correspondence between the terms in (\ref{eqn:khistform}) from here and the terms in (\ref{eqn:lrhist}) in the main text:
      \begin{align*}
        w_\ell &:= \left(\sum_{ q \in \left[b\right]}p_{n}\left(q,\ell\right)\right)\\
        f_{\ell,1} &:= \left(\sum_{ j \in \left[b\right]}\frac{p_{n}\left(j,\ell\right)}{\sum_{ q \in \left[b\right]}p_{n}\left(q,\ell\right)} h_{1,b,j}\right) \\
        f_{\ell,2} &:= h_{1,b,\ell}\\
        f_{i,j} &:= h_{1,b,1}, \forall j>2, \forall i.
      \end{align*}
      Let $V_n$ estimate $\tP_n \triangleq p_{\tX_{n,i}}$ so $\tX_{n,1}, \ldots , \tX_{n,n} \simiid \tP_n$. We will use $V_n$ to construct an estimator $v_n$ for $p_n$.

      Because $\tP_n \in \sH_{d,b}^k$ \footnote{We will use this portion of the proof again for our proof of Theorem \ref{thm:tucklower} \label{fn:lower}} for all $n$ and our contradiction hypothesis we have that $\left\|V_n - \tP_n\right\|_1 \cip 0$. From this it follows that $\left\| U_{d,b}^{-1}(V_n) - U_{d,b}^{-1}(\tP_n)\right\|_1\cip 0 $. Note that $\left[U_{d,b}^{-1}(\tP_n)\right]_{j,\ell,A} = p_n(j,\ell)$ when $A=\left(1,\ldots,1\right)$ and zero otherwise (see (\ref{eqn:histform})). We define the linear operator $B_n: \sT_{d,b} \to \Delta_{b\times k}$ as \sloppy

      \begin{align*}
        \left[B_n(T)\right]_{j,\ell} \triangleq \sum_{A \in \left[b\right]^{ d - 2}} T_{j,\ell, A}
      \end{align*}
      i.e. the linear operator which sums out all modes except for the first two.
      We have that $B_n(U_{d,b}^{-1}(\tP_n)) = p_n$. Now let $v_n = B_n(U_{d,b}^{-1}(V_n))$ be the estimator for $p_n$. Now we have that
      \begin{equation*} \label{eqn:lower}
        \left\|v_n - p_n\right\|_1
        = \left\|B_n(U_{d,b}^{-1}(\tP_n)) - B_n(U_{d,b}^{-1}(V_n)) \right\|_1= \left\|B_n(U_{d,b}^{-1}(\tP_n - V_n)) \right\|_1 .
      \end{equation*}
      We have that $B_n$ is an $\ell^1$-nonexpansive operator due to the triangle inequality,
      \begin{align*}
        \left\|B_n\left(T\right)\right\|_1
        = \sum_{j,l} \l|\sum_{A \in \left[b\right]^{ d - 2}} T_{j,\ell, A}  \r| \le \sum_{j,l}\sum_{A \in \left[b\right]^{ d - 2}} \l| T_{j,\ell, A}  \r|
        = \left\|T\right\|_1,
      \end{align*}
      so the operator norm of $B_n$ is less than or equal to one. We also know that $U^{-1}_{d,b}$ an isometry and $\left\|\tP_n - V_n\right\|_1\cip 0$, so it follows that $\left\|v_n - p_n\right\|_1\cip 0$ for any sequence of $p_n\in \Delta_{\left[b(n)\right] \times \left[k(n)\right]}$. We will now use following theorem from \cite{han15} to show that no such estimator $v_n$ can exist.

      \begin{thm}[\cite{han15} Theorem 2.]
        For any $\zeta \in \left(0,1\right]$, we have
        \begin{align*}
          &\inf_{\hat{p}} \sup_{p\in \Delta_a} \E_p\left\|\hat{p}-p\right\|_1 \ge 
          \frac{1}{8} \sqrt{\frac{ea}{\left(1+\zeta\right)n}} \1 \left( \frac{\left(1 + \zeta\right)n}{a}>\frac{e}{16}\right)\\
          &\quad+ \exp\left(- \frac{2\left(1 + \zeta\right)n}{a}\right) \1 \left( \frac{\left(1 + \zeta\right)n}{a}\le\frac{e}{16}\right)
          - \exp\left(-\frac{\zeta^2n}{24}\right) - 12 \exp\left(- \frac{\zeta^2 a}{32 \left(\log a \right)^2}\right)
        \end{align*}
        where the infimum is over all estimators.
      \end{thm}
      Our estimator is equivalent to estimating a categorical distribution with $a = bk$ categories. Letting $\zeta = 1$, $bk \to \infty$, and $n\to \infty$, with $n/\left(bk\right) \to 0$, we get that for sufficiently large $n$
      \begin{align*}
        \inf_{\hat{p}} \sup_{p\in \Delta_{bk}} \E_p\left\|\hat{p}-p\right\|_1 \ge \exp\left(- \frac{4n}{bk}\right) - \exp\left(-\frac{n}{24}\right) - 12 \exp\left(- \frac{bk}{32 \left(\log bk \right)^2}\right)
      \end{align*}
      whose right hand side converges to 1. From this we get that 
      \begin{align*}
        \liminf_{n\to \infty} \sup_{p_n\in \Delta_{bk}} \E_{p_n}\left\|v_n-p_n\right\|_1  > \frac{1}{2}
      \end{align*}
      which contradicts $\left\|v_n - p_n\right\|_1\cip 0 $ for arbitrary sequences $p_n$.
    \end{proof}

    \begin{proof}[Proof of Theorem \ref{thm:tucklower}]

      We will proceed by contradiction. Suppose $V_n$ is an estimator violating the theorem statement, i.e. there exist sequences $b\to \infty$ and $k\to \infty$ with $n/\left(bk + k ^d\right) \to 0$ and $b\ge k$ such that, for all $\varepsilon >0$,
      \begin{align*}
        \sup_{p \in \sD_d }P\left(\left\|V_n - p\right\|_1 > 3 \min_{q\in \tsH_{d,b}^k}\left\|p -q \right\|_1 + \varepsilon  \right) \to 0.
      \end{align*}
      Since $n/(bk + k^d) \to 0$ we have that $(bk + k^d)/n \to \infty$ so there is a subsequence $n_i$ such that $b(n_i)k(n_i)/n_i \to \infty$ or $k(n_i)^d/n_i\to \infty$, or equivalently $n_i/(b(n_i)k(n_i)) \to 0$ or $n_i/k(n_i)^d \to 0$. We will show that both cases lead to a contradiction. We will let $b$ and $k$ be functions of $n_i$ implicitly when defining limits. \\
      \textbf{Case $n_i/(bk) \to 0$:}
      We proceed similarly to the proof of Theorem \ref{thm:lower}. Let $\left(p_n\right)_{n=1}^\infty$, $\tP_n$, and $\sX_n$ be defined as in the proof of Theorem \ref{thm:lower}. Note that $\sH_{d,b}^k \subset \tsH_{d,b}^k$ (see proof of Lemma \ref{lem:tuckbias}) and thus $\tP_n \in \tsH_{d,b}^k$. We can proceed exactly as in our proof of Theorem \ref{thm:lower} at footnote \ref{fn:lower}, by simply replacing $\sH_{d,b}^k$ with $\tsH_{d,b}^k$ and $n$ with $n_i$ which finishes this case.\\
      \textbf{Case $n_i/k^d \to 0$:} Let $(p_n)_{n=1}^\infty$ be a sequence of elements in $\sT_{d,k}$ which represents distributions over $[k]^d$. Let $\sX_n \triangleq  \left(X_{n,1},\ldots,X_{n,n} \right)$ with $X_{n,1},\ldots, X_{n,n}\simiid p_n$. Let $\tsX_n = \left(\tX_{n,1},\ldots, \tX_{n,n} \right)$ which are independent random variables with $\tX_{n,i}\sim h_{d,b,X_{n,i}}$. Let $\tP_n$ be the density for $\tX_{n,i}$. Note that $k\le b$. So we have that
      \begin{align*}
        \tP_n 
        &= \sum_{S\in [k]^d} p_{\tX_{n,i}|X_{n,i}= S}P(X_{n,i} = S)\\
        &= \sum_{S\in [k]^d} h_{d,b,S}p_n(S)\\
        &= \sum_{S\in [k]^d}p_n(S)\prod_{i=1}^d h_{1,b,S_i}
      \end{align*}
      and thus $\tP_n \in \tsH_{d,b}^k$. We proceed as in Theorem \ref{thm:lower} to find an estimator for elements of $\sT_{d,k}$ which is equivalent to estimating elements of $\Delta^{k^d}$ which is impossible since $n_i/k^d \to 0$.
    \end{proof}

    \newpage
    \section{Finite Sample Theoretical Results} \label{appx:finite}
    In this section we cover the proofs for the finite-sample results in our paper, Section 2.2.2 in the main text. This includes proofs related to estimator bias on Lipschitz continuous functions. We note that all projections in this paper are in their respective $L^2$ space. As noted in the main text we will be using \emph{set projection}, $\proj_S x \triangleq \arg \min_{s\in S} \left\|x-s\right\|_2$ \cite{bauschke2011convex}. When $S$ is a (closed) linear subspace this is equivalent to a linear projection.  Every instance of set projection in this work yields a unique minimizer. As in the main text let $\lip_L$ be the set of $L$-Lipschitz continuous densities on $[0,1]$ and let $mL = \sup_{f \in \lip_L}\left\|f\right\|_2$.

    \subsection{Bias}

    Theorem \ref{thm:l1bias} is in the main text but is not used as is in any of our other proofs and is meant simply to be illustrative of the behavior of the bias as in the main text. We will get the proof of this theorem out of the way before moving on to the core results of this portion of the proofs. Note that $\lambda$ is the standard Lebesgue measure.
    \begin{proof}[Proof of Theorem \ref{thm:l1bias}]
      From H\"older's Inequality we have that for any function $f:[0,1]^d \to \rn$ that
      \begin{equation} \label{eqn:holders}
        \left\|f \right\|_1 = \left\|f\cdot \1 \right\|_1 \le \left\|f \right\|_2 \left\|\1\right\|_2 = \left\|f\right\|_2.
      \end{equation}
      Applying this directly to the inequality from Theorem \ref{thm:rank-one-l2-bias} we have
      \begin{equation*}
        \left \| \prod_{i=1}^d f_i - \proj_{\sH_{d,b}^1}\prod_{i=1}^d f_i \right\|^2_2 \le m_L^{2d}- \left(m_L^2 - \frac{L^2}{12b^2}  \right) ^d.
      \end{equation*}
    \end{proof}

    We now developing the core results needed for the paper. Because of the consequence of H\"older's Inequality \eqref{eqn:holders} we will focus mainly on integrated squared distance between functions.

    \begin{lem} \label{lem:midapprox}
      Let $f:[a,b]\to \rn$ be an $L$-Lipschitz continuous function. Then  
      $$\min_{\alpha \in \rn} \int_a^b \left(\alpha - f(x)\right)^2 dx \le \frac{L^2\left(b - a \right)^3 }{12}.$$
    \end{lem}
    \begin{proof}[Proof of Lemma \ref{lem:midapprox}]
      Let $\alpha = f\left(\frac{a+b}{2}\right)$. Then we have that
      \begin{align*}
        \int_a^b \left(\alpha - f(x) \right)^2dx
        &\le  \int_a^b \left(L\left|x-\frac{a+b}{2}\right|\right)^2 dx\\
        &=  L^2\int_a^b \left(x-\frac{a+b}{2}\right)^2 dx\\
        &=  L^2\int_{a-\frac{a+b}{2}}^{b-\frac{a+b}{2}}x^2 dx\\
        &=  L^2\int_{\frac{a-b}{2}}^{\frac{b-a}{2}}x^2 dx\\
        &=  \frac{L^2}{3}\left[x^3 \right]_{\frac{a-b}{2}}^{\frac{b-a}{2}}\\
        &=  \frac{L^2}{3}2  \left(\frac{b-a}{2}\right)^3\\
        &= \frac{L^2\left(b - a \right)^3 }{12}.
      \end{align*}
    \end{proof}
    \begin{lem} \label{lem:1dhistproj}
      Let $f$ be an $L$-Lipschitz function on $[0,1]$, then $$\left\|f-\proj_{\spn \left(\sH_{1,b}\right)}f\right\|_2^2 \le \frac{L^2}{12b^2}.$$
    \end{lem}
    \begin{proof}[Proof of Lemma \ref{lem:1dhistproj}]
      Applying Lemma~\ref{lem:midapprox} we have that
      \begin{align}
          \left\|f-\proj_{\spn \left(\sH_{1,b}\right)}f\right\|_2^2
          &= \min_{w\in \rn^b} \left\|f - \sum_{i=1}^b w_i h_{1,b,i} \right\|_2^2 \notag\\
          &= \min_{w\in \rn^b} \left\|f - \sum_{i=1}^b w_i b \1_{\left[\frac{i-1}{b},\frac{i}{b} \right)} \right\|_2^2 \notag \\
          &= \min_{w\in \rn^b} \int_{[0,1]} \left( f(x) -\sum_{i=1}^b  w_i \1\left(\frac{i-1}{b}\le x <\frac{i}{b} \right)  \right)^2 dx\notag \\
          &= \min_{w\in \rn^b} \int_{[0,1]} \left(\sum_{i=1}^b \left(f(x) -  w_i\right) \1\left(\frac{i-1}{b}\le x <\frac{i}{b} \right)  \right)^2 dx \label{eqn:cross1}\\
          &= \min_{w\in \rn^b}  \int_{[0,1]} \sum_{i=1}^b \left( \left(f(x) -  w_i\right) \1\left(\frac{i-1}{b}\le x <\frac{i}{b} \right)  \right)^2 dx\label{eqn:cross2}\\
          &= \min_{w\in \rn^b} \sum_{i=1}^b \int_{\frac{i-1}{b}}^\frac{i}{b} \left( f(x) -  w_i \right)^2 dx\notag\\
          &\le b \frac{L^2}{12 b^3} = \frac{L^2}{12b^2}, \notag
      \end{align}
      where \eqref{eqn:cross1} to \eqref{eqn:cross2} is justified since, when distributing the square, the cross terms of the form $\1\left(\frac{i-1}{b}\le x <\frac{i}{b} \right)\1\left(\frac{j-1}{b}\le x <\frac{j}{b} \right)$ are equal to zero when $i\neq j$.
    \end{proof}

    \begin{lem}\label{lem:projprod}
      Let $f_1,\ldots,f_d \in L^2([0,1])$. Then 
      $$ \proj_{\spn\left(\sH_{d,b}\right)} \prod_{i=1}^d f_i = \prod_{i=1}^d \proj_{\spn\left(\sH_{1,b}\right)}f_i. $$
    \end{lem}
    \begin{proof}[Proof of Lemma \ref{lem:projprod}]
      We will be using a multi-index $A = \left(i_1,\ldots,i_d\right) \in [b]^{ d}$ so that $\prod_{j=1}^d h_{1,b,i_j} = h_{d,b,A}$ and note that $h_{d,b,A}\perp h_{d,b,A'}$ for $A\neq A'$, so $h_{d,b,A}$ is an \emph{orthogonal} basis for $\sH_{d,b}$. We have the following
      \begin{align}
        \prod_{i=1}^d \proj_{\spn\left(\sH_{1,b}\right)} f_i 
        &= \prod_{i=1}^d\sum_{j=1}^b \frac{h_{1,b,j}}{\left\|h_{1,b,j}\right\|_2^2}\left<h_{1,b,j},f_i\right>\notag\\
        \begin{split}&= \left( \frac{h_{1,b,1}}{\left\|h_{1,b,1}\right\|_2^2}\left<h_{1,b,1},f_1\right>+ \cdots +\frac{h_{1,b,b}}{\left\|h_{1,b,b}\right\|_2^2}\left<h_{1,b,b},f_1\right> \right)\otimes \\
        & \quad \cdots \otimes\left( \frac{h_{1,b,1}}{\left\|h_{1,b,1}\right\|_2^2}\left<h_{1,b,1},f_d\right> + \cdots +\frac{h_{1,b,b}}{\left\|h_{1,b,b}\right\|_2^2}\left<h_{1,b,b},f_d\right> \right).\end{split}\label{eqn:bigproduct}
      \end{align}
      Now, distributing the terms in \eqref{eqn:bigproduct} and consolidating the subscripts into $A$ we get that \eqref{eqn:bigproduct} is equal to
      \begin{align*}
        \sum_{A \in \left[b\right]^{ d}} \prod_{i=1}^d \frac{h_{1,b,A_i}}{\left\|h_{1,b,A_i}\right\|_2^2}\left<h_{1,b,A_i},f_i\right>
        &=\sum_{A \in \left[b\right]^{ d}}  \frac{\prod_{i=1}^d h_{1,b,A_i}}{\prod_{i=1}^d\left\|h_{1,b,A_i}\right\|_2^2}\prod_{i=1}^d\left<h_{1,b,A_i},f_i\right>\\
        &= \sum_{A \in \left[b\right]^{ d}}  \frac{h_{d,b,A}}{\left\|h_{d,b,A}\right\|_2^2}\left<h_{d,b,A},\prod_{i=1}^df_i\right>\\
        &=\proj_{\spn\left(\sH_{d,b}\right)} \prod_{i=1}^d f_i.
      \end{align*}
    \end{proof}
    
    \begin{lem}\label{lem:spanproj}
        Let $f \in L^2 \left(\left[0,1\right]^d \right)$ be a probability density. Then
        \begin{equation*}
            \proj_{\spn\left(\sH_{d,b}\right)} f = \proj_{\sH_{d,b}} f.
        \end{equation*}
    \end{lem}
    
    \begin{proof}[Proof of Lemma \ref{lem:spanproj}]
        Since $h_{d,b,A}\perp h_{d,b,A'}$ for $A\neq A'$ we have
        \begin{equation*}
            \proj_{\spn\left(\sH_{d,b}\right)} f
            =\sum_{A \in \left[b\right]^{ d}} \frac{h_{d,b,A}}{\left\|h_{d,b,A}\right\|_2^2}\left<h_{d,b,A},f\right>
        \end{equation*}
        and
        \begin{equation*}
            \proj_{\spn\left(\sH_{d,b}\right)} f
            =\sum_{A \in \left[b\right]^{ d}} w_A h_{d,b,A}
        \end{equation*}
        where
        \begin{equation*}
            w_A
            = \frac{\left<h_{d,b,A},f\right>}{\left\|h_{d,b,A}\right\|_2^2}.
        \end{equation*}
        Clearly $w_A\ge 0$ so to finish we need only show that $\sum_{A \in \left[b\right]^{ d}}w_A=1$. To this end we have
        \begin{align*}
            \sum_{A \in \left[b\right]^{ d}}w_A
            &= \sum_{A \in \left[b\right]^{ d}} \frac{\left<h_{d,b,A},f\right>}{\left\|h_{d,b,A}\right\|_2^2}\\            
            &= \sum_{A \in \left[b\right]^{ d}} \int_{[0,1]^d} b^d \1\left(x \in \Lambda_{d,b,A}\right)f(x) dx /b^d\\
            &=  \int_{[0,1]^d}  \sum_{A \in \left[b\right]^{ d}} \1\left(x \in \Lambda_{d,b,A}\right)f(x) dx \\
            &= \int_{[0,1]^d} f(x) dx\\
            &=1.
        \end{align*}
    \end{proof}
    
    \begin{cor}\label{cor:rankproj}
        Let $f_1,\ldots,f_d \in L^2([0,1])$ be probability densities, then 
            $$ \proj_{\sH_{d,b}} \prod_{i=1}^d f_i = \prod_{i=1}^d \proj_{\sH_{1,b}}f_i = \proj_{\sH_{d,b}^1} \prod_{i=1}^d f_i. $$
    \end{cor}
    \begin{proof}[Proof of Corollary \ref{cor:rankproj}]
        The first equality follows from Lemmas \ref{lem:projprod} and \ref{lem:spanproj}. To second equality follows from the first equality, the observation that $\sH_{d,b}^1 \subset \sH_{d,b}$, and that Lemma \ref{lem:spanproj} implies $\prod_{i=1}^d \proj_{\sH_{1,b}}f_i \in \sH_{d,b}^1$.
    \end{proof}
    
    \begin{thm} \label{thm:rank-one-l2-bias}
      Let $m_L=\sup_{f \in \lip_L}\left\|f\right\|_2$ and let $b^2 \ge L^2/12$ then, for any $f_1,\ldots,f_d \in \lip_L$, we have that
      $$
      \left \| \prod_{i=1}^d f_i - \proj_{\sH_{d,b}^1}\prod_{i=1}^d f_i \right\|_2^2 \le m_L^{2d}- \left(m_L^2 - \frac{L^2}{12b^2} \right)^d.
      $$
    \end{thm}

    \begin{proof}[Proof of Theorem \ref{thm:rank-one-l2-bias}]
      We have
      \begin{align}
        \left\| \prod_{i=1}^d f_i - \proj_{\sH_{d,b}^1} \prod_{i=1}^d  f_i\right\|_2^2
        &=\left\| \prod_{i=1}^d f_i - \proj_{\sH_{d,b}} \prod_{i=1}^d  f_i\right\|_2^2& \text{Corollary \ref{cor:rankproj}} \notag \\
        &= \left\| \prod_{i=1}^d f_i\right\|_2^2 - \left\| \proj_{\sH_{d,b}} \prod_{i=1}^d f_i \right\|_2^2& \text{ (Lemma \ref{lem:spanproj} with linear projection}\label{eqn:projtrick} \\
        &= \left\| \prod_{i=1}^d f_i\right\|_2^2 - \left\| \prod_{i=1}^d \proj_{\sH_{1,b}}f_i \right\|_2^2 & \text{Corollary \ref{cor:rankproj}} \notag \\
        &= \prod_{i=1}^d \left\|  f_i\right\|_2^2 -  \prod_{i=1}^d \left\| \proj_{\sH_{1,b}}f_i \right\|_2^2 & \text{Def'n of tensor product norm}.\notag
      \end{align}
      Noting that $\left\| \proj_{\sH_{1,b}}f_i \right\|_2^2 = \left\|f_i\right\|_2^2 - \left\|f_i - \proj_{\sH_{1,b}} f_i \right\|_2^2$ (rearrangement of the property used in \eqref{eqn:projtrick}) we have that
      \begin{equation}
        \left\| \prod_{i=1}^d f_i - \prod_{i=1}^d \proj_{\sH_{d,b}} f_i\right\|_2^2 = \prod_{i=1}^d\left\|f_i\right\|_2^2 - \prod_{i=1}^d\left(\left\| f_i\right\|_2^2 - \left\|f_i - \proj_{\sH_{1,b}}f_i \right\|_2^2 \right).\label{eqn:prodproj}
      \end{equation} 

      We will now turn our attention to the subtrahend on the right hand side of the previous equation.
      Note that the product terms $\left\| f_i\right\|_2^2 - \left\|f_i - \proj_{\sH_{1,b}}f_i \right\|_2^2 = \left\| \proj_{\sH_{1,b}}f_i \right\|_2^2>0$ so $\prod_{i=1}^d\left(\left\| f_i\right\|_2^2 - \left\|f_i - \proj_{\sH_{1,b}}f_i \right\|_2^2 \right)$ is a product of positive values. From Lemmas \ref{lem:1dhistproj} and \ref{lem:spanproj} we have that $\left\|f_i - \proj_{\sH_{1,b}}f_i \right\|_2^2 \le \frac{L^2}{12b^2}$. So we have
      $$
      \left\| f_i\right\|_2^2 - \left\|f_i - \proj_{\sH_{1,b}}f_i \right\|_2^2 \ge \left\| f_i\right\|_2^2 - \frac{L^2}{12b^2}.
      $$
      From H\"older's Inequality we have that 
      \begin{equation*}
        \left\|f \cdot \1 \right\|_1 \le \left\|f\right\|_2 \left\|\1\right\|_2 \Rightarrow \left\|f\right\|_2\ge 1.
      \end{equation*}
      Combining this with the hypothesis $\frac{L^2}{12b^2} \le 1$ we have that
      $$
      \left\| f_i\right\|_2^2 - \left\|f_i - \proj_{\sH_{1,b}}f_i \right\|_2^2 \ge \left\| f_i\right\|_2^2 - \frac{L^2}{12b^2} \ge0.
      $$
      For a pair of tuples $b_i\ge a_i\ge0$ it follows that $\prod b_i \ge \prod a_i$ and thus we have that 
      $$
      \prod_{i=1}^d\left(\left\| f_i\right\|_2^2 -  \left\|f_i - \proj_{\sH_{1,b}}f_i \right\|_2^2\right) \ge \prod_{i=1}^d\left( \left\| f_i\right\|_2^2 - \frac{L^2}{12b^2}\right).
      $$ 
      Plugging this back into the RHS of \eqref{eqn:prodproj} we get
      \begin{equation}
        \prod_{i=1}^d\left\|f_i\right\|_2^2 - \prod_{i=1}^d\left(\left\| f_i\right\|_2^2 - \left\|f_i - \proj_{\sH_{1,b}}f_i \right\|_2^2 \right) \le \prod_{i=1}^d\left\|f_i\right\|_2^2 - \prod_{i=1}^d\left(\left\| f_i\right\|_2^2 - \frac{L^2}{12b^2} \right). \label{eqn:lipapprox}
      \end{equation} 
      For some arbitrary $j$ we perform the following derivative
      \begin{align*}
        \frac{d}{d\left\|f_j\right\|_2^2 }\left( \prod_{i=1}^d\left\|f_i\right\|_2^2 - \prod_{i=1}^d\left(\left\| f_i\right\|_2^2 - \frac{L^2}{12b^2} \right)\right)
        &= \prod_{i\neq j} \left\|f_i\right\|_2^2 - \prod_{i\neq j} \left(\left\| f_i\right\|_2^2 - \frac{L^2}{12b^2} \right)\\
        &\ge 0.
      \end{align*}
      Thus we can find an upper bound to the RHS of \eqref{eqn:lipapprox} by maximizing $\|f_i\|_2^2$ over $f_i$, thus yielding
      \begin{equation*}
        \left \| \prod_{i=1}^d f_i - \proj_{\sH_{d,b}}\prod_{i=1}^d f_i \right\|_2^2 \le m_L^{2d}- \left(m_L^2 - \frac{L^2}{12b^2}  \right) ^d.
      \end{equation*}
    \end{proof}

    Calculating the value of $M_L$ is quite involved and is done in the next subsection. However we will first present the following result which gives a more tractable bound on bias.
    \begin{prop}
      \label{prop:l2projbnd}
      Let $L\geq 2$, $b^2\geq L^2/12$, and let $f_1,\ldots,f_d$ be elements of $\lip_L$. Then 
      \begin{equation*}
        \left \| \prod_{i=1}^d f_i - \proj_{\sH_{d,b}^1}\prod_{i=1}^d f_i \right\|_2^2 \le\frac{dL^{\frac{d+3}{2}}}{12b^2}\left[\frac{\sqrt{8}}{3}\right]^{d-1}.
      \end{equation*}
      If $0 \le L\leq 2$ instead of $L\ge 2$, then
      \begin{equation*}
        \left \| \prod_{i=1}^d f_i - \proj_{\sH_{d,b}^1}\prod_{i=1}^d f_i \right\|_2^2
        \leq d\frac{L^2}{12b^2} \exp\left (\frac{(d-1)L^2}{12}\right).
      \end{equation*}
    \end{prop}

    \begin{proof}[Proof of Proposition \ref{prop:l2projbnd}]
      For the $L\geq 2$ case, applying Theorem \ref{thm:rank-one-l2-bias} and Proposition \ref{prop:M_L} gives us
      \begin{equation*} 
        \left \| \prod_{i=1}^d f_i - \proj_{\sH_{d,b}^1}\prod_{i=1}^d f_i \right\|_2^2 
        \le \left[\frac{2\sqrt{2L}}{3}\right]^d-\left[\frac{2\sqrt{2L}}{3}-\frac{L^2}{12b^2}\right]^d.
      \end{equation*}
        We will use the identity $a^n-b^n=(a-b)\sum_{i=0}^{n-1}a^{i}b^{n-1-i}$ with $a= \frac{2\sqrt{2L}}{3}$ and $b=\frac{2\sqrt{2L}}{3}-\frac{L^2}{12b^2}$. Note that $a>b>0$. From $a^d-b^d=(a-b)\sum_{i=0}^{d-1}a^{i}b^{d-1-i}$ we have
      $$
      a^d-b^d=(a-b)\sum_{i=0}^{d-1}a^{i}b^{d-1-i}<(a-b)\sum_{i=0}^{d-1}a^{i}a^{d-1-i} = d(a-b)a^{d-1}
      $$
      and thus
      \begin{align*}
        \left[\frac{2\sqrt{2L}}{3}\right]^d-\left[\frac{2\sqrt{2L}}{3}-\frac{L^2}{12b^2}\right]^d
        &\leq d\frac{L^2}{12b^2}\left[\frac{2\sqrt{2L}}{3}\right]^{d-1}\\
        &=\frac{dL^{\frac{d+3}{2}}}{12b^2}\left[\frac{\sqrt{8}}{3}\right]^{d-1}.
      \end{align*}
      The other case proceeds analogously. If $L=0$ the proposition statement is trivial, so we can assume $L>0$ and thus $a>b>0$ as before. Now we can finish with
      \begin{align*}
        \left \| \prod_{i=1}^d f_i - \proj_{\sH_{d,b}^1}\prod_{i=1}^d f_i \right\|_2^2 
        &\le \left[\frac{L^2}{12}+1\right]^d-\left[\frac{L^2}{12}+1-\frac{L^2}{12b^2}\right]^d \\
        &\le d\frac{L^2}{12b^2} \left[\frac{L^2}{12}+1\right]^{d-1}\\
        &\leq d\frac{L^2}{12b^2} \exp\left (\frac{(d-1)L^2}{12}\right).
      \end{align*}
    \end{proof}

    We find that the rate of the convergence of this bias term doesn't depend on the Lipschitz constants nor dimension and always has order $O\left(b^{-2}\right)$. Applying a square root and applying H\"older's Inequality as we did in \eqref{eqn:holders} we find that the $L^1$ bias convergence rate is $O(b^{-1})$, regardless of dimension.

    \subsubsection{Maximizing the $L^2$ norm of a Lipschitz density}

    In this section, we compute the value of $m_L$, the largest possible $L^2$ norm of an $L$-Lipschitz density function on $[0,1]$. We will show that if $L\leq 2$, $m_L^2=L^2/12+1$, and if $L\geq 2$, $m_L^2=2\sqrt{2L}/3$ (i.e. Proposition \ref{prop:l2bias}). This will follow directly from the more general Proposition~\ref{prop:M_L} below.

    We will need several lemmas, starting with a discretization of the problem.

    For the following results we are going to use a notion of \emph{monotonic rearrangement}. For a finite sequence of real numbers $(y_1,\ldots,y_n)$ its rearrangement $R(y) = \widetilde{y} = (\widetilde{y}_1,\ldots,\widetilde{y}_m)$ which is a reordering of $y$ so that it is increasing, it essentially sorts $y$ to be ascending. Interestingly this concept can also be applied to functions \cite{Kawohl1985}. It is usually defined so that the function is decreasing for the functional case. Interestingly if one monotonically rearranges a Lipschitz continuous function, $f:\rn^+ \to \rn$ for example, then its rearrangement will also be Lipschitz continuous with a Lipschitz constant no larger than that for $f$ (and potentially smaller).
    The following lemma says that the discrete equivalent of the ``Lipschitz constant'' of a sequence is always higher than that of its monotone reordering. 

    \begin{lem}
      \label{lem:monotonize}
      Let $y=(y_1,\ldots,y_n)$ be a sequence of real numbers, and let $\widetilde{y}=(\widetilde{y}_1,\ldots,\widetilde{y}_n)$ be its monotone reordering so that $\widetilde{y}_1\leq \widetilde{y}_2\leq\cdots\leq \widetilde{y}_n$. We have 
      $$\max_{i\leq n-1} |y_{i+1}-y_i|\geq \max_{i\leq n-1} \widetilde{y}_{i+1}-\widetilde{y}_i.$$
    \end{lem}
    \begin{proof}[Proof of Lemma \ref{lem:monotonize}]
      If $y$ is constant or $n=1$ then we are done, so we will assume $y$ is nonconstant. Let $i^*$ be such that $\widetilde{y}_{i^*+1}-\widetilde{y}_{i^*}=\max_{i\leq n-1} \widetilde{y}_{i+1}-\widetilde{y}_i$. Define the sets 
      \begin{align*}
        A:&=\{i\in \{1,2,\ldots,n\}: y_i\leq \widetilde{y}_{i^*}\}\\
        B:&=\{i\in \{1,2,\ldots,n\}: y_i\geq \widetilde{y}_{i^{*}+1}\}.
      \end{align*}
      These sets partition $[n]$ so $A\sqcup B=\{1,2,\ldots,n\}$ and $A$ and $B$ are non empty. Thus there must exist $j\in [n]$ such that $j\in A$ and $j+1\in B$ or vice versa. For the first case we have that
      $$
      y_{j+1} - y_j \ge \widetilde{y}_{i^*+1}-\widetilde{y}_{i^*} = \max_{i\leq n-1} \widetilde{y}_{i+1}-\widetilde{y}_i
      $$
      and for the vice versa case we have
      $$
      \left|y_{j+1} - y_j\right|=y_j- y_{j+1} \ge \widetilde{y}_{i^*+1}-\widetilde{y}_{i^*} = \max_{i\leq n-1} \widetilde{y}_{i+1}-\widetilde{y}_i.
      $$
      So we have demonstrated that there exits a $j$ such that $ \left|y_{j+1} - y_j\right| \ge \max_{i\leq n-1} \widetilde{y}_{i+1}-\widetilde{y}_i$.
    \end{proof}

    We are now in a position to prove the following discrete version of Proposition~\ref{prop:M_L} below. 

    \begin{prop}
      \label{prop:discrete}
      Let $n\in \mathbb{N}_{\ge 2}$, $U>0,L>0$, and let $S_n$ denote the set of sequences $y=(y_1,\ldots,y_n)$ such that the following conditions are satisfied: 
      \begin{align} 
        y_i & \geq 0 \quad \quad (\forall i) \label{eqn:discrete-nonneg} \\
        \left|y_{i+1}-y_i\right|&\leq L \quad \quad (\forall   1\leq i\leq n-1 ) \label{eqn:discrete-lipschitz}  \\
        \sum_{i=1}^n y_i&=U. \label{eqn:discrete-sum}
      \end{align}

      A sequence $y^{*}=(y_1,\ldots,y_n)\in S_n$, which maximizes the quantity $\sum_{i=1}^n y_i^2$ subject to the conditions above is given by the following formulae. 

      If $L\le \frac{2U}{n(n-1)}$, for all $i$, 
      \begin{align}
        \label{eqn:discretesmallL}
        y^*_i=(i-1)L+\left[U/n-(n-1)L/2\right].
      \end{align}
      If $L\ge \frac{2U}{n(n-1)}$, for all $i$, 
      \begin{align}
        y^*_i=\left((i-n+m)L+U/m-L(m+1)/2\right)_+,
        \label{eqn:discretelargeL}
      \end{align}
      where $x_+$ denotes the positive part of $x$ and $m$ is the smallest integer such that $m(m+1)L\geq 2U$.
    \end{prop}

    \begin{proof}[Proof of Proposition \ref{prop:discrete}]
      First note that $S_n$ is compact since the intersection of the closed sets denoted by \eqref{eqn:discrete-nonneg} and \eqref{eqn:discrete-sum} is compact and the set denoted by \eqref{eqn:discrete-lipschitz} is closed and thus further intersecting with that set gives us compact set. Because $y\mapsto \sum_{i=1}^n y_i^2 = \left\|y\right\|_2^2$ is continuous it must therefore attain a maximum on $S_n$. Let $\widetilde{\why}\in S_n$ be some sequence which attains the maximum. Define the sequence $\why$ as a monotone increasing (non-decreasing) rearrangement of $\widetilde{\why}$. By Lemma~\ref{lem:monotonize}, we have  $\why\in S_n$ and clearly $\left\|\why\right\|_2^2 = \left\|\widetilde{\why}\right\|_2^2$ so $\why$ is also a maximizer.

      Consider the following class of sequences. 
      
      \textit{$\fL$-sequences:} A sequence $x= \left(x_1,\ldots,x_n\right)$ is an $\fL$-sequence if
      \begin{enumerate}[nosep]
        \item $x \in S_n$.
        \item $x$ is non-decreasing. 
        \item For all $i < n$, $x_i=0$ or $x_{i+1}-x_i=L$.
      \end{enumerate}
      We are first going to show that $y^*$ is the only element in this class.
      
      Consider some arbitrary $\fL$-sequence $y$ be and let $k=\min(i: y_i\neq 0)$. We have $y_i-y_k=L(i-k)$ for all $i\geq k$ and therefore
      \begin{align}
        U
        &=\sum_{i=1}^n y_i\notag \\
        &= \sum_{i=k}^n y_k + L\sum_{i=k+1}^n (i-k) \label{eqn:badsum}\\
        \Rightarrow U&=(n-k+1)y_k+ \frac{L(n-k)(n-k+1)}{2}. \label{eqn:intermediate}
      \end{align} 
      One can check that \eqref{eqn:intermediate} holds when $k=n$, even though the summation in \eqref{eqn:badsum} is ill-posed. We will now split into two cases, corresponding to \eqref{eqn:discretesmallL} and \eqref{eqn:discretelargeL}, to show that there is only one $\fL$-sequence for fixed $U,L,n$ and that this sequence is given by $y^*$. 
      
      \textbf{Case 1: $\frac{Ln(n-1)}{2} <  U$.}\newline
     If $\frac{Ln(n-1)}{2}< U$ then \eqref{eqn:intermediate} cannot hold unless $k=1$ because if $k\ge 2$ then $y_k\leq L$ (since $y_{k-1}=0$) and 
     \begin{align*}
        U 
        &= (n-k+1)y_k+ \frac{L(n-k)(n-k+1)}{2}\\
        &\leq (n-1)L+ \frac{L(n-2)(n-1)}{2}\\
        &=\frac{Ln(n-1)}{2}< U,
     \end{align*}
     a contradiction.
     
     Solving for $y_k$ in \eqref{eqn:intermediate} with $k=1$ we get 
               \begin{align*}
                y_k=y_1=\frac{U}{n}-\frac{L(n-1)}{2},
              \end{align*}
      so $k$ and $y_k$ are both unique, thus $S_n$ contains only one element. We see that this coincides with the expression for $y^*$ in \eqref{eqn:discretesmallL} by noting that the derived $y_1$ is equal to $y^*_1$. The situation where $L= \frac{2U}{n(n-1)}$ in \eqref{eqn:discretesmallL} will be addressed at the end of Case 2.
        
     \textbf{Case 2: $\frac{Ln(n-1)}{2}\ge U$.}\newline
     For this case we need $k\geq 2$ since $y_k>0$ and letting $k=1$ in \eqref{eqn:intermediate} yields
     \begin{align*}
         U 
         &= (n-k+1)y_k+ \frac{L(n-k)(n-k+1)}{2}\\
         &> \frac{Ln(n-1)}{2} \ge U
     \end{align*}
     a contradiction.
     This implies $y_{k-1}=0$, so $ y_k\leq L$. Define $m \triangleq n-k+1$. From \eqref{eqn:intermediate} we have
     \begin{align}
         U
         &= (n-k+1)y_k+ \frac{L(n-k)(n-k+1)}{2} \notag\\
         &= m y_k+ \frac{L(m-1)m}{2}\label{eqn:case2}\\
         &= m\left(y_k+\frac{L(m-1)}{2}\right). \notag
     \end{align}
     
     Since $y_k > 0$ we have $U \ge \frac{Lm(m-1)}{2}$. Additionally since $y_k \le L$,
     \begin{align*}
         U
         &\le m\left(L+\frac{L(m-1)}{2}\right)\\
         &= \frac{Lm(m+1)}{2}.
     \end{align*}
     It now follows that $m$  must be the smallest integer such that $\frac{Lm(m+1)}{2}\ge U$.
     
     Solving for $y_k$ in \eqref{eqn:case2} we have
              \begin{align*}
                y_k
                =\frac{U}{m}-\frac{L(m-1)}{2},
              \end{align*}
    so both $k$ and $y_k$ are unique so there is only one element in $S_n$. Further note that
              \begin{align*}
                y_k
                &=\frac{U}{m}-\frac{L(m-1)}{2}\\
                &=\frac{U}{m}-\frac{L(m+1)}{2}+L.
              \end{align*}
     From the definition of $S_n$ it follows that 
     \begin{align*}
         y_i 
         &= \left( L (i-k) + y_k \right)_+\\
         &= \left( L (i-n+m-1) + \frac{U}{m}-\frac{L(m+1)}{2}+L \right)_+\\
         &= \left( L (i-n+m) + \frac{U}{m}-\frac{L(m+1)}{2} \right)_+
     \end{align*}
      which coincides with $y^*$  from \eqref{eqn:discretelargeL}. 
      
      To finish this case we will show that when $\frac{Ln(n-1)}{2}= U$ then \eqref{eqn:discretesmallL} equals \eqref{eqn:discretelargeL}. To see this first note that $\frac{Ln(n-1)}{2}= U$ implies $m = n-1$. Therefore \eqref{eqn:discretelargeL} equals
      \begin{align*}
          y^*_i
          &= \left((i-n+m)L+\frac{U}{m}-\frac{L(m+1)}{2}\right)_+\\
          &= \left((i-1)L+\frac{U}{n-1}-\frac{Ln}{2}\right)_+\\
          &= \left((i-1)L+\frac{Ln}{2}-\frac{Ln}{2}\right)_+\\
          &= (i-1)L
      \end{align*}
     and accordingly \eqref{eqn:discretesmallL} equals
     \begin{align*}
         y^*_i
         &=(i-1)L+\left[\frac{U}{n}-\frac{(n-1)L}{2}\right] \\
         &=(i-1)L+\left[\frac{(n-1)L}{2}-\frac{(n-1)L}{2}\right] \\
         &= (i-1)L.
     \end{align*}
     This finishes Case 2. 
        
        We now have that $y^*$ satisfies property 1 of $\fL$-sequences (in particular \eqref{eqn:discrete-sum} was nontrivial) as well as properties 2 and 3.
        This concludes the proof that $y^*$ is the only $\fL$-sequence.

      Now we are going to show that $\why$ is also a an $\fL$-sequence which will complete this proof, since $y^*$ is the only $\fL$-sequence. For sake of contradiction assume that $\why$ is not an $\fL$-sequence. From this there must exist $i_*\in \{1,2,\ldots,n-1\}$ such that $\why_{i_*+1}-\why_{i_*}<L$ and $\why_{i_*}\neq 0$. We then define a sequence $t$ by:
      \begin{equation}
        t_{i}=
        \begin{cases}\label{sol}
          \why_i                 &\text{if} \quad i\neq i_*,i_*+1\notag\\
          \why_{i_*}-\delta      &\text{if} \quad i= i_*\notag \\
          \why_{i_*+1}+\delta    &\text{if} \quad i= i_*+1,
        \end{cases}
      \end{equation}
      where $\delta=\min\left(\frac{L-\why_{i_*+1}+\why_{i_*}}{2}, \why_{i_*}  \right)$. Note that $0< \delta \le L/2$, a fact which we will be using extensively. 

      We will now show that $t\in S_n$.

      \paragraph{$t$ satisfies \eqref{eqn:discrete-nonneg}:}
      Note that by construction, since $\delta\leq \why_{i_*}$ and $t_{i_*}\geq 0$ (and clearly, $t_{i}\geq 0$ for $i\neq i_*$).

      \paragraph{$t$ satisfies \eqref{eqn:discrete-lipschitz}:}
      The sequences $t$ and $\why$ differ on only two indices so \eqref{eqn:discrete-lipschitz} holds trivially for any pair of indices not involving $i_*$ or $i_*+1$. We will address the remaining cases:

      \textbf{Case 1 $\left|t_{i_*+1 }-t_{i_*}\right|$:}  We now have that
      \begin{align*}
        t_{i_*+1}-t_{i_*}
        =& \why_{i_*+1}-\why_{i_*}+2\delta\\
        \leq& \why_{i_*+1}-\why_{i_*}+L-\why_{i_*+1}+\why_{i_*}\\
        =&L.
      \end{align*}
      Since $\why$ is monotonic and $\delta >0 $ we further have that $\why_{i_*+1}-\why_{i_*}+2\delta>0$ thereby finishing this case.

      \textbf{Case 2 $\left|t_{i_*+2 }-t_{i_*+1}\right|$:}
      We have that
      \begin{align*}
        t_{i_*+2}-t_{i_*+1}
        =&\why_{i_*+2}-\why_{i_*+1}-\delta \\
        \geq& -\delta\\
        \geq&  -L/2
      \end{align*}
      and 
      \begin{equation*}
        t_{i_*+2}-t_{i_*+1}
        = \left(\why_{i_*+2}-\why_{i_*+1}\right)- \delta
        \leq L
      \end{equation*}
      thereby completing this case.

      \textbf{Case 3: $\left|t_{i_* }-t_{i_*-1}\right|$:}
      This is virtually identical to the previous case so we omit it.

      Thus, we have that $t$ satisfies condition \eqref{eqn:discrete-lipschitz}.

      \paragraph{$t$ satisfies \eqref{eqn:discrete-sum}:} finally, it is clear that $t_{i_*}+t_{i_*+1}= \why_{i_*}+\why_{i_*+1}$ and thus $\sum_{i} t_i=\sum_{i} \zeta_i=U$.

      Hence, $t$ satisfies \eqref{eqn:discrete-sum} and we have proved the claim that $t\in S_n$.

      Now, writing $\emi$ for $\frac{\why_{i_*+1}+\why_{i_*}}{2}$ and $\Delta$ for $\frac{\why_{i_*+1}-\why_{i_*}}{2}$:
      \begin{align*}
        \sum_{i=1}^n (t_i^2- \why_{i}^2)&= t_{i_*+1}^2+t_{i_*}^2-\why_{i_*+1}^2-\why_{i_*}^2 \\
        & = (\why_{i_*+1}+\delta)^2+(\why_{i_*}-\delta)^2-\why_{i_*+1}^2-\why_{i_*}^2  \\
        & = (\emi+\Delta+\delta)^2 + (\emi-\Delta-\delta)^2 -(\emi+\Delta)^2 - (\emi-\Delta)^2 \\
        & = 2\emi^2+2 (\Delta+\delta)^2 -[2\emi^2+2 \Delta^2]= 4\Delta\delta +2\delta^2 \geq 2\delta^2 >0,
      \end{align*}
      where at the last line, we have used the fact that since by construction, since $\zeta \neq y^*$, we have $\delta>0$. This shows that: 
      \begin{equation*}
        \sum_{i=1}^n t_i^2 > \sum_{i=1}^n \why_i^2,
      \end{equation*}
      a contradiction. Hence we conclude that $\zeta$ must equal $y^*$. 
    \end{proof}
    
    We can now proceed with the statement and proof of the continuous case.
    \begin{prop}
      \label{prop:M_L}
      Let $L,U>0$ be given and let $S$ be the set of functions $f:[0,1]\rightarrow \mathbb{R}^+$ such that the following conditions are satisfied: 
      \begin{enumerate}[nosep]
        \item $f$ is Lipschitz continuous with Lipschitz constant $L$.
        \item $\int_{0}^1 f(x)dx=U$.
      \end{enumerate}
      Let $f_*^{U,L}$ be defined as follows
      \begin{alignat}{2}\label{eqn:definef*small}
        f_*^{U,L}(x)&=\left(x-\frac{1}{2}\right)L+U  &\quad \text{if } L\leq 2U\\                
        f_*^{U,L}(x)&=L\left(x-1+\sqrt{\frac{2U}{L}}\right)_+  &\text{if } L> 2U.        \label{eqn:definef*large}
      \end{alignat}
      Then $f_*^{U,L} \in \arg \max_{f \in S} \int_{0}^1f^2(x)dx$ and 
      \begin{equation*}
        \label{eqn:explicitform}
         \max_{f \in S} \int_{0}^1 f^2(x)dx = M(U,L) \triangleq \frac{W^3L^2}{12}+ \frac{U^2}{W},
      \end{equation*}
      where $W=1$ if $L\leq 2U$ and $W=\sqrt{\frac{2U}{L}}$ if $L> 2U$.
     Equivalently, $M=\frac{L^2}{12}+U^2$ if $L\leq 2U$ and $M=\frac{2\sqrt{2}}{3}\sqrt{L}U^{3/2}$ if $L\geq 2U$.
    \end{prop}
    
    \begin{remark*}[Proposition \ref{prop:M_L} applied to probability densities]\normalfont
        The highly general formation of Proposition \ref{prop:M_L} is useful for its proof. However in density estimation only the case where $f_*$ is a pdf and $U=1$ is particularly meaningful. This gives
        \begin{align*}
            f_*(x)&=\left(x-\frac{1}{2}\right)L        &     \left\|f_*\right\|_2^2 &= \frac{L^2}{12}+1 & & \text{if } L\leq 2\\                
            f_*(x)&=L\left(x-1+\sqrt{\frac{2}{L}}\right)_+ &  \left\|f_*\right\|_2^2 &= \frac{2\sqrt{2}}{3}\sqrt{L}   &   &\text{if } L> 2.
      \end{align*}
    \end{remark*}
    \begin{proof}[Proof of Proposition \ref{prop:M_L}]
      We first introduce some notation. For $n\in \nn$ and  $U',L'\ge 0$ we will denote by $\sG^n \left(U',L'\right)$ the nonnegative sequence $(y_1,y_2,\ldots,y_n)$ maximizing $\sum_{i=1}^n y_i^2$ subject to $\sum_{i=1}^n y_i\frac{1}{n}=U'$ and $|y_{i+1}-y_i|\leq \frac{L'}{n}$ (i.e. $\sG^{n}\left(U',L'\right)=y^*$ from Proposition~\ref{prop:discrete} with $U\leftarrow U'n$ and $L\leftarrow \frac{L'}{n}$). 
      We similarly write $M_n\left(U',L'\right)$  for the optimal value, i.e., $M_n\left(U',L'\right)=\sum_{i=1}^n (\sG^n \left(U',L'\right)_i)^2 \frac{1}{n}$.
      For ease of notation we set $\sG^n \triangleq \sG^n \left(U',L'\right)$, $M_n \triangleq M_n\left(U',L'\right)$, and $f_* = f_*^{U',L'} $. Later we will set $U'\leftarrow U$ and $L'\leftarrow L$ to prove the proposition statement, but it will be useful to establish some results for general $U'$ and $L'$. We use the prime symbol to help avoid confusion between proving the final proposition statement and proving supporting results. Unless otherwise specified, all limits are taken as $n \to \infty$.
      
      Let $\tsG^n$ be the piecewise linear function from $[0,1]$ to $\mathbb{R}^+$  with $\tsG^n(i/n)=\sG^n_i$ for $i\in [n]$ and $\tsG^n(0)=\sG^n_1$.
      \begin{lem}\label{lem:claim1}
        The following limits hold,
        \begin{gather*}
            \int_0^1 \tsG^n(x)dx \rightarrow U' \\
            \int_0^1\tsG^n(x)^2dx-M_n\rightarrow 0.
        \end{gather*}
     
      \end{lem}

      \begin{proof}[Proof of Lemma \ref{lem:claim1}]
      Note that the function $\tsG^n$ is $L'$-Lipschitz. The Lipschitz continuity also implies the following bounds on the difference between the Riemann sums below and their corresponding integrals:
      \begin{align*}
        \left| \int_0^1\tsG^n(x)dx-U'\right|
        &= \left| \int_0^1\tsG^n(x)dx-\sum_{i=1}^n\sG^n_i\frac{1}{n} \right|\\
        &\le \sum_{i=1}^n \left| \int_{(i-1)/n}^{i/n}\tsG^n(x)dx -\sG^n_i\frac{1}{n} \right|\\
        &= \sum_{i=1}^n \left| \int_{(i-1)/n}^{i/n}\tsG^n(x) -\sG^n_i dx  \right|\\
        &\le \sum_{i=1}^n  \int_{(i-1)/n}^{i/n}\left| \tsG^n(x) -\sG^n_i\right| dx  \\
        &\le \sum_{i=1}^n  \int_{(i-1)/n}^{i/n}\frac{L'}{n} dx  \\
        &= n\frac{L'}{n^2}\rightarrow 0,
      \end{align*}
      and also 
      \begin{align*}
        \left|\int_0^1\tsG^n(x)^2dx- M_n\right|
        &= \left|\int_0^1\tsG^n(x)^2dx- \sum_{i=1}^n(\sG^n_i)^2\frac{1}{n}\right| \\
        &=\left|\sum_{i=1}^n\left[\int_{(i-1)/n}^{i/n}\tsG^n(x)^2-(\sG^n_i)^2dx\right]\right|\\
        &=\left|\sum_{i=1}^n\left[\int_{(i-1)/n}^{i/n}\left[\tsG^n(x)-(\sG^n_i)\right]\left[\tsG^n(x)+(\sG^n_i)\right]dx\right]\right|\\
        &\leq n\frac{1}{n}\frac{L'}{n}\max_{i}\left[2\sG^n_i\right]\\
        &\leq \frac{L'}{n} \left[2\left(U'+L'\right)\right]\rightarrow 0,
      \end{align*}
      where at the last line we have used the fact that $\sG^n_i\leq U'+L'$ for all $i$. This is because $\sum_{i=1}^n \sG^n_i=U'n$ implies there exists an $i_*$ such that $\sG^n_{i_*}\le U'$ and the Lipschitz condition further implies (for all $i$): 
      \begin{equation}
        \sG^n_i\leq \sG^n_{i_*}+\frac{L'}{n}|i-i_*|\leq U'+L'. \label{eqn:UL-bound}
      \end{equation}
      \end{proof}
      
      \begin{lem}\label{lem:claim2}
         The sequence of functions $\tsG^n\rightarrow f_*$ pointwise.
      \end{lem}

      \begin{proof}[Proof of Lemma \ref{lem:claim2}]
       Given a continuous function $f:[0,1]\to \rn$, that is also differentiable on $(0,1)$, for all $x\in [0,1]$ the fundamental theorem of calculus implies
      \begin{align*}
          &\int_0^x \partial_y f(y) dy = f(x)-f(0)\\
          \Rightarrow &f(x) =  \int_0^x \partial_y f(y) dy + f(0)\\
          \Rightarrow& \left|f(x)\right| \le \int_0^x \left|\partial_y f(y)\right| dy + \left|f(0)\right| \le \int_0^1 \left|\partial_y f(y)\right| dy + \left|f(0)\right|.
      \end{align*}
      Suppose $f$ was not differentiable, but was still continuous, at some  $c \in (0,x)$. We would still have
      \begin{align*}
          f(x)
          &= f(x) - f(c) +f(c) \\
          &= \int_c^x \partial_y f(y) dy  + \int_0^c \partial_y f(y) dy + f(0)\\
          \Rightarrow \left|f(x)\right| 
          &\le \int_0^1 \left|\partial_y f(y)\right| dy + \left|f(0)\right|.
      \end{align*}
      Both $f_*$ and $\tsG^n$ are piecewise linear. They are therefore continuous and their derivatives exist on all but a finite set of points, so for all $x\leq 1$,
      \begin{align}
        \label{eq:simplification?}
        |f_*(x)-\tsG^n(x)|\leq \left| f_*(0)-\tsG^n(0)\right|+ \int_0^1|\partial_yf_*(y)-\partial_y\tsG^n(y)|dy.
      \end{align}

      We now go on with the proof of Lemma \ref{lem:claim2} in two cases. 

      \textit{Case 1: $L'\le  2U'$}
      
      Note that if $L'\le 2U'$, then for all $n\ge 2$ we have that $\frac{L'}{n}\leq \frac{2U'n}{n(n-1)}$ and thus $\sG^n$ is defined by \eqref{eqn:discretesmallL}. Then for any $n$ we have $\tsG^n(i/n)-\tsG^n((i-1)/n)=L'/n$ for all $i\geq 2$, and $\tsG^n(1/n)-\tsG^n(0/n)=0$. In particular, we have under these conditions that $\partial_x\tsG^n(x)=L'$ for all $x> 1/n$ and $\partial_x\tsG^n(x)=0$ for $x < 1/n$. Hence by \eqref{eq:simplification?} we have for all $x\le 1$,
      \begin{align*}
        |f_*(x)-\tsG^n(x)|
        &\leq \left| f_*(0)-\tsG^n(0)\right|+ \int_0^1|\partial_yf_*(y)-\partial_y\tsG^n(y)|dy\\
        &\leq \left|U'-\frac{1}{2}L'-\left[U'-\frac{n-1}{n}\frac{L'}{2}\right]\right|+\int_0^{1/n}|\partial_yf_*(y)-\partial_y\tsG^n(y)|dy\\
        &\leq \frac{L'}{2n}+\frac{L'}{n}\rightarrow 0.
      \end{align*}

      \textit{Case 2: $L'> 2U'$}
      
      Since $L' > 2U'$, for sufficiently large $n$ we have that  $\frac{L'}{n} > \frac{2U'n}{n(n-1)}$ and $\sG^n$ is defined using \eqref{eqn:discretelargeL}. Since we are interested the limit as $n\to \infty$ we will proceed using \eqref{eqn:discretelargeL}. As in the proof of Proposition~\ref{prop:discrete} we let $k$ be the smallest natural number with $\sG^n_k>0$, noting also that $k\ge 2$ (see Case 2 in that proof). 
      
      For all $i\leq k-1$ we have $\tsG^n(i/n)-\tsG^n((i-1)/n)=0$ so $\partial_x\tsG^n(x) = 0$ for all $x<(k-1)/n$. Similarly for all $i\geq k+1$ we have $\tsG^n(i/n)-\tsG^n((i-1)/n)=L'/n$ so $\partial_x\tsG^n(x) = L'$ for all $x> (k+1)/n$. We also have that $\tsG^n(0) = 0 $.

      By definition~\eqref{eqn:definef*large} we have $\partial_x f_*(x)=0$ for all $x< \left [1-\sqrt{\frac{2U'}{L'}}\right]$ and $\partial_x f_*(x)=L'$ for all $x> 1-\sqrt{\frac{2U'}{L'}}$. We also have $f_*(0)=0$.

      Again from the proof of Proposition~\ref{prop:discrete} we know $k=n-m+1$, where $m$ is the smallest integer such that $m(m+1)L'/n\geq 2U'n$. It follows that $m(m+1)/n^2 \to 2U'/L'$. Since $m \to \infty$ and $m^2/n^2$ doesn't diverge we have that $m/n^2 \to 0$ and thus $m/n\rightarrow \sqrt{2U'/L'}$. Substituting in $k = n-m+1$ gives $k/n\rightarrow 1-\sqrt{2U'/L'}$.
      
      Let $\varepsilon > 0$. For sufficiently large $n$ both $(k+1)/n$ and $(k-1)/n$ lie in $ \left[1-\sqrt{2U'/L'}-\varepsilon, 1-\sqrt{2U'/L'}+ \varepsilon \right]$. Letting $x\le 1$, for large enough $n$ the following holds 
      by \eqref{eq:simplification?}
      \begin{align*}
        |f_*(x)-\tsG^n(x)|&\leq \left| f_*(0)-\tsG^n(0)\right|+ \int_0^1|\partial_yf_*(y)-\partial_y\tsG^n(y)|dy\\
        &=\left| f_*(0)-\tsG^n(0)\right|+\int_{1-\sqrt{2U'/L'}-\varepsilon}^{1-\sqrt{2U'/L'}+ \varepsilon}|\partial_yf_*(y)-\partial_y\tsG^n(y)|dy\\
        &\le 2 \varepsilon \left(\max_y |\partial_yf_*(y)|+ |\partial_y\tsG^n(y)| \right)\\
        &\le 2 \varepsilon \left(L' + L' + U'\right) \quad \quad \quad  \text{using \eqref{eqn:UL-bound}}.
      \end{align*}
      Since this holds for all $\varepsilon>0$, $|f_*(x)-\tsG^n(x)|$ goes to zero, and we have pointwise convergence for Case 2.
      \end{proof}
      
        \begin{lem}\label{lem:claim3}
            We have that $M_n\rightarrow \int_0^1 f_*(x)^2 dx$.
        \end{lem}
      \begin{proof}[Proof of Lemma \ref{lem:claim3}]
        From \eqref{eqn:UL-bound} we can bound $\max_x \tsG^n\left(x\right)^2 \le \left(U'+L'\right)^2$. This bound, along with the pointwise convergence of Lemma \ref{lem:claim2}, allows us to apply the dominated convergence theorem, thus giving $\int_0^1 \tsG^n(x)^2dx\rightarrow \int_0^1 f_*(x)^2dx$. Simply applying Lemma \ref{lem:claim1} finishes our proof.
      \end{proof}

      We can now proceed with the proof of the proposition. First we will calculate the integral $\int_0^1 f_*(x)^2dx$, thus establishing that $M(U',L')= \int_0^1 f_*(x)^2dx$. To avoid confusion and to be fully precise: $M$ is defined by the $\triangleq$ symbol in the proposition statement; the left equality in that equation being proven. 
      
      We define $W'$ from $U'$ and $L'$ analogously to the way $W$ is defined from $U$ and $L$ in the proposition statement. Note that $f_*$ is linear on the interval $\left[1-W',1\right]$ and zero elsewhere. From Lemmas \ref{lem:claim1} and \ref{lem:claim2} and using the dominated convergence theorem as before we have that $\int_0^1 f_*(x)dx = U'$. Thus we can write  
      \begin{align*}
        M\left(U',L'\right)
        &=\int_0^1 [f_*(x)]^2dx\\
        &=\int_{1-W'}^1 [f_*(x)]^2dx\\
        &=\int_{1-W'}^1 \left[f_*(x)-\frac{U'}{W'}+\frac{U'}{W'}\right]^2dx\\
        &=\int_{1-W'}^1 \left[f_*(x)-\frac{U'}{W'}\right]^2dx+\int_{1-W'}^1 \left[\frac{U'}{W'}\right]^2dx+2\int_{1-W'}^1 \left[f_*(x)-\frac{U'}{W'}\right]\left[\frac{U'}{W'}\right]dx\\
        &=\int_{1-W'}^1 \left[f_*(x)-\frac{U'}{W'}\right]^2dx+\int_{1-W'}^1 \left[\frac{U'}{W'}\right]^2dx,
      \end{align*}
      where at the last line we have used the fact that $\int_{1-W'}^1f_*(x)-\frac{U'}{W'}dx=0$.

      Clearly $\int_{1-W'}^1 [\frac{U'}{W'}]^2dx=\frac{U'^2}{W'}$. Further, $f_*|_{[1-W',1]}-\frac{U'}{W'}$ is linear with slope $L'$ and $f_*\left(1- W'/2\right)-\frac{U'}{W'} = 0 $, so $f_*|_{[1-W',1]}(x)-\frac{U'}{W'}$ is antisymmetric about $x=1-\frac{W'}{2}$. Using these facts we can continue, 
      \begin{align*}
        \int_0^1 [f_*(x)]^2dx&=\int_{1-W'}^1 [f_*(x)]^2dx\\
        &=\frac{U'^2}{W'}+2\int_{1-\frac{W'}{2}}^1 \left[f_*(x)-\frac{U'}{W'}\right]^2dx\\
        &=\frac{U'^2}{W'} +2\int_{1-\frac{W'}{2}}^1 L'^2\left(x-1+\frac{W'}{2}\right)^2dx\\
        &=\frac{U'^2}{W'} +2L'^2\int_0^{\frac{W'}{2}} x^2dx\\
        &=\frac{U'^2}{W'} +2 L'^2 \frac{(W'/2)^3}{3}= \frac{W'^3L'^2}{12}+\frac{U'^2}{W'},
      \end{align*}
      as expected.
      
      We now fix $U' \leftarrow U$ and $L' \leftarrow L$. We have shown $\int_0^1 f_*(x)dx = U$ which, along direct inspection of the definition of $f_*$, establishes that $f_*$ satisfies the properties of $S$. Let $g\in S$ be arbitrary. We will show that  $\int_{0}^1 g(x)^2dx\leq  \int_0^1 f_*(x)^2dx$, which demonstrates that $f_* \in \arg \max_{f \in S} \int_{0}^1 f^2(x)dx$ and finishes our proof.

      For all $n\ge 2$, define the sequence $g^{n}=(g^{n}_1,\ldots,g^{n}_n)$ with $g^{n}_i=g(\frac{i}{n})$. From the Lipschitz continuity of $g$ we have as $n\rightarrow \infty$,
      \begin{align}
        \label{phantomdissum}
        \sum_{i=1}^n g^n_i\frac{1}{n}\rightarrow \int_0^1 g(x)dx =U,
      \end{align} 
      and 
      \begin{align}
        \label{phantomdissumsq}
        \sum_{i=1}^n \left(g^n_i\right)^2\frac{1}{n}\rightarrow \int_0^1 g(x)^2dx.
      \end{align}

      Let $\varepsilon>0$ be arbitrary. By \eqref{phantomdissum} we can set $N$ such that  $\sum_{i=1}^n g^n_i\frac{1}{n} \leq U+\varepsilon$  for all $n\geq N$. Note that $\left| g^n_{i+i} -g^n_i \right| \le L/n$ and $g^n_i \ge 0$ for all $i$.
      Let $Q_n \triangleq  \frac{1}{n}\left(n\varepsilon + nU - \sum_{i=1}^n g^{n}_i\right)$. For $n\ge N$ we have that $Q_n \ge 0$ and the following hold,
      \begin{alignat*}{2}
          g^{n}_i+Q_n &\ge 0 \quad &\text{ for all $i$}\\
          \left|\left(g^{n}_i+Q_n\right) - \left(g^{n}_{i+1}+Q_n\right)\right| &\le L/n &\text{ for all $i$}\\
          \frac{1}{n}\left(\sum_{i=1}^n g^{n}_i+Q_n\right) &= U+\varepsilon.
      \end{alignat*}
      It follows that
      \begin{alignat*}{3}
        \int_0^1 g(x)^2dx
        &=\lim_{n \to \infty} \sum_{i=1}^n (g^{n}_i)^2\frac{1}{n}\\
        &\le \limsup_{n \to \infty} \sum_{i=1}^n \left(g^{n}_i+Q_n\right)^2\frac{1}{n}\\
        &\le \lim_{n \to \infty} M_n\left(U+\varepsilon,L\right) \\
        &= M\left(U+\varepsilon,L\right). & \text{Lemma \ref{lem:claim3}} 
      \end{alignat*}

      Since $M\left(\cdot,L\right)$ is continuous and $\varepsilon$ was arbitrary we have that $\int_0^1 g(x)^2dx \le M\left(U,L\right)$.
    \end{proof}
    
    \subsubsection{Lower bound for Lipschitz densities in $L^1$.}
    \begin{lem} \label{lem:midapproxL1}
      Let $f:[a,b]\to \rn$ be the function $f(x)= Lx+c$ for some $a,b,c,L\in \mathbb{R}$ with $a<b$. Then  
      \begin{equation}
        \min_{\alpha \in \rn} \int_a^b \left|\alpha - f(x)\right| dx = \frac{L\left(b - a \right)^2 }{4}, \label{eqn:midapproxL1-min}
      \end{equation}
      and
      \begin{equation}
        \arg \min_{\alpha \in \rn} \int_a^b \left|\alpha - f(x)\right| dx = L\frac{a+b}{2} + c = \frac{1}{b-a} \int_a^b f(x)dx.\label{eqn:midapproxL1-arg}
      \end{equation}
    \end{lem}

    \begin{proof}[Proof of Lemma \ref{lem:midapproxL1}]
      When $L = 0$ the result follows from just setting $\alpha = c$. We will assume that $L\neq 0$.

      Suppose that $\alpha> \max(f(a),f(b))$, then we have 
      \begin{align*}
        \int_a^b \left|\alpha - f(x)\right| dx 
        & = \int_a^b \alpha - f(x) dx \\
        & = \int_a^b \alpha - \max\left(f(a),f(b)\right) + \max(f(a),f(b)) - f(x) dx \\
        & = \int_a^b \left|\alpha-\max(f(a),f(b))\right|+\left|\max(f(a),f(b)) - f(x)\right| dx\\
        & > \int_a^b\left|\max(f(a),f(b)) - f(x)\right| dx.
      \end{align*}
      Therefore $\alpha > \max(f(a),f(b))$ cannot be the minimizer since we can simply let $\alpha = \max(f(a),f(b))$ and we have a better minimizer. So we have that $\alpha \le \max(f(a),f(b))$ and a similar argument gives us that $\alpha\geq  \min(f(a),f(b))$. Now we have that $\alpha \in \left[\min_{x\in[a,b]}f(x),\max_{x\in[a,b]}f(x)  \right]$. From this and the continuity of $f$ there exists $r \in [a,b]$ such that $f(r)=\alpha$, specifically $\alpha=Lr+c$. We now assume that $L>0$ as the other case ($L < 0$) is analogous.  

      Continuing with this $r$ we get
      \begin{align*}
        \int_a^b \left|\alpha - f(x)\right| dx 
        & = \int_a^b \left|Lr+c - (Lx+c)\right| dx \\
        & = \int_a^b L\left|r - x\right| dx \\
        & = L\int_a^r (r-x)dx + L\int_{r}^b(x-r)dx\\
        & = L(r-a)^2/2+L(b-r)^2/2 \\
        & = L(r-a)^2/2+L(r-b)^2/2 \\
        & = \frac{L(\delta-A)^2}{2}+\frac{L(\delta+A)^2}{2} &\text{letting } \delta:= r - \frac{a+b}{2} \text{ and } A := \frac{a-b}{2}   \\
        & = L\left(A^2 + \delta^2 \right) \\
        & = \frac{L\left(b - a \right)^2 }{4}+L\delta^2.
      \end{align*}
      Upon noting that the last line is minimized for $\delta=0$ we arrive at \eqref{eqn:midapproxL1-min} and the first equality in \eqref{eqn:midapproxL1-arg}. The second equality in \eqref{eqn:midapproxL1-arg} follows from noting
      \begin{align*}
          \int_a^b f(x) dx 
          = \left[\frac{1}{2}Lx^2 +cx \right]_{x=a}^{x=b} 
          = \frac{1}{2}L \left(b^2 - a^2\right) + c (b-a)
          = \left(L \frac{a+b}{2} + c\right)(b-a). 
      \end{align*}
    \end{proof}

    The following lemma addresses the point in the main text ``[w]e show in the appendix that this decays at a rate of $O\left(b^{-1}\right)$ and that \emph{this rate is tight}.''  We use it later to show a lower bound for the rate of convergence of the standard histogram.
 
    \begin{lem} \label{lem:l1lower}
      Let $b\in \nn$ and $f_L:[0,1]\rightarrow \mathbb{R}$ be a collection of pdfs indexed by $L\in [0,\infty)$ via
      \begin{equation*}
        f_L(x) = 
        \begin{cases} 
          1+\left(x-\frac{1}{2}\right)L& L\le 2 \\
          \left(xL-L+\sqrt{2L}\right)_{+}& L \ge 2.
        \end{cases}
      \end{equation*}
      We have that
      \begin{enumerate}[nosep,label=(\alph*)]
        \item If $L\le 2$ then $\min_{g\in \mathcal{H}_{1,b}} \|f_L-g\|_1=\frac{L}{4b}$. \label{item:smallL}
        \item If $L\geq 2$ and $b$ is a multiple of $\sqrt{L/2}$, we have $\min_{g\in \mathcal{H}_{1,b}} \|f_L-g\|_1=\frac{\sqrt{2L}}{4b}$. \label{item:bigL}
      \end{enumerate}
    \end{lem}

    \begin{proof}[Proof of Lemma \ref{lem:l1lower}]
      We will begin by optimizing over $\spn\left(\sH_{1,b}\right)$ and will then show that the optimum lies in $\sH_{1,b}$. To begin
      \begin{align}
        \min_{g \in \spn\left(\sH_{1,b}\right)} \|f_L-g\|_1
        &= \min_w \int_0^1 \left|f_L(x) - \sum_{i=1}^b w_i b \1 \left( (i-1)/b\le x< i/b \right) \right| dx \label{eqn:l1lower-arg}\\
        &= \min_{\tw} \sum_{i=1}^b \int_{(i-1)/b}^{i/b}\left|f_L(x) -  \tw_i \right| dx.\notag
      \end{align}
      For case \ref{item:bigL} there is a breakpoint at $x$ satisfying $xL-L+\sqrt{2L}=0 \iff x = 1- \sqrt{\frac{2}{L}}$. Since $b$ is a multiple of $\sqrt{L/2}$, there exists an integer $z$ such that 
      \begin{align}
          &b = z\sqrt{L/2} \notag\\
          \iff & z/b = \sqrt{2/L} \notag \\  
          \iff & (b-z)/b = 1-\sqrt{2/L} \label{eqn:b-multiple}
      \end{align}
      and we have for both cases \ref{item:smallL} and \ref{item:bigL} that $f_L$ is linear on the bins $[(i-1)/b,< i/b)$.
      
      Applying Lemma~\ref{lem:midapproxL1} we have the following for \ref{item:smallL},
      \begin{equation*}
          \min_{\tw} \sum_{i=1}^b \int_{(i-1)/b}^{i/b}\left|f_L(x) -  \tw_i \right| dx 
          = \sum_{i=1}^b \frac{L\left({i/b} - (i-1)/b\right)^2}{4} = \frac{L}{4b}.
      \end{equation*}
      For \ref{item:bigL} we split the summation between the breakpoint
      \begin{align}
          &\min_{\tw} \sum_{i=1}^b \int_{(i-1)/b}^{i/b}\left|f_L(x) -  \tw_i \right| dx\notag \\
          &= \min_{\tw} \sum_{i\in [b]:\frac{i}{b} \le 1- \sqrt{\frac{2}{L}}} \int_{(i-1)/b}^{i/b}\left|f_L(x) -  \tw_i \right| dx +
            \sum_{i\in [b]:\frac{i}{b} > 1- \sqrt{\frac{2}{L}}} \int_{(i-1)/b}^{i/b}\left|f_L(x) -  \tw_i \right| dx\notag \\
          &= \sum_{i\in [b]:\frac{i}{b} \le 1- \sqrt{\frac{2}{L}}} 0 + \sum_{i\in [b]:\frac{i}{b} > 1- \sqrt{\frac{2}{L}}} \frac{L}{4b^2}\notag \\
          &=\left|\left\{i\in [b]:\frac{i}{b} > 1- \sqrt{\frac{2}{L}}\right\}\right| \frac{L}{4b^2}. \label{eqn:case-bigL}
      \end{align}
      From \eqref{eqn:b-multiple} and because $1- \sqrt{2/L}\ge 0$ it follows that $b\left( 1- \sqrt{\frac{2}{L}}\right)$ is a nonnegative integer so
      \begin{align*}
          \left|\left\{i\in [b]:\frac{i}{b} > 1- \sqrt{\frac{2}{L}}\right\}\right|
          &= b- \left|\left\{i\in [b]:\frac{i}{b} \le 1- \sqrt{\frac{2}{L}}\right\}\right|\\
          &= b- \left|\left\{i\in [b]:i \le b\left( 1- \sqrt{\frac{2}{L}}\right)\right\}\right|\\
          &= b- \left( b\left( 1- \sqrt{\frac{2}{L}}\right)\right)\\
          &= b\sqrt{\frac{2}{L}}.
      \end{align*}
      Now \eqref{eqn:case-bigL} is equal to 
      $$
        b\sqrt{\frac{2}{L}} \frac{L}{4b^2} = \frac{\sqrt{2L}}{4b}.
      $$
      
      We will show that the argument for the minimum $w$ (with $bw=\tw$) from \eqref{eqn:l1lower-arg} lies in $\Delta_b$ to finish the proof. From the second equality in \eqref{eqn:midapproxL1-arg} we clearly get that $\tw_i \ge 0$ for all $i$. Again using the second equality in \eqref{eqn:midapproxL1-arg} we have
      \begin{align*}
          \sum_{i=1}^b w_i 
           = \frac{1}{b}\sum_{i=1}^b \tw_i
          &= \frac{1}{b}\sum_{i=1}^b \frac{1}{i/b- (i-1)/b}\int_{(i-1)/b}^{i/b} f_L(x) dx \\
          & =\sum_{i=1}^b \int_{(i-1)/b}^{i/b} f_L(x) dx \\
          & = \int_0^1 f_L(x) dx 
           =1.
      \end{align*}
    \end{proof}

    \subsection{Finite Sample Rate: Multi-view}
    This section contains the finite-sample bounds from Section 2.2.2 in the main text. The results here are a bit stronger than those from the main text at the cost of being a bit less concise. In the main text we weakened the results by assuming that Lipschitz constants were greater than or equal to 2. The next section contains corresponding results for Tucker models.
    \label{sec:finitemulti}
    First we will prove the following finite-sample bound.
    \begin{prop}[Proposition \ref{prop:finiteant} in main text]
      \label{prop:finiteantappx}
      Let $d,b,k,n\in \nn$ and $0<\delta\leq 1$. There exists an estimator $V_n\in \sH_{d,b}^k$ such that for all densities $p \in \sD_d$ the following holds with probability at least $1-\delta$ 
      \begin{align}
        \label{eqn-appx:finitebound}
        \|p-V_n\|_1\leq \min_{q\in \sH_{d,b}^k} 3\|p-q\|_1+ 7\sqrt{\frac{2bdk\log(4bdkn)}{n}}+7\sqrt{\frac{\log(\frac{3}{\delta})}{2n}}
      \end{align}
      where $V_n$ is a function of $X_1,\ldots,X_n \simiid p$.
    \end{prop}
    
    \begin{proof}[Proof of Proposition~\ref{prop:finiteantappx}]
      We begin by showing that \eqref{eqn-appx:finitebound}, and therefore the proposition, holds if $n=1$, and $d,b,k\ge 1$. Bounding the second term in the next summation with $b,d,k\ge 1$ and $n=1$ yields the following inequality
      \begin{align*}
        3\min_{q\in\sH_{d,b}^k}\|p-q\|_1+ 7\sqrt{\frac{2bdk\log(4bdkn)}{n}}+7\sqrt{\frac{\log(\frac{3}{\delta})}{2n}}
        &\geq 7\sqrt{\frac{2bdk\log(4bdkn)}{n}}\\
        &\geq 7\sqrt{2\log(4)}\\
        &\geq 7.
      \end{align*}
     The triangle inequality gives us $\|p-V_n\|_1\leq 2\leq 7$. Therefore Proposition \ref{prop:total} holds for $n=1$. We will now proceed assuming that $n\ge2.$

      Let $1\ge \varepsilon>0$ be arbitrary. By Corollary~\ref{cor:lrcover} there exists $p_1,\ldots,p_M \in \sH_{d,b}^k$ such that $M\le\left(\frac{4bd}{\varepsilon} \right)^{bdk} \left( \frac{4k}{\varepsilon}\right)^k$ and for all $q\in \sH^k_{d,b}$ there exists $i\leq M$ with  $\|p_i-q\|_1\leq \varepsilon$.
      Applying Lemma~\ref{lem:densalg} with the same $\varepsilon$ gives a deterministic algorithm $V_n$ that, given at least $\frac{\log(3M^2/\delta)}{2\varepsilon^2}$ samples from a density $p$, outputs an index $j\in [M]$ where, with probability at least $1-\delta$, the following holds 
      \begin{align*} 
        \left\|p_j - p \right\|_1 \notag
        \le& 3 \min_{i \in \left[M\right]} \left\|p_i - p\right\|_1 + 4 \varepsilon \notag \\
        \le& 3\min_{q\in \sH_{d,b}^k}  \min_{i \in \left[M\right]} \left(\left\|p_i - q\right\|_1 + \left\|q - p\right\|_1\right) + 4 \varepsilon \notag\\
        =& \min_{q\in \sH_{d,b}^k} 3 \left( \min_{i \in \left[M\right]}\left\|p_i - q\right\|_1\right) + 3\left\|q - p\right\|_1 + 4 \varepsilon \\
        \le& 7\varepsilon+3\min_{q\in \sH_{d,b}^k}\|p-q\|_1\notag
      \end{align*}
       (we are loosing $\delta/3$ from Lemma~\ref{lem:densalg} to $\delta$ for convenience). 

      Note that
      \begin{equation}
        \frac{\log(3M^2/\delta)}{2\varepsilon^2}
        =\frac{\log(M)}{\varepsilon^2} + \frac{\log\left(3/\delta\right)}{2\varepsilon^2}. \label{eqn:Mexpression}
      \end{equation}
      We will now bound $\log(M)$ which, because $\varepsilon$ is positive and $\log$ is strictly increasing, will give us an upper bound on the previous term. The following follows from the fact that $b,d$ and $k$ are all greater than or equal to 1 and $1\ge\varepsilon>0$,
      \begin{align*}
        M
        &\le\left(\frac{4bd}{\varepsilon} \right)^{bdk} \left( \frac{4k}{\varepsilon}\right)^k\\
        &\le\left(\frac{4bdk}{\varepsilon} \right)^{bdk} \left( \frac{4bdk}{\varepsilon}\right)^{bdk}\\
        &= \left( \frac{4bdk}{\varepsilon}\right)^{2bdk}.
      \end{align*}
      Applying this to \eqref{eqn:Mexpression} we have 
      \begin{equation}\label{eqn:mv-sample-eps-select}
        \frac{\log(3M^2/\delta)}{2\varepsilon^2}
        \leq \frac{2bdk\log( \frac{4bdk}{\varepsilon})}{\varepsilon^2}+\frac{\log(\frac{3}{\delta})}{2\varepsilon^2}.
      \end{equation}
      The rest of the proof will be primarily concerned with choosing $\varepsilon\in \left(0,1\right]$ so that the RHS of \eqref{eqn:mv-sample-eps-select} is less than or equal to $n$ so the hypotheses of Lemma~\ref{lem:densalg} are satisfied. We begin by eliminating some settings where selecting $V_n$ trivial; we will then apply Lemma~\ref{lem:densalg} for the remaining settings.
      
      Observe that if $n< 4bdk\log(4bdkn)$, then 
      $$
      7\sqrt{\frac{2bdk\log(4bdkn)}{n}}>2,$$
      and inequality~\eqref{eqn-appx:finitebound} holds trivially.
      Similarly, if $n < \log\left(\frac{3}{\delta}\right)$, then 
      $$7\sqrt{\frac{\log(\frac{3}{\delta})}{2n}}\geq  \frac{7}{\sqrt{2}}>2,$$
      and again inequality~\eqref{eqn-appx:finitebound} holds trivially. 

      Thus, we can proceed with the setting 
      \begin{align}
        \label{eqn:saved}
        n&\geq 2\max\left(2bdk\log(4bdkn),\frac{\log(\frac{3}{\delta})}{2}\right)\notag\\
        &\geq 2bdk\log(4bdkn)+\frac{\log(\frac{3}{\delta})}{2}.
      \end{align}

      Defining the function $\rho(\varepsilon):=  \frac{2bdk\log( \frac{4bdk}{\varepsilon})}{\varepsilon^2}+\frac{\log(\frac{3}{\delta})}{2\varepsilon^2}$, inequality~\eqref{eqn:saved} implies that 
      \begin{align}
        \label{eqn:mvtright}
        \rho(1)=\frac{2bdk\log( \frac{4bdk}{1})}{1}+\frac{\log(\frac{3}{\delta})}{2}\leq n.
      \end{align}
      Furthermore, 
      \begin{align}
        \label{eqn:mvtleft}\lim_{x\rightarrow 0}\rho(x)=\infty.
      \end{align}

      Together with the mean value theorem, \eqref{eqn:mvtright} and~\eqref{eqn:mvtleft} imply that we can now pick $1\geq \varepsilon> 0$ such that \begin{align}
        \label{balance}
        \rho(\varepsilon)=\frac{2bdk\log( \frac{4bdk}{\varepsilon})}{\varepsilon^2}+\frac{\log(\frac{3}{\delta})}{2\varepsilon^2}=n.
      \end{align} 
      We can now apply the estimator from Lemma~\ref{lem:densalg} to select the estimator $V_n$. As we have shown before the estimator in Lemma~\ref{lem:densalg} outputs a density in $\sH_{d,b}^k$ such that 
      \begin{equation}
        \label{keyy}
        \|V_n-p\|_1\leq 7\varepsilon+3\min_{q\in \sH_{d,b}^k}\|p-q\|_1.
      \end{equation}
      By \eqref{balance}, we have 
      \begin{equation*}
        \varepsilon=\sqrt{\frac{2bdk\log( \frac{4bdk}{\varepsilon})}{n}+\frac{\log(\frac{3}{\delta})}{2n}}\geq \sqrt{\frac{1}{2n}}\geq \frac{1}{n},
      \end{equation*}
      since $0<\delta, \varepsilon \leq 1$ and $n\geq 2$. Using this in \eqref{balance}, we obtain
      \begin{align*}
        \varepsilon
        &=\sqrt{\frac{2bdk\log( \frac{4bdk}{\varepsilon})}{n}+\frac{\log(\frac{3}{\delta})}{2n}}\\
        &\le\sqrt{\frac{2bdk\log( 4bdkn)}{n}+\frac{\log(\frac{3}{\delta})}{2n}}\\
        &\leq \sqrt{\frac{2bdk\log(4bdkn)}{n}}+\sqrt{\frac{\log(\frac{3}{\delta})}{2n}}.
      \end{align*}
      The result follows upon plugging this back into inequality~\eqref{keyy}.
    \end{proof}
    Now we can prove the key result from the paper.
    \begin{prop}[Theorem \ref{thm:class} in main text]
      \label{prop:total}
      Let $L\geq 2$, $0<\delta\le 1$ and $k,n \in \nn$. Then there exists $b$ and an estimator $V_n \in \sH_{d,b}^k$ such that for any density $p \triangleq \sum_{i=1}^k w_i \prod_{j=1}^d p_{i,j}$ where $p_{i,j} \in \lip_L$ and $w$ is in the probability simplex, the following holds with probability at least $1-\delta$,
      \begin{align}\label{eqn:1.1}
        \|V_n-p\|_1\leq \frac{21dk^{1/3}L^{\frac{d+3}{12}}}{n^{\frac{1}{3}}}\sqrt{\log(3Ldkn)}+7\sqrt{\frac{\log(\frac{3}{\delta})}{2n}}
      \end{align}
      where $V_n$ is a function of $X_1,\ldots,X_n \simiid p$.
      
      This also holds with ``$L\leq 2$'' replacing ``$L\ge 2$'' and the following inequality replacing \eqref{eqn:1.1}
      \begin{align}\label{eqn:1.2}
        \|V_n-p\|_1\leq  \sqrt{d}\frac{k^{1/3}}{n^{1/3}}\left[L^{1/3}\exp\left(\frac{L^2(d-1)}{24}\right)+20\sqrt{\log(7dnk)}\right]+7\sqrt{\frac{\log(\frac{3}{\delta})}{2n}}.
      \end{align}
    \end{prop}

    \begin{proof}[Proof of Proposition~\ref{prop:total}]
      We begin with the $L\ge 2$ case and the other case will follow with some minor adjustments. From H\"older's Inequality followed by Proposition \ref{prop:l2projbnd}, if $b^2\geq L^2/12$ and $L\geq 2$, then for any collection $f_1,\ldots,f_d$ in $\lip_L$ we have that
      \begin{equation*}
        \left \| \prod_{j=1}^d f_j - \proj_{\sH_{d,b}^1}\prod_{j=1}^d f_j \right\|_1^2
        \le \left \| \prod_{j=1}^d f_j - \proj_{\sH_{d,b}^1}\prod_{j=1}^d f_j \right\|_2^2 
        \le\frac{dL^{\frac{d+3}{2}}}{12b^2}\left[\frac{\sqrt{8}}{3}\right]^{d-1}
        \le \frac{dL^{\frac{d+3}{2}}}{9b^2}.
      \end{equation*}
      Taking the square root, we have 
      \begin{equation*}
        \left \| \prod_{j=1}^d f_j - \proj_{\sH_{d,b}^1}\prod_{j=1}^d f_j \right\|_1 
        \leq \frac{\sqrt{d}L^{\frac{d+3}{4}}}{3b} .
      \end{equation*}

      Consider some density $p$ like that from the theorem statement  
      $$
      p=\sum_{i=1}^k w_i\prod_{j=1}^d p_{i,j}.
      $$

      Since  $\sum_{i=1}^k w_i\prod_{j=1}^d \proj_{\sH_{1,b}}p_{i,j}$ is an element of $\sH_{d,b}^k$, using Corollary \ref{cor:rankproj} gives us 
      \begin{align*}
        \min_{q\in \sH_{d,b}^k} \left \| p - q \right\|_1
        & \le\left\|\sum_{i=1}^k w_i\prod_{j=1}^d p_{i,j}-\sum_{i=1}^k w_i\prod_{j=1}^d \proj_{\sH_{1,b}}p_{i,j}\right\|_1\notag \\
        & =\left\|\sum_{i=1}^k w_i\prod_{j=1}^d p_{i,j}-\sum_{i=1}^k w_i\proj_{\sH_{d,b}^1}\prod_{j=1}^d p_{i,j}\right\|_1\notag\\
        &\leq \sum_{i=1}^k w_i \left\|\prod_{j=1}^d p_{i,j}- \proj_{\sH_{d,b}^1} \prod_{j=1}^d  p_{i,j}\right\|_1\notag \\
        &\le \sum_{i=1}^k w_i \frac{\sqrt{d}L^{\frac{d+3}{4}}}{3b} =\frac{\sqrt{d}L^{\frac{d+3}{4}}}{3b}.
      \end{align*}

      Invoking the estimator $V_n$ from Proposition~\ref{prop:finiteantappx} and using the previous bound  it follows that, for any choice of $k$ and $b$ such that $b^2\geq L^2/12$, there exists an estimator $V_n\in \sH_{d,b}^k$ where for all densities $p$ from our theorem statement, with probability at least $1-\delta$, the following holds
      \begin{align}
        \label{eqn:mv-estimator-middle}
        \|p-V_n\|_1\leq \frac{\sqrt{d}L^{\frac{d+3}{4}}}{b}+ 7\sqrt{\frac{2bdk\log(4bdkn)}{n}}+7\sqrt{\frac{\log(\frac{3}{\delta})}{2n}}.
      \end{align}

      We now set $b=\left \lceil n^{\frac{1}{3}}L^{\frac{d+3}{6}}k^{-1/3}\right \rceil$.  Note that if $n<k L$ then \eqref{eqn:1.1} is trivially satisfied since $L\ge 2$ and
      \begin{equation*}
         \frac{21dk^{1/3}L^{\frac{d+3}{12}}}{n^{\frac{1}{3}}}\geq \frac{21dk^{1/3}L^{\frac{1+3}{12}}}{n^{\frac{1}{3}}}\geq \frac{21dk^{1/3}L^{\frac{1}{3}}}{n^{\frac{1}{3}}}\geq 21.
      \end{equation*}
      Thus we can proceed with the assumption that $n\geq k L$ for the $L\ge 2$ case of this proof. Under this assumption, noting again that $L\ge 2$, we have the following bound on $b$,
      \begin{equation*}
          b\geq n^{\frac{1}{3}}L^{\frac{d+3}{6}}k^{-1/3}\geq L^{1/3}L^{\frac{1+3}{6}} =  L,
      \end{equation*} 
      which implies $b^2\geq \frac{L^2}{12}$. So for the remainder of the $L\ge 2$ case of this proof we will proceed with the assumption $b^2\geq \frac{L^2}{12}$ and use the estimator from \eqref{eqn:mv-estimator-middle}.
      
      Since $n\geq kL \ge k$ and $L\ge 2$, we also have $1\le n^{\frac{1}{3}}L^{\frac{d+3}{6}}k^{-1/3}\leq b\leq 2n^{\frac{1}{3}}L^{\frac{d+3}{6}}k^{-1/3}$. Thus
      \begin{align*}
        7\sqrt{\frac{2bdk\log\left(4bdkn\right)}{n}}
        &\leq 7\sqrt{\frac{4dk^{2/3}n^{\frac{1}{3}}L^{\frac{d+3}{6}}\log\left(8L^{\frac{d+3}{6}}dk^{2/3}n^{4/3}\right)}{n}} \\       
        &= 14\frac{\sqrt{d}k^{1/3}L^{\frac{d+3}{12}}}{n^{1/3}}\sqrt{\log\left(8L^{\frac{d+3}{6}}dk^{2/3}n^{4/3}\right)} \\
        &\le 14\frac{\sqrt{d}k^{1/3}L^{\frac{d+3}{12}}}{n^{1/3}}\sqrt{\log\left(3^{2d}L^{2d}d^{2d}k^{2d}n^{2d}\right)} \\
        &= 14\frac{\sqrt{d}k^{1/3}L^{\frac{d+3}{12}}}{n^{1/3}} \sqrt{2d\log\left(3Ldkn\right)} \\
        & \leq 20\frac{dk^{1/3}L^{\frac{d+3}{12}}}{n^{1/3}}\sqrt{\log\left(3Ldkn\right)}.
      \end{align*}
      Using this with \eqref{eqn:mv-estimator-middle} and applying $n^{\frac{1}{3}}L^{\frac{d+3}{6}}k^{-1/3}\leq b$ to the first summand in \eqref{eqn:mv-estimator-middle} we have
      \begin{align*}
        \|p-V_n\|_1
        &\leq \frac{\sqrt{d}k^{1/3}L^{\frac{d+3}{12}}}{n^{\frac{1}{3}}}+\frac{20dk^{1/3}L^{\frac{d+3}{12}}}{n^{\frac{1}{3}}}\sqrt{\log(3Ldkn)}+7\sqrt{\frac{\log(\frac{3}{\delta})}{2n}}\\
        &\leq \frac{21dk^{1/3}L^{\frac{d+3}{12}}}{n^{\frac{1}{3}}}\sqrt{\log(3Ldkn)}+7\sqrt{\frac{\log(\frac{3}{\delta})}{2n}},
      \end{align*}
      as expected. This finishes the $L\ge 2$ case.

      For the $L\leq 2$ case, we obtain (using Proposition~\ref{prop:l2projbnd} and taking the square root along with some simple manipulations) instead,
      \begin{equation*}
        \min_{q \in \sH_{d,b}^k} \left \| p - q \right\|_1 \leq \sqrt{d}\frac{L}{3b}\exp\left(\frac{L^2(d-1)}{24}\right),
      \end{equation*}
      and instead of \eqref{eqn:mv-estimator-middle}, 
      \begin{equation}
        \|p-V_n\|_1\leq \sqrt{d}\frac{L}{b}\exp\left(\frac{L^2(d-1)}{24}\right)+ 7\sqrt{\frac{2bdk\log(4bdkn)}{n}}+7\sqrt{\frac{\log(\frac{3}{\delta})}{2n}}.\label{eqn:mv-smallL}
      \end{equation}
         Since $L\leq 2$, we clearly have $\frac{L^2}{12} < b^2$  so we can use Proposition~\ref{prop:l2projbnd} without issue. Set $b=\left \lceil n^{1/3}k^{-1/3}L^{2/3}\right \rceil$ (we can assume $L>0$ and thus $b\ge 1$ since the $L=0$ case can be solved by simply setting $V_n$ to output the uniform distribution). If $k> n$ then \eqref{eqn:1.2} is trivially satisfied so we can proceed with the assumption that $k\le n$. 
      
        Since $n\geq k$ and $n^{1/3}k^{-1/3} \ge 1$ it follows that $b\geq n^{1/3}k^{-1/3}L^{2/3}$ and $b\leq \left\lceil n^{1/3}k^{-1/3}2^{2/3}\right\rceil \leq \left\lceil n^{1/3}k^{-1/3}2\right\rceil \leq 3n^{1/3}k^{-1/3}$. Letting $C = 7\sqrt{\frac{\log(\frac{3}{\delta})}{2n}}$ we can bound \eqref{eqn:mv-smallL} as follows
      \begin{align*}
        \|p-V_n\|_1
        &\leq \sqrt{d}\frac{L}{b}\exp\left(\frac{L^2\left(d-1\right)}{24}\right) + 7\sqrt{\frac{2bdk\log\left(4bdkn\right)}{n}} + C\\
        &\leq \sqrt{d}\frac{L}{b}\exp\left(\frac{L^2\left(d-1\right)}{24}\right) + 7 \sqrt{\frac{2\left(3n^{1/3}k^{-1/3}\right)k}{n}}  \sqrt{d\log\left(4\left(3n^{1/3}k^{-1/3}\right)dkn\right)} + C\\
        &\leq \sqrt{d}\frac{L^{1/3}k^{1/3}}{n^{1/3}}\exp\left(\frac{L^2\left(d-1\right)}{24}\right)+\frac{7k^{1/3}}{n^{1/3}}\sqrt{6d\log\left(12dn^{4/3}k^{2/3}\right)} + C \\
        &\leq \sqrt{d}\frac{L^{1/3}k^{1/3}}{n^{1/3}}\exp\left(\frac{L^2\left(d-1\right)}{24}\right)+\frac{7k^{1/3}}{n^{1/3}}\sqrt{6d\log\left(7^{4/3}d^{4/3}n^{4/3}k^{4/3}\right)} + C \\
        &\leq \sqrt{d}\frac{L^{1/3}k^{1/3}}{n^{1/3}}\exp\left(\frac{L^2\left(d-1\right)}{24}\right)+\frac{7k^{1/3}}{n^{1/3}}\sqrt{8d\log\left(7dnk\right)}+ C \\
        &\leq \sqrt{d}\frac{L^{1/3}k^{1/3}}{n^{1/3}}\exp\left(\frac{L^2\left(d-1\right)}{24}\right)+\frac{20k^{1/3}\sqrt{d}}{n^{1/3}}\sqrt{\log\left(7dnk\right)} +C\\
        &=\sqrt{d}\frac{k^{1/3}}{n^{1/3}}\left[L^{1/3}\exp\left(\frac{L^2(d-1)}{24}\right)+20\sqrt{\log(7dnk)}\right]+7\sqrt{\frac{\log(\frac{3}{\delta})}{2n}},
      \end{align*}
      which finishes the proof for the $L\le 2$ case.
    \end{proof}

    \subsection{Finite Sample Rate: Tucker} \label{rateantoinetucker}
    Here we present results for the Tucker decomposition that are analogous to those in the last section. The results here were omitted from the main text ``[f]or brevity, and because the results are virtually direct analogues of their multi-view histogram counterparts...'' We begin again with a proof of a finite-sample bound, which will then be used to prove a distribution-free bound.
    \begin{prop}
      \label{prop:finiteantucker}
      Let $d,b,k,n\in \nn$ and $0<\delta\leq 1$. There exists an estimator $V_n\in\tsH_{d,b}^k$ such that for all densities $p\in \sD_d$, the following holds with probability at least $1-\delta$  
      \begin{align}
        \label{eqn:tuckerfiniteb}
        \|p-V_n\|_1\leq 3 \min_{q\in \tsH_{d,b}^k} \|p-q\|_1+ 7\sqrt{\frac{2bdk\log( 4bdn)}{n}}+7\sqrt{\frac{2k^d\log( 4k^dn)}{n}}+7\sqrt{\frac{\log(\frac{3}{\delta})}{2n}}
      \end{align}
      where $V_n$ is a function of $X_1,\ldots,X_n \simiid p$.
    \end{prop}

    \begin{proof}[Proof of Proposition~\ref{prop:finiteantucker}]
      This proof is very similar to the proof of Proposition~\ref{prop:finiteantappx} and we will provide fewer intermediate steps when they are virtually identical to those that proof.

      Similarly to the proof of Proposition~\ref{prop:finiteantappx}, we can assume $n\geq 2$. Let $1\geq \varepsilon>0$ be arbitrary. 
       By Corollary~\ref{cor:tuckcover} there exists $p_1,\ldots,p_M \in \tsH_{d,b}^k$ such that $M \le \left(\frac{4bd}{\varepsilon} \right)^{bdk} \left( \frac{4k^d}{\varepsilon}\right)^{k^d}$ and for all $q\in \tsH^k_{d,b}$ there exists $i\leq M$ with  $\|p_i-q\|_1\leq \varepsilon$.

      Now, by applying Lemma~\ref{lem:densalg} with the same $\varepsilon$, there exists a deterministic algorithm which, for all densities $p$, can output an index $j\in [M]$ such that 
      \begin{align*} 
        \left\|p_j - p \right\|_1  \leq 7\varepsilon+3\min_{q\in\tsH_{d,b}^k}\|p-q\|_1,
      \end{align*}
      with probability at least $1-\delta$ given at least $N$ samples from the distribution, where 
      \begin{equation*}
        N\ge\frac{\log(3M^2/\delta)}{2\varepsilon^2}
        =\frac{\log(M)}{\varepsilon^2} + \frac{\log\left(3/\delta\right)}{2\varepsilon^2}.
      \end{equation*}
       Now we can bound
        \begin{equation*}
            \frac{\log(3M^2/\delta)}{2\varepsilon^2}
        \leq  \frac{2bdk\log( \frac{4bd}{\varepsilon})}{\varepsilon^2}+\frac{2k^d\log( \frac{4k^d}{\varepsilon})}{\varepsilon^2}+\frac{\log(\frac{3}{\delta})}{2\varepsilon^2}.
        \end{equation*}

      Note that:
      \begin{itemize}[nosep] 
        \item If $n\leq 6bdk\log( 4bdn)$, then $7\sqrt{\frac{2bdk\log( 4bdn)}{n}}\geq 7/\sqrt{3}\geq 2$, making \eqref{eqn:tuckerfiniteb} trivial.
        \item If $n\leq 6k^d\log( 4k^dn)$ then $7\sqrt{\frac{2k^d\log( 4k^dn)}{n}}\geq 7/\sqrt{3}\geq 2$, also making \eqref{eqn:tuckerfiniteb} trivial.
        \item Similarly, if $n\leq 3\log(\frac{3}{\delta})$, then  $7\sqrt{\frac{\log(\frac{3}{\delta})}{2n}}\geq 2$, making \eqref{eqn:tuckerfiniteb} trivial yet again.
      \end{itemize} 
      Thus we can assume that 
      \begin{align}
        n&\geq 3 \max\left(2bdk\log( 4bdn), 2k^d\log( 4k^dn),\log\left(\frac{3}{\delta}\right)\right) \notag \\
        &\geq 2bdk\log( 4bdn)+2k^d\log( 4k^dn)+\log\left(\frac{3}{\delta}\right). \label{eqn:saved-tucker}
      \end{align}
        We now define  $\rho(\varepsilon)$ as 
      \begin{equation}
         \rho(\varepsilon) := \frac{2bdk\log( \frac{4bd}{\varepsilon})}{\varepsilon^2}+\frac{2k^d\log( \frac{4k^d}{\varepsilon})}{\varepsilon^2}+\frac{\log(\frac{3}{\delta})}{2\varepsilon^2}.
      \end{equation} 
      So \eqref{eqn:saved-tucker} gives us $\rho(1)\leq n$ and, as before, $\lim_{x\rightarrow 0}\rho(x)=\infty$. Together with the mean value theorem it follows we can now pick $1\geq \varepsilon> 0$ such that 
      \begin{equation}
      \rho(\varepsilon) = \frac{2bdk\log( \frac{4bd}{\varepsilon})}{\varepsilon^2}+\frac{2k^d\log( \frac{4k^d}{\varepsilon})}{\varepsilon^2}+\frac{\log(\frac{3}{\delta})}{2\varepsilon^2} =n.
          \label{balance2}
      \end{equation}
      We can now apply the estimator from Lemma~\ref{lem:densalg} to select the estimator $V_n$. As we have shown before the estimator in Lemma~\ref{lem:densalg} outputs a density in $\sH_{d,b}^k$ such that
      \begin{align}
        \label{keyy2}
        \|V_n-p\|_1\leq 7\varepsilon+3\min_{q \in \tsH_{d,b}^k}\|p-q\|_1.
      \end{align}
        Now, note that by \eqref{balance2}, we have 
        \begin{equation*}
          \varepsilon=\sqrt{\frac{2bdk\log( \frac{4bd}{\varepsilon})}{n}+\frac{2k^d\log( \frac{4k^d}{\varepsilon})}{n}+\frac{\log(\frac{3}{\delta})}{2n}}\geq \sqrt{\frac{1}{2n}}\geq \frac{1}{n},
        \end{equation*}
        since $0<\delta,\varepsilon\leq 1$ and $n\ge 2$. Plugging this back into \eqref{balance2}, we obtain
        \begin{align*}
          \varepsilon
          &\leq \sqrt{\frac{2bdk\log( 4bdn)}{n}+\frac{2dk^d\log( 4k^dn)}{n}+\frac{\log(\frac{3}{\delta})}{2n}}\\
          &\le \sqrt{\frac{2bdk\log( 4bdn)}{n}}+\sqrt{\frac{2k^d\log( 4k^dn)}{n}}+\sqrt{\frac{\log(\frac{3}{\delta})}{2n}}.
        \end{align*}
        The result follows upon plugging this back into inequality~\eqref{keyy2}.
    \end{proof}
    Now we can prove our distribution-free bound.
    \begin{prop}
      \label{total2}
      Let $L\geq 2$, $0<\delta\le 1$ and $k,n \in \nn$. Then there exists $b \in \nn$ and an estimator $V_n \in \tsH_{d,b}^k$ such that for any density $p \triangleq \sum_{S\in [k]^d} W_S \prod_{i=1}^d p_{i,S_i}$ where $p_{i,j}\in \lip_L$ and $W$ is a probability tensor, the following holds with probability at least $1-\delta$,
      \begin{align}
      \label{eqn:total2maineq}
        \|p-V_n\|_1\leq \frac{21dk^{1/3}L^{\frac{d+3}{12}}}{n^{\frac{1}{3}}}\sqrt{\log(3Ldkn)}+7\sqrt{\frac{2k^d\log( 4k^dn)}{n}}+7\sqrt{\frac{\log(\frac{3}{\delta})}{2n}}
      \end{align}
      where $V_n$ is a function of $X_1,\ldots,X_n \simiid p$.

      This also holds with ``$L\leq 2$'' replacing ``$L\ge 2$'' and the following inequality replacing \eqref{eqn:total2maineq}
      \begin{equation*}
        \|p-V_n\|_1\leq \sqrt{d}\frac{k^{1/3}}{n^{1/3}}\left[L^{1/3}\exp\left(\frac{L^2(d-1)}{24}\right)+20\sqrt{\log(7dnk)}\right]+7\sqrt{\frac{2k^d\log( 4k^dn)}{n}}+7\sqrt{\frac{\log(\frac{3}{\delta})}{2n}}.
      \end{equation*}
    \end{prop}

    \begin{proof}[Proof of Proposition~\ref{total2}]
       This proof is very similar to the proof of Proposition \ref{prop:total} and we will provide fewer intermediate steps when they are virtually identical to those that proof. We begin with the $L\ge 2$ case. 
       Consider some density $p$ like that from the theorem statement  
      $$
      p=\sum_{S\in[k]^d} W_S\prod_{i=1}^d p_{i,S_i}.
      $$
       Since  $\sum_{S\in[k]^d} W_S \prod_{i=1}^d \proj_{\sH_{1,b}}p_{i,S_i}$ is an element of $\tsH_{d,b}^k$, using Corollary \ref{cor:rankproj}, H\"older's Inequality, and Proposition~\ref{prop:l2projbnd}  yields the following when  $b^2 \ge L^2/12$
      \begin{align}
        \min_{q\in \tsH_{d,b}^k} \left \| p - q \right\|_1
        & \le\left\|\sum_{S\in[k]^d} W_S\prod_{i=1}^d p_{i,S_i}-\sum_{S\in[k]^d} W_S\prod_{i=1}^d \proj_{\sH_{1,b}}p_{i,S_i}\right\|_1\notag \\
        & =\left\|\sum_{S\in[k]^d} W_S\prod_{i=1}^d p_{i,S_i}-\sum_{S\in[k]^d} W_S\proj_{\sH_{d,b}^1}\prod_{i=1}^d p_{i,S_i}\right\|_1\notag\\
        &\leq \sum_{S\in[k]^d} W_S \left\|\prod_{i=1}^d p_{i,S_i}- \proj_{\sH_{d,b}^1} \prod_{i=1}^d  p_{i,S_i}\right\|_1\notag \\
        &\le \sum_{S\in[k]^d} W_S \frac{\sqrt{d}L^{\frac{d+3}{4}}}{3b} =\frac{\sqrt{d}L^{\frac{d+3}{4}}}{3b}.\label{eqn:notreallynew}
      \end{align}
      Combining this with the estimator from Proposition~\ref{prop:finiteantucker} we get that, for any $b$ such that $b^2\geq L^2/12$, we have
      \begin{equation*}
        \|p-V_n\|_1\leq \frac{\sqrt{d}L^{\frac{d+3}{4}}}{b}  +  7\sqrt{\frac{2bdk\log( 4bdn)}{n}}+7\sqrt{\frac{2k^d\log( 4k^dn)}{n}}+7\sqrt{\frac{\log(\frac{3}{\delta})}{2n}}.
      \end{equation*}
      If $n< k L$, the RHS of \eqref{eqn:total2maineq} is greater than 2, which means that \eqref{eqn:total2maineq} holds trivially. Thus we assume $n\geq k L$. 
      Since $b$ doesn't appear in the third summand in the previous inequality, and the first, second, and fourth summand are exactly the same as those from \eqref{eqn:mv-estimator-middle} in the proof Proposition~\ref{prop:total}, we can set  $b=\left \lceil n^{\frac{1}{3}}L^{\frac{d+3}{6}}k^{-1/3}\right \rceil$ and proceed exactly as we did in that proof (again $b^2 \ge L^2/12$ so Proposition~\ref{prop:l2projbnd} holds). We then obtain
      \begin{equation*}
        \|p-V_n\|_1\leq \frac{21dk^{1/3}L^{\frac{d+3}{12}}}{n^{\frac{1}{3}}}\sqrt{\log(3Ldkn)}+7\sqrt{\frac{2k^d\log( 4k^dn)}{n}}+7\sqrt{\frac{\log(\frac{3}{\delta})}{2n}},
      \end{equation*}
      as expected. 

      For the $L\leq 2$ case, \eqref{eqn:notreallynew} becomes
      \begin{equation*}
        \min_{q\in \tsH_{d,b}^k} \left \| p - q \right\|_1 \leq \sqrt{d}\frac{L}{3b}\exp\left(\frac{L^2(d-1)}{24}\right),
      \end{equation*}
      from which we get
      \begin{equation*}
        \|p-V_n\|_1\leq \sqrt{d}\frac{L}{b}\exp\left(\frac{L^2(d-1)}{24}\right)+  7\sqrt{\frac{2bdk\log( 4bdn)}{n}}+7\sqrt{\frac{2k^d\log( 4k^dn)}{n}}+7\sqrt{\frac{\log(\frac{3}{\delta})}{2n}}.
      \end{equation*}
      Since $b$ doesn't appear in the third summand in the previous inequality, and the rest of the inequality is exactly the same as in the proof of Proposition~\ref{prop:total}, we can again let $b=\left \lceil n^{1/3}k^{-1/3}L^{2/3}\right \rceil$, assume $n\ge k$ and, proceed identically as we did in the proof of Proposition~\ref{prop:total}. This yields
      \begin{equation*}
        \|p-V_n\|_1\leq \sqrt{d}\frac{k^{1/3}}{n^{1/3}}\left[L^{1/3}\exp\left(\frac{L^2(d-1)}{24}\right)+20\sqrt{\log(7dnk)}\right]+7\sqrt{\frac{2k^d\log( 4k^dn)}{n}}+7\sqrt{\frac{\log(\frac{3}{\delta})}{2n}}, 
      \end{equation*}
      as expected.
    \end{proof}

    \newpage
    \subsection{Lower Bound: Standard Histogram} 
    \begin{proof}[Proof of Proposition \ref{prop:multilower}]
      Let $p\in \sD_d$ with $p = \prod_{i=1}^d p_i$ where $p_1$ is the density from Lemma \ref{lem:l1lower} for $L=2$, i.e. $p_1(x) = 2x$, and $p_i(x)\equiv 1$ for $i>1$. Let $\left(Y_1,\ldots,Y_d\right) \sim V_n$ and $\left(X_1,\ldots,X_d\right) \sim p$. Because total variation distance is never increased through mappings of the random variables (see Theorem 5.2 in \cite{devroye01}) we have that $\left\|V_n - p\right\|_1 \ge \left\|f -p_1 \right\|_1$ where $f$ is the probability density associated with $Y_1$ . We will now show that $f$ is an element of $\sH_{1,b}$. Let $S\subset [0,1]$ be an arbitrary (Borel) measurable set and note that $V_n$ has the form $\sum_{A \in \left[b\right]^d} \hat{w}_A h_{d,b,A}$. Then we have that 
      \begin{align*}
        P\left(Y \in S\right)
        & = \int_S f d\lambda \\
        & = \int_{S\times [0,1] \times \cdots \times [0,1]} V_n d \lambda \\
        & = \int_{S\times [0,1] \times \cdots \times [0,1]} \sum_{A \in \left[b\right]^d} \hat{w}_A h_{d,b,A} d \lambda \\
        & = \sum_{A \in \left[b\right]^d}\hat{w}_A \int_{S\times [0,1] \times \cdots \times [0,1]}  h_{d,b,A} d \lambda \\
        & = \sum_{A \in \left[b\right]^d}\hat{w}_A \int_{S\times [0,1] \times \cdots \times [0,1]}\prod_{i=1}^d  h_{1,b,A_i} d \lambda \\
        & = \sum_{A \in \left[b\right]^d}\hat{w}_A \left(\int_S  h_{1,b,A_1}d \lambda\right) \left(\int_{[0,1]}  h_{1,b,A_2}d \lambda \right)\cdots \left(\int_{[0,1]}  h_{1,b,A_d}  d \lambda \right) \\
        & = \sum_{A \in \left[b\right]^d}\hat{w}_A \int_S h_{1,b,A_1} d \lambda\\
        & = \int_S \sum_{A \in \left[b\right]^d}\hat{w}_A h_{1,b,A_1} d \lambda.
      \end{align*}
      note that $\sum_{A \in \left[b\right]^d}\hat{w}_A h_{1,b,A_1}$ is a histogram and thus the density associated with $f$ is a histogram. Using this fact with the earlier mentioned inequality we have that
      \begin{align*}
        \left\|V_n - p \right\|_1 
        &\ge\left\|f - p_1\right\|_1\\
        &\ge \min_{h \in \sH_{1,b}}\left\|h - p_1\right\|_1.
      \end{align*}
      From Lemma \ref{lem:l1lower} it then follows that
      \begin{align*}
        \left\|V_n - p \right\|_1  \ge \frac{1}{2b}.
      \end{align*}
      Let $D>0$ be arbitrary. From our assumption that $n/b^d\to \infty$ it follows that, for sufficiently large $n$, that $n/b^d \ge (2D)^d$ and furthermore
      \begin{align*}
        n/b^d \ge (2D)^d 
        \iff \sqrt[d]{n}/b \ge 2D
        \iff 1/(2b) \ge D/\sqrt[d]{n} \Rightarrow \left\|V_n - p \right\|_1 \ge D/\sqrt[d]{n}
      \end{align*}
      so $\left\|V_n - p\right\|_1 \in \omega(1/\sqrt[d]{n})$ by definition.
    \end{proof}

    \newpage
    \section{Experimental Setting} \label{appx:exp}
    Consider the problem of finding some density estimator $\hat{p}$ with minimal $L_2$ distance to an unknown density $p$ ($p$ is the various projections of MNIST and Diabetes from the main text). This is equivalent to minimizing the squared $L^2$ loss:
    \begin{align}
      &\int_{\left[0,1\right]^d} \left(p(x) - \hat{p}\left(x\right) \right)^2dx \notag\\
      &= \int_{\left[0,1\right]^d}\hat{p}\left(x\right)^2 dx - 2\int_{\left[0,1\right]^d} p(y)\hat{p}(y) dy
      + \int_{\left[0,1\right]^d} p(z)^2 dz.\label{eqn:l2loss}
    \end{align}
    Because the right term in (\ref{eqn:l2loss}) does not depend on $\hat{p}$ it can be ignored when finding optimal $\hat{p}$. The left term in (\ref{eqn:l2loss}) is known. The middle term in (\ref{eqn:l2loss}) can be estimated with the following approximation
    \begin{equation*}
      \int_{\left[0,1\right]^d} p(x)\hat{p}(x) dx = \E_{X\sim p}\left[ \hat{p}(X) \right] \approx \frac{1}{n}\sum_{i=1}^n \hat{p}\left(X_i\right)
    \end{equation*}
    \sloppy where $\sX = X_1,\ldots, X_n \simiid p$. We can use this to find a good estimate $\hat{H}\in \sR^k_{d,b}$ for $p$ which represents $\sH^k_{d,b}$ or $\tsH^k_{d,b}$:

    \begin{align}
      &\arg \min_{\hat{H} \in \sR_{d,b}^k} \int_{\left[0,1\right]^d} \left(\hat{H}\left(x\right) - \hat{p}\left(x\right) \right)^2dx
      = \arg \min_{\hat{H} \in \sR_{d,b}^k} \left<\hat{H},\hat{H}\right> - 2\int_{\left[0,1\right]^d} \hat{H}(x)p(x) dx \notag\\
      &\approx \arg \min_{\hat{H} \in \sR_{d,b}^k} \left<\hat{H},\hat{H}\right> - 2\frac{1}{n}\sum_{i=1}^n \hat{H}(X_i). \label{eqn:l2obj}
    \end{align}

    \sloppy Recall that the standard histogram estimator is $H = H_{d,b}\left(\sX\right) = \frac{1}{n}\sum_{i=1}^n\sum_{A \in \left[b\right]^d}  h_{d,b,A} \1\left(X_i \in \Lambda_{d,b,A}\right)$ and let $\hat{H} = \sum_{A\in \left[b\right]^d}\hat{w}_A h_{d,b,A}$.
    We have the following
    \begin{align*}
      \left<\hat{H},H\right>
      &=  \left<\sum_{A \in \left[b\right]^d} \hat{w}_A h_{d,b,A},\frac{1}{n}\sum_{i=1}^n \sum_{B\in \left[b\right]^d} h_{d,b,B} \1\left(X_i \in \Lambda_{d,b,B}\right) \right>\\
      & =  \frac{1}{n}\sum_{i=1}^n \sum_{A\in \left[b\right]^d}  \hat{w}_A \1\left(X_i \in \Lambda_B\right)b^d =  \frac{1}{n}\sum_{i=1}^n \hat{H}(X_i).
    \end{align*}
    As a consequence (\ref{eqn:l2obj}) is equal to
    \begin{align*}
      \arg \min_{\hat{H} \in \sR_{d,b}^k}\left<\hat{H},\hat{H}\right> -2\left<H,\hat{H}\right>
      &=\arg \min_{\hat{H} \in \sR_{d,b}^k}\left<\hat{H},\hat{H}\right> -2\left<H,\hat{H}\right> + \left<H,H \right>\\
      &= \arg \min_{\hat{H} \in \sR_{d,b}^k}\left\|H-\hat{H}\right\|_2^2. \label{eqn:fitloss}
    \end{align*}
    Using the $U_{d,b}$ operator we can reformulate this into a tensor factorization problem
    \begin{equation*}
      \min_{\hat{T} \in \sQ_{d,b}^k}\left\|H-U_{d,b}(\hat{T})\right\|_2^2
      = \min_{\hat{T} \in \sQ_{d,b}^k}b^d\left\|U_{d,b}^{-1}\left(H\right)-\hat{T}\right\|_2^2
    \end{equation*}
    where $\sQ_{d,b}^k$ could be either $\sT_{d,b}^k$ or $\tsT_{d,b}^k$. Because of this equivalence, to find estimates in $\sH_{d,b}^k$ or $\tsH_{d,b}^k$ we can simply use nonnegative tensor decomposition algorithms, which minimize $\ell^2$ loss, to find NNTF histograms that approximate $H$.

    \newpage
    \section{Nonexistence of Infinite Tensor Decomposition}\label{appx:nonexist}
    Let $p:\rn^2 \to \rn$ be a probability density function. Consider the possibility of decomposing $p$ as follows
    \begin{equation}\label{eqn:decomp}
      p(x,y) = \sum_{i=1}^\infty w_i f_i(x) g_i(y)
    \end{equation}
    where, for all $i$, $w_i \ge0$ and $f_i$ and $g_i$ are probability densities. We are going to show that this is not always possible, which we will do by contradiction. Let $\lambda$ be the Lebesgue measure (dimensionality will be left implicit).  We are going to use the following proposition which we will prove later.

    \begin{prop}\label{prop:weirdset}
      There exists a set $E \subset [0,1]\times [0,1]$ such that $\lambda(E)>0$ and for all non-null measurable sets $A,B \subset [0,1]$ we have that $\lambda(E \cap A\times B) <\lambda( A\times B)$.
    \end{prop}
    Let $E$ be a set satisfying the property in Proposition \ref{prop:weirdset}. Let $\1_S$ be the indicator function a set $S$. We will let $p = \1_E$ and assume that $p$ has a decomposition as in (\ref{eqn:decomp}).

    We will assume that $w_1 >0$. Clearly we have that $p - w_1 f_1 g_1$ is an almost everywhere (a.e.) nonnegative function (all products of functions here are outer products). Let $\varepsilon >0$ such that $A \triangleq f_1^{-1}([\varepsilon,\infty))$ and $B\triangleq g_1^{-1}([\varepsilon,\infty))$ have positive measure. Such an $\varepsilon$ must exist otherwise $f_1$ and $g_1$ are $0$ a.e.. Now we have that $\varepsilon^2 \1_A  \1_B \le f_1  g_1$. And thus $p - w_1\varepsilon^2 \1_A  \1_B \ge 0$ a.e or equivalently $\lambda(E)^{-1}\1_E - w_1\varepsilon^2 \1_{A\times B} \ge 0$ a.e.. From our definition of $E$ we know that $\lambda(A\times B \setminus E) = \lambda(A\times B) - \lambda(E \cap A\times B) >0$ so $\lambda(E)^{-1}\1_E - w_1\varepsilon^2 \1_{A\times B}$ is negative on a set of positive measure, a contradiction.

    We will now address the existence of the set $E$. The most direct statement of the existence of such an $E$ can be found in \cite{kendall02}, the following is the exact statement from the text.
    \begin{thm}[\cite{kendall02} Theorem 2.1]
      There exist Borel measurable subsets $E\subset \left[0,1 \right]^2$ of positive measure which are rectangle free, so that if $A\times B \subseteq E$ then $\area\left(A\times B\right) = 0$.
    \end{thm}
    That paper builds the set $E$ via a random construction and contains an image which showing an example that approximates a randomly sampled $E$. Their construction seems to imply that the condition ``$A\times B \subseteq E$'' was intended to be interpreted measure theoretically, i.e. ``$\area(A\times B \setminus E) =0$''; it is not particularly difficult to construct a measurable subset of $[0,1]\times [0,1]$ which contains all but a null set of $[0,1]\times [0,1]$ and is ``rectangle free'' as described in the theorem statement (see \cite{erdos55} and references therein). If the measure theoretic strengthening is true it would imply the existence of the set $E$ from Proposition \ref{prop:weirdset} above. Since we are not \emph{totally} certain that this strengthening is possible we include a proof of the existence of $E$ above.

    For a topological space $(\Omega,\tau)$ equipped with a Borel measure $\mu$, a set $S\subseteq \Omega$ is called \emph{essentially dense} if, for any nonempty open set $I$, $\mu(I\cap S)>0$. For any measurable set in $\rn^d$ we will equip it with the standard subspace topology and measure induced by the standard Lebesgue measure. There exists a measurable set $D \subset \rn$ such that $D$ and $D^C$ are essentially dense (see \cite{fremlin00} 134J(a)). The following is a simplification of Theorem 1 in \cite{erdos55} that we will use to construct $E$.
    \begin{thm}[\cite{erdos55} Theorem 1]\label{thm:rectangle}
      For a measurable set of the form $E = \left\{(x,y): x-y \in D \right\}$ the following two conditions are equivalent
      \begin{enumerate}[nosep]
        \item $D$ is essentially dense on $\rn$.
        \item $\lambda\left(E\cap A\times B \right) > 0$ for all $A,B$ such that $\lambda(A)\lambda(B)>0$.
      \end{enumerate}
    \end{thm}
    From this we have that $E \triangleq \left\{(x,y): x-y \in D \right\}\cap[0,1]^2$ and $E^C = \left\{(x,y): x-y \in D^C \right\}\cap [0,1]^2$ (we let $E$ live in $[0,1]^2$) have that property that for non null sets $A,B \subset [0,1]$, $\lambda(E \cap A\times B) >0$, $\lambda(E^C \cap A\times B) >0$. Note that $\lambda(E \cap A\times B) + \lambda(E^C \cap A\times B) = \lambda(A\times B)$ and thus $\lambda(E\cap A\times B) < \lambda(A\times B)$ so we have constructed $E$.

    We mention that the $E$ we have constructed contains a non-null rectangle when rotated by $45$ degrees. Thusly rotating $p$ to give $\tp$ allows us to find $f,g$ and $w>0$ such that $\tp - wfg$ is a.e. nonnegative. In \cite{erdos55} the authors discuss the existence of sets $E$ where, for all non null $A,B$, we have that $\lambda(f(E) \cap A\times B)>0$, for all $f$ in certain classes of transforms. These results hint towards research directions of finding transforms so that our data is better approximated by nice NNTF model. 

    \bibliography{bibfile}
    \bibliographystyle{plain}

\end{document}